\documentclass[12pt]{amsart}
\usepackage{amssymb,amsmath,amsthm,amsfonts,times,mathrsfs,multirow}
\newtheorem{theorem}{Theorem}[section]
\newtheorem{lemma}[theorem]{Lemma}
\newtheorem{cor}[theorem]{Corollary}

\newtheorem{drconj}[theorem]{Dyson's Rank Conjecture}
\newtheorem{prop}[theorem]{Proposition}

\theoremstyle{definition}
\newtheorem{definition}[theorem]{Definition}

\theoremstyle{remark}
\newtheorem*{remark}{Remark}

\numberwithin{equation}{section}

\newcommand{\abs}[1]{\lvert#1\rvert}

\newcommand{\SL}{\mbox{SL}}
\newcommand{\GL}{\mbox{GL}}
\newcommand{\SLZ}{\mbox{SL}_2(\mathbb{Z})}

\newcommand{\sgn}{\textnormal{sgn}}
\newcommand{\Parans}[1]{\left(#1\right)}

\newcommand\leg[2]{\genfrac{(}{)}{}{}{#1}{#2}} %Legendre symbol
\newcommand\jacsa[2]{{\genfrac{(}{)}{}{}{#1}{#2}}^{*}} %Knopps jac-sym 1
\newcommand\jacsb[2]{{\genfrac{(}{)}{}{}{#1}{#2}}_{*}} %Knopps jac-sym 2

\newcommand\Mac[3]{M\left(\frac{#1}{#2};#3\right)}
\newcommand\Nac[3]{N\left(\frac{#1}{#2};#3\right)}
\newcommand\Jac[3]{J\left(\frac{#1}{#2};#3\right)}
\newcommand\Jabc[4]{J\left({#1},{#2},{#3};#4\right)}
\newcommand\Lpar[1]{\left(#1\right)}
\newcommand\Mell[3]{\mathcal{M}\left(\frac{#1}{#2};#3\right)}
\newcommand\Nell[3]{\mathcal{N}\left(\frac{#1}{#2};#3\right)}
\newcommand\Fell[4]{\mathcal{F}_{#1}\left(\frac{#2}{#3};#4\right)}
\newcommand\SFell[4]{\mathcal{F}_{#1}^{*}\left(\frac{#2}{#3};#4\right)}

\newcommand\SFTk{\SFell{1}{1}{p}{\frac{z+k}{p}}}
\newcommand\SFTZ{\SFell{1}{1}{p}{\frac{z}{p}}}

\newcommand\Gell[4]{\mathcal{G}_{#1}\left(\frac{#2}{#3};#4\right)}
\newcommand\Jpz[1]{\mathcal{J}\left(\frac{1}{p};{#1}\right)}
\newcommand\JSpz[1]{\mathcal{J}^{*}\left(\frac{1}{p};{#1}\right)}
\newcommand\Rpz[1]{\mathcal{R}_p\left(#1\right)}
\newcommand\fsqiz{\frac{1}{\sqrt{-i z}}}

\newcommand\sqiz{\sqrt{-i z}}
\newcommand\Sz{-\frac{1}{z}}
\newcommand\THA[3]{\Theta_1\Lpar{\frac{#1}{#2};#3}}
\newcommand\twidit[1]{\overset {\text{\lower 3pt\hbox{$\sim$}}}#1}
\newcommand\dtwidit[1]{\overset {\text{\lower 6pt\hbox{$\sim$}}}#1}
\newcommand\Wtwid{\overset {\text{\lower 3pt\hbox{$\sim$}}}W}
\newcommand\gtwid{\overset {\text{\lower 3pt\hbox{$\sim$}}}g}
\newcommand\ttwid{\overset {\text{\lower 3pt\hbox{$\sim$}}}\theta}
\newcommand\mutwid{\overset {\text{\lower 3pt\hbox{$\sim$}}}\mu}
\newcommand\Stwid{\twidit{S}}
\newcommand\Wdtwid{\Wtwid^{*}}
\newcommand\zcon{\overline{z}}
\newcommand\PPMat{\begin{pmatrix} p^2 & 0 \\ 0 & 1 \end{pmatrix}}
\newcommand\AMat{\begin{pmatrix} a & b \\ c & d \end{pmatrix}}
\newcommand\ApMat{\begin{pmatrix} a & pb \\ c/p & d \end{pmatrix}}
\newcommand\AppMat{\begin{pmatrix} a & p^2b \\ c/p^2 & d \end{pmatrix}}
\newcommand\NAMat{\begin{pmatrix} a & bN \\ c/N & d \end{pmatrix}}
\newcommand\BMat{\begin{pmatrix} a' & b' \\ c' & d' \end{pmatrix}}
\newcommand\SMat{\begin{pmatrix} 0 & -1 \\ 1 & 0 \end{pmatrix}}
\newcommand\TMat{\begin{pmatrix} 1 & 1 \\ 0 & 1 \end{pmatrix}}

\newcommand\BpMat{\begin{pmatrix}1 & b' \\ 0 & p^2\end{pmatrix}}
\newcommand\BbpMat{\begin{pmatrix}p & b'' \\ 0 & p\end{pmatrix}}
\newcommand\FL[1]{\left\lfloor#1\right\rfloor}
\newcommand\CL[1]{\left\lceil#1\right\rceil}
\newcommand\ks{k^{*}}

\newcommand\zbpp{\frac{z+b'}{p^2}}

\newcommand\strokeb[2]{\,\left\arrowvert\,\left[#1\right]_#2\right.}
\newcommand\stroke[3]{#1\,\left\arrowvert\,\left[#2\right]_{#3}\right.}
\newcommand\strokevb[3]{\left.#1\,\right\arrowvert\,\left[#2\right]_{#3}}
\newcommand\Kp{\mathcal{K}_{p,0}}
\newcommand\DK[1]{\mathcal{K}_{#1,0}}
\newcommand\ord{\mbox{ord}}         
\newcommand\ORD{\mbox{ORD}}         
\newcommand\hord{\mbox{ord}_{\mbox{\scriptsize holo}}}         
\newcommand\Zsum{\sum_{n\in\mathbb{Z}}}
\newcommand\gA{A}   %%was \gamma_0                           
\newcommand\fpart[1]{\left\{#1\right\}}
\allowdisplaybreaks
\newcommand\mylabel[1]{\label{#1}}
\newcommand\mybibitem[1]{\bibitem{#1}}
\newcommand\thm[1]{\ref{thm:#1}}
\newcommand\lem[1]{\ref{lem:#1}}
\newcommand\corol[1]{\ref{cor:#1}}
\newcommand\conjref[1]{\ref{conj:#1}}
\newcommand\propo[1]{\ref{propo:#1}}
\newcommand\eqn[1]{(\ref{eq:#1})}
\newcommand\sect[1]{\ref{sec:#1}}

\newcommand\subsect[1]{\ref{subsec:#1}}

\begin{document}
%%BEGIN MACROS%%%%%%%%%%%%%%%%%%%%%%%%%%%%%%%%%%%%%%%%%%%%%%%%%%%%%%%%%%%%%%%
\newcommand{\beqs}{\begin{equation*}}
\newcommand{\eeqs}{\end{equation*}}
\newcommand{\beq}{\begin{equation}}
\newcommand{\eeq}{\end{equation}}
%%% special macs
%%END   MACROS%%%%%%%%%%%%%%%%%%%%%%%%%%%%%%%%%%%%%%%%%%%%%%%%%%%%%%%%%%%%%%%
%%PRELIMINARY VERSION
\title[Dyson's rank function]
{Transformation properties for Dyson's rank function}

% Information for first author
\author{F. G. Garvan}
\address{Department of Mathematics, University of Florida, Gainesville,
FL 32611-8105}
%%\curraddr{Department of Mathematics and Statistics,
\email{fgarvan@ufl.edu}
\thanks{The author was supported in part by a grant from 
the Simon's Foundation (\#318714).
A preliminary version of this paper was first given on May 10, 2015
at the 
International Conference on Orthogonal Polynomials and
$q$-series at University of Central Florida,
Orlando, in honour of Mourad Ismail's $\mbox{70}^{th}$ birthday.}
%%\thanks{The second author was supported in part by NSA Grant H98230-09-1-0051.
%%}

%    General info

\date{June 17, 2016}                   

%%\dedicatory{Dedicated to the memory of
%%A.O.L. (Oliver) Atkin}
%%\dedicatory{Dedicated to the memory of Basil Gordon}                  


\begin{abstract}
At the 1987 Ramanujan Centenary meeting Dyson asked for a coherent 
group-theoretical structure for Ramanujan's mock theta functions 
analogous to Hecke's theory of modular forms. 
Many of Ramanujan's mock theta functions can be written in terms of $R(\zeta,q)$,
where $R(z,q)$ is 
the two-variable generating function of Dyson's rank function and 
$\zeta$ is a root of unity. Building on earlier work of Watson, Zwegers, Gordon and
McIntosh, and motivated by Dyson's question,
Bringmann, Ono and Rhoades studied transformation properties of $R(\zeta,q)$.
In this paper we
strengthen and extend the results of Bringmann, Rhoades and Ono, and the later work 
of Ahlgren and Treneer. As an application we give a new proof of Dyson's rank
conjecture and show that Ramanujan's Dyson rank identity modulo $5$ 
from the Lost Notebook has an analogue
for all primes greater than $3$.
The proof of this analogue was inspired by recent work of Jennings-Shaffer on overpartition rank differences mod $7$.
\end{abstract}

\maketitle

%SECTION 1%%%%%%%%%%%%%%%%%%%%%%%%%%%%%%%%%%%%%%%%%%%%%%%%%%%%%%%%%%%%%%%%
\section{Introduction}
\mylabel{sec:intro}

Let $p(n)$ denote the number of partitions of $n$. There are many known congruences
for the partition function. The simplest and most famous were found and proved by
Ramanujan:
\begin{align*}
p(5n+4) &\equiv 0 \pmod{5},\\
p(7n+5) &\equiv 0 \pmod{7},\\
p(11n+6) &\equiv 0 \pmod{11}.
\end{align*}
In 1944, Dyson \cite{Dy44} conjectured striking combinatorial interpretations of the
first two congruences. He defined the rank of a partition as the largest part minus the number
of parts and conjectured that the rank mod $5$ divided the partitions of $5n+4$ into $5$
equal classes and that the rank $7$ divided the partitions of $7n+5$ into $7$ equal classes.
He conjectured the existence a partition statistic he called the \textit{crank} which
would likewise divide the partitions of $11n+6$ into $11$ equal classes. Dyson's mod $5$ and
$7$ rank conjectures were proved by Atkin and Swinnerton-Dyer \cite{At-Sw}. The mod $11$ 
crank conjecture was solved by the author and Andrews \cite{An-Ga88}.

Let $N(m,n)$ denote the number of partitions of $n$ with rank $m$. We let $R(z,q)$ denote
the two-variable generating function for the Dyson rank function so that
$$
R(z,q) = \sum_{n=0}^\infty \sum_m N(m,n)\,z^m\,q^n.
$$

Throughout this paper we will use the standard $q$-notation:
\begin{align*}
(a;q)_{\infty} &= \prod_{k=0}^\infty (1-aq^k),
\\
(a;q)_{n} &= \frac{(a;q)_{\infty}}{(aq^n;q)_{\infty}},
\\
(a_1,a_2,\dots,a_j;q)_{\infty} 
&= (a_1;q)_{\infty}(a_2;q)_{\infty}\dots(a_j;q)_{\infty},
\\
(a_1,a_2,\dots,a_j;q)_{n} 
&=
(a_1;q)_{n}(a_2;q)_{n}\dots(a_j;q)_{n}.
\end{align*}

We have the following identities for the rank generating function $R(z,q)$:
\begin{align}
R(z,q) 
&= 1 + \sum_{n=1}^\infty \frac{q^{n^2}}{(zq,z^{-1}q;q)_n}
\mylabel{eq:Rzqid1}\\
&= \frac{1}{(q;q)_\infty} \Lpar{1 + \sum_{n=1}^\infty 
   \frac{(-1)^n (1 + q^n)(1-z)(1-z^{-1})}
   {(1 - zq^n)(1 - z^{-1}q^n)} \, q^{\frac{1}{2}n(3n+1)}}.
\mylabel{eq:Rzqid2}
\end{align}
See \cite[Eqs (7.2), (7.6)]{Ga88a}.

Let $N(r,t,n)$ denote the number of partitions of $n$ with congruent to $r$ mod $t$,
and let
$\zeta_p=\exp(2\pi i/p)$. Then
\beq
R(\zeta_p,q) = \sum_{n=0}^\infty \Lpar{\sum_{k=0}^{p-1} N(k,p,n)\,\zeta_p^k}\,q^n.
\mylabel{eq:Rzetaid}
\eeq
We restate
\begin{drconj}[1944]
\mylabel{conj:DRC}
\begin{align}
N(0,5,5n+4) &= N(1,5,5n+4) = \cdots  = N(4,5,5n+4) = \tfrac{1}{5}p(5n + 4),
\mylabel{eq:Dysonconj5}\\
N(0,7,7n+5) &= N(1,7,7n+5) = \cdots  = N(6,7,7n+5) = \tfrac{1}{7}p(7n + 5).    
\mylabel{eq:Dysonconj7}
\end{align}
\end{drconj}
Dyson's rank conjecture was first proved by Atkin and Swinnerton-Dyer \cite{At-Sw}.
As noted in \cite{Ga88a}, \cite{Ga88b}, 
Dyson's rank conjecture follows from an identity
in Ramanujan's Lost Notebook \cite[p.20]{Ra88}, \cite[Eq. (2.1.17)]{An-Be-RLNIII}.
We let $\zeta_5$ be a primitive $5$th root of unity. Then
\begin{align}
R(\zeta_5,q) &= A(q^5) + (\zeta_5 + \zeta_5^{-1}-2)\,\phi(q^5) + q\,B(q^5) + 
(\zeta_5+\zeta_5^{-1})\,q^2\,C(q^5)
\mylabel{eq:Ramid5}\\
&\quad 
- (\zeta_5+\zeta_5^{-1})\,q^3\left\{D(q^5) - (\zeta_5^2 + \zeta_5^{-2}  - 2)\frac{\psi(q^5)}{q^5}
\right\},
\nonumber
\end{align}
where
$$
A(q) = \frac{(q^2,q^3,q^5;q^5)_\infty}{(q,q^4;q^5)_\infty^2},\,
B(q) = \frac{(q^5;q^5)_\infty}{(q,q^4;q^5)_\infty},\,
C(q) = \frac{(q^5;q^5)_\infty}{(q^2,q^3;q^5)_\infty},\,
D(q) = \frac{(q,q^4,q^5;q^5)_\infty}{(q^2,q^3;q^5)_\infty^2},
$$
and
$$
\phi(q) = -1 + \sum_{n=0}^\infty \frac{q^{5n^2}}{(q;q^5)_{n+1} (q^4;q^5)_n},\qquad
\psi(q) =  -1 + 
\sum_{n=0}^\infty \frac{q^{5n^2}}{(q^2;q^5)_{n+1} (q^3;q^5)_n}.
$$
We recognize the functions $A(q)$, $B(q)$, $C(q)$, $D(q)$ as being modular forms
except for a power of $q$. We rewrite Ramanujan's identity \eqn{Ramid5} in 
terms of generalized eta-products:
$$
\eta(z) := q^{\frac{1}{24}} \prod_{n=1}^\infty (1 - q^n),\qquad q=\exp(2\pi iz),
$$
and
$$
\eta_{t,r}(z) := q^{\tfrac{t}{2}\,P(r/t)} \prod_{\substack{n>0\\ n \equiv\pm r\pmod{t}}}
(1-q^n),
$$
where $P(t):=\fpart{t}^2 - \fpart{t} + 1/6$ and $0 < r < t$ are integers. We have
\begin{align}
&q^{-\frac{1}{24}}\left(R(\zeta_5,q) 
 - (\zeta_5 + \zeta_5^4 - 2) \,\phi(q^5) 
 + (1+ 2 \zeta_5 + 2\zeta_5^4 ) \,q^{-2}\,\psi(q^5)\right)  
\mylabel{eq:Ramid5V2}\\
&
= \frac{\eta(25z) \, \eta_{5,2}(5z)}{\eta_{5,1}(5z)^2}
+
\frac{\eta(25z)}{\eta_{5,1}(5z)}
+ (\zeta_5 + \zeta_5^4) \, 
\frac{\eta(25z)}{\eta_{5,2}(5z)}
- (\zeta_5 + \zeta_5^4) \, 
\frac{\eta(25z)\, \eta_{5,1}(5z)}{\eta_{5,2}(5z)^2}.
\nonumber
\end{align}
Equation \eqn{Ramid5}, or equivalently \eqn{Ramid5V2}, give the
$5$-dissection of the $q$-series expansion of $R(\zeta_5,q)$.
We observe that the function on the right side of \eqn{Ramid5V2} is a weakly
holomorphic modular form (with multiplier) of weight $\tfrac{1}{2}$
on the group $\Gamma_0(25) \cap \Gamma_1(5)$.

In Theorem \thm{mainJp0} below we generalize the analogue of 
Ramanujan's result \eqn{Ramid5V2} to all primes $p>3$.
%%See our main Theorem \thm{mainJp} and its Corollary \corol{Rpz} below
%%in Section \sect{EXTRAM}. 
For $p>3$ prime and $1 \le a \le \tfrac{1}{2}(p-1)$ define
\beq
\Phi_{p,a}(q) 
:= 
\begin{cases}
\displaystyle
\sum_{n=0}^\infty \frac{q^{pn^2}} 
{(q^a;q^p)_{n+1} (q^{p-a};q^p)_n},
 & \mbox{if $0 < 6a < p$,}\\
\displaystyle
-1 + \sum_{n=0}^\infty \frac{q^{pn^2}} 
{(q^a;q^p)_{n+1} (q^{p-a};q^p)_n},
 & \mbox{if $p < 6a < 3p$,}\\
\end{cases}
\mylabel{eq:phipadef}
\eeq
and
\begin{align}      
\Rpz{z} 
&:= q^{-\frac{1}{24}} R(\zeta_p,q) 
- 
\chi_{12}(p) \,
\sum_{a=1}^{\frac{1}{2}(p-1)} (-1)^a 
 \,
  \left( \zeta_p^{3a + \frac{1}{2}(p+1)} + \zeta_p^{-3a - \frac{1}{2}(p+1)}  
  \right.
\mylabel{eq:Rpdef}\\
&\hskip 2in \left.  - 
    \zeta_p^{3a + \frac{1}{2}(p-1)} - \zeta_p^{-3a - \frac{1}{2}(p-1)}\right)
  \,
q^{\tfrac{a}{2}(p-3a)-\tfrac{p^2}{24}}\,\Phi_{p,a}(q^p),
\nonumber               
\end{align}      
where
\beq
\chi_{12}(n) := \leg{12}{n} = 
\begin{cases}
1 & \mbox{if $n\equiv\pm1\pmod{12}$,}\\
-1 & \mbox{if $n\equiv\pm5\pmod{12}$,}\\
0 & \mbox{otherwise,}
\end{cases}
\mylabel{eq:chi12}
\eeq
and as usual $q=\exp(2\pi iz)$ with  $\Im(z)>0$.
One of our main results is
\begin{theorem}
\mylabel{thm:mainJp0}
Let $p > 3$ be prime. Then the function 
$$
\eta(p^2 z) \, \Rpz{z} 
$$
is a weakly holomorphic modular form of weight $1$ on the group
$\Gamma_0(p^2) \cap \Gamma_1(p)$.
\end{theorem}
\begin{remark}
The form of this result is suggested by Ramanujan's identity \eqn{Ramid5V2}.
The proof of this result uses the theory of weak harmonic Maass forms and
was inspired by a recent result of Jennings-Shaffer \cite{JS16a} on 
overpartition rank differences mod $7$. Jennings-Shaffer was the first to 
prove a result of this type using the theory of weak harmonic Maass forms.  
\end{remark}
As a consequence we have
\begin{cor}
\mylabel{cor:Kp0}
Let $p > 3$ be prime and $s_p=\tfrac{1}{24}(p^2-1)$. Then the function 
$$
\prod_{n=1}^\infty (1-q^{pn})\, 
\Bigg(
\sum_{n=\CL{\frac{1}{p}(s_p)}}^\infty \left(\sum_{k=0}^{p-1} N(k,p,pn -s_p)\,\zeta_p^k\right)q^n
$$                      
is a weakly holomorphic modular form of weight $1$ on the group
$\Gamma_1(p)$.
\end{cor}
\begin{remark}
Corollary \corol{Kp0} is the case $m=0$ of Proposition \propo{Kpmprop}(i) below.
\end{remark}

In Ramanujan's identity \eqn{Ramid5} we see that the coefficient of $q^{5n+4}$ is zero,
and this implies Dyson's rank $5$ conjecture \eqn{Dysonconj5} in view of \eqn{Rzetaid}.
The analog of \eqn{Ramid5} for the prime $7$ does not appear in Ramanujan's lost notebook
although Ramanujan wrote  the left side \cite[p.19]{Ra88} and wrote some of the functions
involved in coded form on \cite[p.71]{Ra88}. The complete identity is given by \
Andrews and Berndt \cite[Eq. (2.1.42)]{An-Be-RLNIII}. As noted by the author \cite[Theorem 4, p.20]{Ga88b} Ramanujan's identity \eqn{Ramid5} is actually equivalent to 
Atkin and Swinnerton-Dyer's \cite[Theorem 4, p.101]{At-Sw}.
In Section \sect{DRC}
we give a new proof of Dyson's Conjecture \conjref{DRC}
as well as some of Atkin and Hussain's \cite{At-Hu} results on the rank mod $11$
and O'Brien's \cite{OB} results on the rank mod $13$.

We now explain the connection between Ramanujan's mock theta functions
and Dyson's rank functions. 
On \cite[p.20]{Ra88} Ramanujan gives four identities for some 
mock theta functions
of order $5$. For example,
$$
\chi_0(q) = 2 + 3 \phi(q) - A(q),
$$
where 
$$
\chi_0(q) = \sum_{n=0}^\infty \frac{q^n}{(q^{n+1};q)_n} = 1 + \sum_{n=0}^\infty
\frac{q^{2n+1}}{(q^{n+1};q)_{n+1}}.
$$
The other three identities correspond to \cite[Eqs (3.2), (3.6), (3.7)]{An-Ga89}.
These identities are the Mock Theta Conjectures which were later proved by Hickerson \cite{Hi88a}.
Thus Ramanujan's mock theta functions of order $5$ are related to $R(z,q)$ when $z=\zeta_5$.
   All of Ramanujan's third order mock theta functions can also be written in
terms of $R(z,q)$.  For example,
\begin{align*}
f(q) &= \sum_{n\ge0} \frac{q^{n^2}}{(-q;q)_n^2} = R(-1,q),\\
\omega(q) &= \sum_{n\ge0} \frac{q^{2n(n+1)}}{(q;q^2)_{n+1}^2} =  g(q;q^2),
\end{align*}
where 
\beq
g(x,q) = x^{-1}\left(-1 + \frac{1}{1-x} R(x,q) \right).
\mylabel{eq:gR}
\eeq
A catalogue of these and analogous identities for Ramanujan's mock theta functions
is given by
Hickerson and Mortenson \cite[Section 5]{Hi-Mo12}.

The main emphasis of this paper is transformations for Dyson's rank function
$R(\zeta,q)$ and thus for Ramanujan's mock theta functions.
We begin with a quote from Freeman Dyson.
\begin{quote}
The mock theta-functions give us tantalizing hints of a grand synthesis
still to be discovered. Somehow it should be possible to build them
into a coherent group-theoretical structure, analogous to the structure
of modular forms which Hecke built around the old theta-functions of
Jacobi. This remains a challenge for the future.

\smallskip
\hskip2in Freeman Dyson, 1987

\hskip2in Ramanujan Centenary Conference
\end{quote}
In this paper we continue previous work on Dyson's challenge.
First we describe the genesis of this work.
Watson \cite{Wa36} found transformation formulas for the third order 
functions in 
terms of Mordell integrals. For example,
$$
q^{-1/24} f(q) = 2 \sqrt{\frac{2\pi}{\alpha}} q_1^{4/3} \omega(q_1^2)
 + 4 \sqrt{\frac{3\alpha}{2\pi}} \int_0^\infty e^{-3\alpha x^2/2}
\frac{\sinh \alpha x}{\sinh 3\alpha x/2} \, dx,
$$
where
$$
q=\exp(-\alpha),\quad q_1 = \exp(-\beta),\quad \alpha \beta = \pi^2.
$$
The derivation of
transformation formulas for the other Ramanujan mock theta functions was carried
out
in a series of papers by Gordon and McIntosh. A summary of these results can be found
in \cite{Go-McI}.

The big breakthrough came when Zwegers \cite{Zw00}, \cite{Zw-thesis}  realised
how Ramanujan's mock theta functions occurred as the holomorphic part of certain
real analytic modular forms. An example for the third order functions $f(q)$, $\omega(q)$,
$\omega(-q)$ is given in
\begin{theorem}[Zwegers \cite{Zw00}]
Define $F(z) = (f_0, f_1, f_2)^T$ by

%%\pause
$$
f_0(z) = q^{-1/24} f(q), \quad
f_1(z) = 2q^{1/3} \omega(q^{1/2}), \quad
f_1(z) = 2q^{1/3} \omega(-q^{1/2}), 
$$
where
$q=\exp(2\pi iz),\quad z \in \mathfrak{h}$.
Define
$$
G(z) = 2i\sqrt{3} \int_{-\overline{z}}^{i\infty} 
\frac{(g_0(\tau),g_1(\tau),g_2(\tau))^T}{\sqrt{-i(\tau+z)}}\,d\tau,
$$
where
\begin{align*}
g_0(\tau) &= \sum_n (-1)^n \left(n+\frac{1}{3}\right) q^{3(n + 1/3)^2/2},\\
g_1(\tau) &= -\sum_n \left(n+\frac{1}{6}\right) q^{3(n + \tfrac{1}{6})^2/2},\\
g_2(\tau) &= \sum_n \left(n+\frac{1}{3}\right) q^{3(n + 1/3)^2/2}.
\end{align*}
Then
$$
H(\tau) = F(\tau) - G(\tau)
$$
is a (vector-valued) real analytic modular form of weight $1/2$ 
satisfying
$$
H(\tau+1) = \begin{pmatrix}
\zeta_{24}^{-1} & 0 & 0 \\
0 & 0 & \zeta_3 \\
0 & \zeta_3 & 0 
\end{pmatrix}\,H(\tau),
$$
$$
\frac{1}{\sqrt{-i\tau}}H\left(\frac{-1}{\tau}\right) = 
\begin{pmatrix}
0 & 1 & 0 \\
1 & 0 & 0 \\
0 & 0 & -1
\end{pmatrix}\,H(\tau).
$$
\end{theorem}

Bringmann and Ono \cite{Br-On10}, \cite{Br08b}, \cite{On08} 
extended this theorem to $R(\zeta,q)$ when
$\zeta$ is a more general root of unity. 
\begin{theorem}[Bringmann and Ono, Theorem 3.4 \cite{Br-On10}]
\mylabel{thm:BOvecMF}
Let $c$ be an odd positive integer and let $0 < a < c$. Then $q^{-1/24} R(\zeta_c^a,q)$ is
the holomorphic part of a component of a vector valued weak Maass form of weight $1/2$
for the full modular group $\SL_2(\mathbb{Z})$.
\end{theorem}
\begin{remark}
By component of a vector valued  weak Maass form we mean an element of the set $V_c$ defined
on \cite[p.441]{Br-On10}. We give this result explicitly below in
Corollary \corol{Vc}.
The definition of a vector valued Maass form of weight $k$ for 
$\SL_2(\mathbb{Z})$ is given on \cite[p.440]{Br-On10}.
\end{remark}

Bringmann and Ono used this theorem to obtain a
subgroup of the full modular group on which $q^{-1/24} R(\zeta,q)$ is the holomorphic part a
weak Maass form of weight $1/2$. We state their theorem in the case where
$\zeta$ is a $p$th root of unity.
\begin{theorem}[Bringmann and Ono \cite{Br-On10}]
\mylabel{thm:BrOnD}
Let $p>3$ be prime, and $0 < a < p$. 
Define
$$
\theta(\alpha,\beta;z) := \sum_{n\equiv\alpha\pmod{\beta}} n q^{n^2}\qquad 
q=\exp(2\pi iz),
$$
$$
\Theta\left(\frac{a}{p};z\right) :=
\sum_{m\pmod{2p}} (-1)^m \sin\left(\frac{a\pi(6m+1)}{p}\right)
\theta\left(6m+1,12p,\frac{z}{24}\right),
$$
$$
S_1\left(\frac{a}{p};z\right) :=
-i\sin\left(\frac{\pi a}{p}\right) 2\sqrt{2} p
\int_{-\overline{z}}^{i\infty} 
\frac{
       \Theta\left(\frac{a}{p};24p^2\tau\right)\,d\tau
     }
{\sqrt{-i(\tau+z)}},
$$
Then
$$
D\left(\frac{a}{p};z\right) := q^{-p^2}\,R\left(\zeta_p^{a};q^{24p^2}\right) 
- S_1\left(\frac{a}{p};z\right)
$$
is a weak Maass form of weight $1/2$ on $\Gamma_1(576\cdot p^4)$.
\end{theorem}

Bringmann, Ono and Rhoades \cite{Br-On-Rh} applied this theorem to prove
\begin{theorem}[Bringmann, Ono and Rhoades; Theorem 1.1 \cite{Br-On-Rh}]
\mylabel{thm:BrOnRh}
Suppose $t\ge5$ is prime, $0\le r_1,r_2 < t$ and $0\le d< t$. Then the following are
true:
\begin{enumerate}
\item[(i)]
If $\leg{1-24d}{t}=-1$, then 
$$
\sum_{n=0}^\infty (N(r_1,t,tn+d)-N(r_2,t,tn+d))q^{24(tn+d)-1}
$$
is a weight $1/2$ weakly holomorphic modular form on $\Gamma_1(576\cdot t^6)$.
\item[(ii)]
Suppose that $\leg{1-24d}{t}=1$. If $r_1,r_2\not\equiv \tfrac{1}{2}(\pm1\pm \alpha)\pmod{t}$,
where $\alpha$ is any integer for which $0\le \alpha < 2t$ and $1-24d\equiv\alpha^2\pmod{2t}$, 
then
$$
\sum_{n=0}^\infty (N(r_1,t,tn+d)-N(r_2,t,tn+d))q^{24(tn+d)-1}
$$
is a weight $1/2$ weakly holomorphic modular form on $\Gamma_1(576\cdot t^6)$.
\end{enumerate}
\end{theorem}

Ahlgren and Treneer \cite{Ah-Tr08} strengthened this theorem to include the case
$24d\equiv 1 \pmod{t}$.
\begin{theorem}[Ahlgren and Treneer \cite{Ah-Tr08}]
\mylabel{thm:AhTr}
Suppose $p\ge5$ is prime and that $0\le r < t$. Then
$$
\sum_{n=1}^\infty \Lpar{N\Lpar{r,p,\frac{pn+1}{24}} 
- \frac{1}{t} p\Lpar{\frac{pn+1}{24}}}\,q^n
$$
is a weight $1/2$ weakly holomorphic modular form on $\Gamma_1(576\cdot p^4)$.
\end{theorem}
\begin{remark}
Ahlgren and Trenneer \cite{Ah-Tr08} also derived many analogous results when $t$ is not
prime.
\end{remark}

In this paper we strengthen many of the results of Ahlgren, Bringmann, Ono, Rhoades and Trenneer
\cite{Ah-Tr08}, \cite{Br-On10}, \cite{Br-On-Rh}.
In particular
\begin{enumerate}
\item[(i)] We make Theorem \thm{BOvecMF} (Bringmann and Ono \cite[Theorem 3.4]{Br-On10}
more explicit. See Theorem \thm{Gtrans} below. In this theorem we have corrected
Bringmann and Ono's definition of the functions
$\mathcal{G}_2\Lpar{\frac{a}{c};z}$ and $\mathcal{G}_2(a,b,c;z)$.
\item[(ii)] For the case $p>3$ prime we strengthen 
Theorem \thm{BrOnD} (Bringmann and Ono \cite[Theorem 1.2]{Br-On10}).  We show that the group
$\Gamma_1(576\cdot p^4)$ can be enlarged to  $\Gamma_0(p^2) \cap \Gamma_1(p)$ but
with a simple multiplier. See Corollary \corol{Ftrans}. We prove more.
In Theorem \thm{mainthm} we basically determine the image
of Bringmann and Ono's function
$D\left(\frac{a}{p};\frac{z}{24p^2}\right)$ under the group $\Gamma_0(p)$.
\item[(iii)] 
We strengthen Theorem \thm{BrOnRh}(i) 
(Bringmann, Ono and Rhoades \cite[Theorem 1.1(i)]{Br-On-Rh}), 
and Theorem \thm{AhTr} (Ahlgren and Treneer \cite[Theorem 1.6, p.271]{Ah-Tr08}).
In both cases 
we enlarge the group $\Gamma_1(576\cdot p^4)$ to the group  $\Gamma_1(p)$ 
but with a simple multiplier. See Corollary \corol{Rpm} below.
\item[(iv)] We take a different approach to Bringmann, Ono and Rhoades 
Theorem \thm{BrOnRh}(ii). We handle the residue classes $\leg{1-24d}{t}=1$
in a different way. For these residue classes we show the definition of the 
functions involved may be adjusted to make them holomorphic modular forms.
In particular see equation \eqn{Kpm2prop}, Theorem \thm{Kpthm} and its
Corollary \corol{Rpm}.
%%\item[(v)] As mentioned before we are able to show how Ramanujan's Dyson rank
%%identity \eqn{Ramid5} or \eqn{Ramid5V2} may be extended to all primes $p > 3$.
%%See Section \sect{EXTRAM} and Theorem \thm{mainJp} in particular.
%%This result is completely new.
%%\item[(vi)] We give a new proof of the Dyson Rank Conjecture \eqn{Dysonconj5}--\eqn{Dysonconj7}.
%%See Section \subsect{proofDRC}.
%%\item[(vii)] We also use our method to derive Atkin and Hussain's \cite{At-Hu} results for the
%%rank mod $11$ and O'Brien's \cite{OB} results for rank mod $13$.
%%See Sections \subsect{rank11}--\subsect{rank13} below.
\end{enumerate}

The paper is organized as follows. In Section \sect{prelim} we go over
Bringmann and Ono's transformation results for various Lambert series and Mordell
integrals. We make all results explicit and correct some errors. In Section \sect{VMFS}
we describe Bringmann and Ono's vector-valued Maass form of weight $\tfrac{1}{2}$,
correcting some definitions and making all results explicit. In Section \sect{MFM}
we derive a Maass form multiplier for Bringmann and Ono's function 
$\mathcal{G}_1\Lpar{\frac{\ell}{p};z}$ for the group $\Gamma_0(p)$ when $p > 3$
is prime. In Section \sect{EXTRAM} we prove Theorem \thm{mainJp0}
which extends Ramanujan's Dyson rank identity
\eqn{Ramid5} or \eqn{Ramid5V2} to all primes $p > 3$.
In Section \sect{DRC}
we give a new proof of the Dyson Rank Conjecture \eqn{Dysonconj5}--\eqn{Dysonconj7}.
We extend Ahlgren, Bringmann, Ono, Rhoades and Treneer's results, mentioned in  (ii)--(iv)
above, for all primes $p >3$.
We also use our method to derive Atkin and Hussain's \cite{At-Hu} results for the
rank mod $11$ and O'Brien's \cite{OB} results for rank mod $13$.
%%See Sections \subsect{rank11}--\subsect{rank13} below.
In Section \sect{conclusion} we mention the recent results of
Hickerson and Mortenson \cite{Hi-Mo16}, who have a different approach and method
of proof of Bringmann and Ono's work.

%SECTION 1B%%%%%%%%%%%%%%%%%%%%%%%%%%%%%%%%%%%%%%%%%%%%%%%%%%%%%%%%%%%%%%%%
%SECTION 2%%%%%%%%%%%%%%%%%%%%%%%%%%%%%%%%%%%%%%%%%%%%%%%%%%%%%%%%%%%%%%%%
\section{Preliminaries}
\mylabel{sec:prelim}

Following \cite{Br-On10} we define a number of functions.
Suppose $0 < a < c$ are integers, and assume throughout that $q:=\exp(2\pi iz)$.
We define
\begin{align*}
\Mac{a}{c}{z} &:= \frac{1}{(q;q)_\infty} \sum_{n=-\infty}^\infty 
                  \frac{(-1)^n q^{n+\frac{a}{c}}}
                       {1-q^{n+\frac{a}{c}}} \, q^{\frac{3}{2}n(n+1)}
\\
\Nac{a}{c}{z} &:= \frac{1}{(q;q)_\infty} \Lpar{1 + \sum_{n=1}^\infty 
                  \frac{(-1)^n (1 + q^n)\left(2 - 2\cos\left(\frac{2\pi a}{c}\right)\right)}
                       {1 - 2\cos\left(\frac{2\pi a}{c}\right)q^n + q^{2n}} 
                        \, q^{\frac{1}{2}n(3n+1)}}.
\end{align*}
For integers $0 \le a < c$, $0 < b < c$ define
$$                   
M(a,b,c;z) := 
\frac{1}{(q;q)_\infty} \sum_{n=-\infty}^\infty 
                  \frac{(-1)^n q^{n+\frac{a}{c}}}
                       {1- \zeta_c^b q^{n+\frac{a}{c}}} \, q^{\frac{3}{2}n(n+1)},
$$                
where $\zeta_c := \exp(2\pi i/c)$. In addition for $\frac{b}{c}\not\in\{0,\frac{1}{6},
\frac{1}{2},\frac{5}{6}\}$ define
$$
k(b,c) :=
\begin{cases}
0 & \mbox{if $0 < \frac{b}{c} < \frac{1}{6}$}, \\
1 & \mbox{if $\frac{1}{6} < \frac{b}{c} < \frac{1}{2}$}, \\
2 & \mbox{if $\frac{1}{2} < \frac{b}{c} < \frac{5}{6}$}, \\
3 & \mbox{if $\frac{5}{6} < \frac{b}{c} < \frac{1}{6}$},  
\end{cases}
$$
and
$$                    
N(a,b,c;z):= 
\frac{1}{(q;q)_\infty} \left(
  \frac{i \zeta_{2c}^{-a} q^{\frac{b}{2c}}}{2\Lpar{1 - \zeta_c^{-a} q^{\frac{b}{c}}}}
  + \sum_{m=1}^\infty K(a,b,c,m;z) q^{\frac{m(3m+1)}{2}}\right),
$$                   
where
\begin{align*}
&K(a,b,c,m;z)\\
&:=
(-1)^m \frac{
\sin\Lpar{\frac{\pi a}{c} - \Lpar{\frac{b}{c} + 2k(b,c)m}\pi z}
+\sin\Lpar{\frac{\pi a}{c} - \Lpar{\frac{b}{c} - 2k(b,c)m}\pi z} q^m}
{1 - 2\cos\Lpar{\frac{2\pi a}{c} - \frac{2\pi bz}{c}} q^m + q^{2m}}.
\end{align*}

%%%%%%%%%%%%%%%%%%%%% ADDED 06.16.16 %%%%%%%%%%%%%%%%%%%%%%%%%%%%%%%%%%%%%%%%%%
We need the following identities.
\beq
-1 + \frac{1}{1-z}\sum_{n=0}^\infty \frac{q^{n^2}}{(zq,z^{-1}q;q)_n}
=
\frac{z}{(q)_\infty} 
\sum_{n=-\infty}^\infty \frac{(-1)^n q^{\frac{3}{2}n(n+1)}}{1-zq^n},
\mylabel{eq:Lamid1}
\eeq
and
\beq
\frac{1}{(q)_\infty} 
\sum_{n=-\infty}^\infty \frac{(-1)^n q^{\frac{3}{2}n(n+1)}}{1-zq^n}
=
\frac{z}{(q)_\infty} 
\sum_{n=-\infty}^\infty \frac{(-1)^n q^{n}}{1-zq^n}
 q^{\frac{3}{2}n(n+1)}.
\mylabel{eq:Lamid2}
\eeq
Equation \eqn{Lamid1} is \cite[Eq.(7.10), p.68]{Ga88a} and 
\eqn{Lamid2} is an easy exercise.
Replacing $q$ by $q^c$ and $z$ by $q^a$ in \eqn{Lamid1}, \eqn{Lamid2} we find
that
\begin{align}
\sum_{n=0}^\infty \frac{q^{cn^2}}{(q^a;a^c)_{n+1} (q^{c-a};q^c)_n}
&= 1 + \frac{q^a}{(q^c;q^c)_\infty} 
\sum_{n=-\infty}^\infty \frac{(-1)^n q^{\frac{3c}{2}n(n+1)}}{1-q^{cn+a}}
\mylabel{eq:Macid}\\
&= 1 + q^a \, \Mac{a}{c}{cz}.
\nonumber
\end{align}
By \eqn{Rzqid2} we have
\beq
R(\zeta_c^a,q) = \Nac{a}{c}{z}.
\mylabel{eq:RNacid}
\eeq
%%%%%%%%%%%%%%%%%%%%%%%%%%%%%%%%%%%%%%%%%%%%%%%%%%%%%%%%%%%%%%%%%%%%%%%%%%%%%%%

    Extending earlier work of Watson \cite{Wa36}, 
Gordon and McIntosh \cite{Go-McI03}, Bringmann and Ono \cite{Br-On10}
found transformation formula for all these functions in terms of Mordell integrals.
We define the following Mordell integrals
$$
\Jac{a}{c}{\alpha} :=
\int_0^\infty e^{-\frac{3}{2}\alpha x^2} 
\frac{\cosh\Lpar{\Lpar{\frac{3a}{c}-2}\alpha x} +
      \cosh\Lpar{\Lpar{\frac{3a}{c}-1}\alpha x}}
     {\cosh(3\alpha x/2) }\, dx,
$$
and
$$
J(a,b,c;\alpha) :=
\int_{-\infty}^\infty e^{-\frac{3}{2}\alpha x^2 + 3 \alpha x \frac{a}{c}}
\frac{\Lpar{\zeta_c^b e^{-\alpha x} + \zeta_c^{2b} e^{-2\alpha x}}}
{\cosh\Lpar{3\alpha\frac{x}{2} - 3\pi i\frac{b}{c}} }\,dx.
$$
Following \cite{Br-On10} we adjust the definitions of the 
$N$- and $M$-functions so that the transformation formulas
are tidy. We define
\begin{align}
\Nell{a}{c}{z} &:=\csc\Lpar{\frac{a\pi}{c}}\, q^{-\frac{1}{24}}\,\Nac{a}{c}{z},
\mylabel{eq:Ndef}\\
\Mell{a}{c}{z} &:=2 q^{\frac{3a}{2c}\left(1-\frac{a}{c}\right) - \frac{1}{24}} \, \Mac{a}{c}{z},
\mylabel{eq:Mdef}\\
\mathcal{M}(a,b,c;z) &:= 
2 q^{\frac{3a}{2c}\left(1-\frac{a}{c}\right) - \frac{1}{24}} \, M(a,b,c;z), 
\mylabel{eq:Mabcdef}\\
\mathcal{N}(a,b,c;z) &:= 
4 \exp\Lpar{-2\pi i \frac{a}{c} k(b,c) + 3\pi i \frac{b}{c}\left(\frac{2a}{c}-1\right)}
\, \zeta_c^{-b} \, q^{\frac{b}{c}k(b,c) - \frac{3b^2}{2c^2} - \frac{1}{24}}\,
N(a,b,c;z).
\mylabel{eq:Nabcdef}
\end{align}
We now restate Bringmann and Ono's \cite[Theorem 2.3, p.435]{Br-On10} more explicitly.
\begin{theorem}
\mylabel{thm:BrOnThm}
Suppose that $c$ is a positive odd integer, and that $a$ and $b$ are integers
for which $0 \le a < c$ and $0< b < c$.
\begin{enumerate}
\item[(1)]
For $z\in\mathfrak{h}$ we have
\begin{align}
\Nell{a}{c}{z+1} &= \zeta_{24}^{-1}\,\Nell{a}{c}{z},\mylabel{eq:Ntrans1}\\
\mathcal{N}(a,b,c;z+1)&=
\begin{cases}
\zeta_{2c^2}^{3b^2} \, \zeta_{24}^{-1} \,\mathcal{N}(a-b,b,c;z) & \mbox{if $a\ge b$},\\
-\zeta_{2c^2}^{3b^2} \, \zeta_{c}^{-3b}\,\zeta_{24}^{-1}\, \mathcal{N}(a-b+c,b,c;z) & \mbox{otherwise},
\end{cases}\mylabel{eq:Ntrans2}\\
\Mell{a}{c}{z+1} &= \zeta_{2c}^{5a} \, \zeta_{2c^2}^{-3a^2}\,\zeta_{24}^{-1}\,
\mathcal{M}(a,a,c;z),\mylabel{eq:Mtrans1}\\
\mathcal{M}(a,b,c;z+1) &=
\zeta_{2c}^{5a} \, \zeta_{2c^2}^{-3a^2}\,\zeta_{24}^{-1}\,
\begin{cases}
\mathcal{M}(a,a+b,c;z) & \mbox{if $a+b<c$},\\
\Mell{a}{c}{z} & \mbox{if $a+b=c$},\\
\mathcal{M}(a,a+b-c,c;z) & \mbox{otherwise},
\end{cases}
\mylabel{eq:Mtrans2}
\end{align}
where $a$ is assumed to be positive in the first and third formula.
\item[(2)]
For $z\in\mathfrak{h}$ we have
\begin{align}
\fsqiz\,\Nell{a}{c}{\Sz} &= \Mell{a}{c}{z} + 2\sqrt{3}\sqiz\,\Jac{a}{c}{-2\pi iz},\mylabel{eq:Ntrans3}\\
\fsqiz\,\mathcal{N}\Lpar{a,b,c;\Sz}&=\mathcal{M}(a,b,c;z) + \zeta_{2c}^{-5b}\sqrt{3}
\sqiz\,J(a,b,c;-2\pi iz),\mylabel{eq:Ntrans4}\\
\fsqiz\,\Mell{a}{c}{\Sz} &= \Nell{a}{c}{z} - \frac{2\sqrt{3}i}{z}\,\Jac{a}{c}{\frac{2\pi i}{z}},
\mylabel{eq:Mtrans3}\\
\fsqiz\,\mathcal{M}\Lpar{a,b,c;\Sz} &= \mathcal{N}(a,b,c;z) - \zeta_{2c}^{-5b}\,
\frac{\sqrt{3}i}{z}\,
J\Lpar{a,b,c;\frac{2\pi i}{z}},
\mylabel{eq:Mtrans4}
\end{align}
where again $a$ is assumed to be positive in the first and third formula.
\end{enumerate}
\end{theorem}

We will write each of the Mordell integrals as a period integral of a theta-function.
Before we can do this we need some results of Shimura \cite{Sh}.
For integers $0\le k < N$ we define
$$
\ttwid(k,N;z) := \sum_{m=-\infty}^\infty (Nm+k) \exp\Lpar{\frac{\pi i z}{N}(Nm+k)^2}.
$$
We note that this corresponds to $\theta(z;k,N,N,P)$ in Shimura's 
notation \cite[Eq.(2.0), p.454]{Sh} (with $n=1$, $\nu=1$, and $P(x)=x$).
For integers $0\le a,b < c$ we define
$$
\Theta_1(a,b,c;z) := \zeta_{c^2}^{3ab}\, \zeta_{2c}^{-a}\,
\sum_{m=0}^{6c-1} (-1)^m \sin\Lpar{\frac{\pi}{3}(2m+1)}
\exp\Lpar{\frac{-2\pi i m a}{c}} \ttwid(2mc-6b+c,12c^2;z),
$$
and
\begin{align*}
\Theta_2(a,b,c;z) &:= \sum_{\ell=0}^{2c-1} \left( (-1)^\ell \, 
\exp\Lpar{\frac{-\pi ib}{c}(6\ell+1)} \, \ttwid(6c\ell + 6a + c,12c^2;z) \right.\\
&\qquad\qquad \left. + (-1)^\ell
\exp\Lpar{\frac{-\pi ib}{c}(6\ell-1)} \, \ttwid(6c\ell + 6a - c,12c^2;z) \right).
\end{align*}
%%%%%%%%%%%%%%%%%%%%%%%%%%%%%%%%%%%%%%%%%%%%%%%%%%%%%%%%%%%%%%%%%%%%%%%%%%%
%added 01.23.16 * added mylabel 06.07.16
An easy calculation gives
\begin{align}
\Theta_1(a,b,c;z) &=
6c\,\zeta_{c^2}^{3ab}\, \zeta_{2c}^{-a}\,
\sum_{n=-\infty}^\infty (-1)^n \Lpar{\frac{n}{3} + \frac{1}{6} - \frac{b}{c}}\,
\sin\Lpar{\frac{\pi}{3}(2n+1)}\,
\exp\Lpar{\frac{-2\pi i n a}{c}} 
\mylabel{eq:Theta1abcid}\\
&\hskip 2in \times \exp\Lpar{3\pi i z\Lpar{\frac{n}{3} + \frac{1}{6} - \frac{b}{c}}^2}.
\nonumber
\end{align}
%%%%%%%%%%%%%%%%%%%%%%%%%%%%%%%%%%%%%%%%%%%%%%%%%%%%%%%%%%%%%%%%%%%%%%%%%%%
We calculate the action of $\SL_2(\mathbb{Z}) $ on each of these theta-functions.
\begin{prop}
\mylabel{propo:thetatrans}
For integers $0\le a,b < c$ and $\tau\in\mathfrak{h}$ we have
\begin{align}
\Theta_1(a,b,c;z+1) &= \zeta_{2c^2}^{-3b^2}\,\zeta_{24}\, \Theta_1(a+b,b,c;z), 
\mylabel{eq:tt1}\\    
(-iz)^{-\frac{3}{2}}\,\Theta_1\Lpar{a,b,c;-\frac{1}{z}}
                       &= -\frac{\sqrt{3}i}{2}\, \Theta_2(a,b,c;z), 
\mylabel{eq:tt2}\\    
\Theta_2(a,b,c;z+1) &= \zeta_{2c^2}^{3a^2} \,\zeta_{24} \, \Theta_2(a,b-a,c;z),
\mylabel{eq:tt3}\\
(-iz)^{-\frac{3}{2}}\,\Theta_2\Lpar{a,b,c;-\frac{1}{z}}
                       &= \frac{2\sqrt{3}i}{3}\, \Theta_1(a,b,c;z).
\mylabel{eq:tt4}
\end{align}
\end{prop}
\begin{proof}
Transformations \eqn{tt1}, \eqn{tt3} are an easy calculation.
By \cite[Eq. (2.4), p.454]{Sh} we have
\begin{align*}
&(-iz)^{-3/2}\,\ttwid\Lpar{2mc - 6b + c, 12c^2;-\frac{1}{z}} \\
&\quad = 
(-i)\,(12c^2)^{-1/2}\,
\sum_{k\pmod{12c^2}} \exp\Lpar{\frac{\pi ik}{6c^2}(2mc + c - 6b)}\,
\ttwid(k,12c^2;z)
\end{align*}
Therefore
\begin{align*}
&(-iz)^{-3/2}\,\Theta_1\Lpar{a,b,c;-\frac{1}{z}} \\
&\quad = 
(-i)\,(12c^2)^{-1/2}\, \zeta_{c^2}^{3ab}\, \zeta_{2c}^{-a}\,
\sum_{m\pmod{6c}} (-1)^m \sin\Lpar{\frac{\pi}{3}(2m+1)}
\exp\Lpar{\frac{-2\pi i m a}{c}} \\
&\qquad {\hskip 1in}
\sum_{k\pmod{12c^2}} \exp\Lpar{\frac{\pi ik}{6c^2}(2mc + c - 6b)}\,
\ttwid(k,12c^2;z)\\
&\quad = 
\frac{(-i)}{2c\sqrt{3}} 
\,\zeta_{c^2}^{3ab}\, \zeta_{2c}^{-a}\,
\sum_{k\pmod{12c^2}} \ttwid(k,12c^2;z)\\
&\qquad {\hskip 0.5in}
\sum_{m\pmod{6c}} (-1)^m \exp\Lpar{2\pi i\left\{
\frac{k}{12c^2}(2mc + c - 6b) - m\frac{a}{c}\right\}}\,\\
&\qquad\qquad {\hskip 1in}
\times \frac{1}{2i}\Lpar{\exp\Lpar{\frac{\pi i}{3}(2m+1)} - 
                         \exp\Lpar{-\frac{\pi i}{3}(2m+1)}}\\
&\quad = 
\frac{(-1)}{4c\sqrt{3}} 
\,\zeta_{c^2}^{3ab}\, \zeta_{2c}^{-a}\,
\sum_{k\pmod{12c^2}} \exp\Lpar{\frac{\pi ik}{6c^2}(c-6b)}\,\ttwid(k,12c^2;z)\\
&\qquad {\hskip 1in}
   \left(\exp\Lpar{\frac{\pi i}{3}} 
   \sum_{m\pmod{6c}} \exp\Lpar{\frac{2\pi i}{6c}(k -6a + 5c)m} \right.\\ 
&\qquad \qquad {\hskip 1.5in}
    -
   \left.\exp\Lpar{-\frac{\pi i}{3}} 
   \sum_{m\pmod{6c}} \exp\Lpar{\frac{2\pi i}{6c}(k -6a + c)m} \right) \\
&\quad = \frac{-\sqrt{3}}{2} 
\,\zeta_{c^2}^{3ab}\, \zeta_{2c}^{-a}\,
\left(
         \sum_{
         \substack{k\pmod{12c^2} \\ k\equiv 6a - 5c \pmod{6c}}}
          \exp\Lpar{\frac{\pi i}{3} + \frac{\pi ik}{6c^2}(c-6b)}\,\ttwid(k,12c^2;z)\right.\\
& \quad {\hskip 1in}
\left.
         -
         \sum_{
         \substack{k\pmod{12c^2} \\ k\equiv 6a - c \pmod{6c}}}
          \exp\Lpar{-\frac{\pi i}{3} + \frac{\pi ik}{6c^2}(c-6b)}\,\ttwid(k,12c^2;z)\right).
\end{align*}
In the sum above we let $k = 6c\ell + 6a \pm c$ where $0 \le \ell \le 2c-1$, so
that 
$$
\pi i\Lpar{\frac{k}{6c^2}(c-6b)  \pm\frac{1}{3}} =
\pi i\Lpar{\frac{a}{c^2}(c-6b)  \pm \frac{1}{2} + \ell - \frac{b}{c}(6\ell\pm1)},
$$
and we find
\begin{align*}
&(-iz)^{-3/2}\,\Theta_1\Lpar{a,b,c;-\frac{1}{z}} \\
&\quad = 
\frac{-i\sqrt{3}}{2}\, \sum_{\ell=0}^{2c-1} \left( (-1)^\ell \, 
\exp\Lpar{\frac{-\pi ib}{c}(6\ell+1)} \, \ttwid(6c\ell + 6a + c,12c^2;z) \right.\\
&\qquad\qquad \left. + (-1)^\ell
\exp\Lpar{\frac{-\pi ib}{c}(6\ell-1)} \, \ttwid(6c\ell + 6a - c,12c^2;z) \right).\\
&\quad = 
 -\frac{\sqrt{3}i}{2}\, \Theta_2(a,b,c;z),
\end{align*}
which is transformation \eqn{tt2}.
Transformation \eqn{tt4} follows immediately from \eqn{tt2}.
\end{proof}

In addition we need to define
\beq
\THA{a}{c}{z} := \sum_{n=-\infty}^\infty (-1)^n (6n+1) \sin\Lpar{\frac{\pi a(6n+1)}{c}}
\exp\Lpar{3\pi iz\Lpar{n + \frac{1}{6}}^2}. 
\mylabel{eq:Theta1acdef}
\eeq
This coincides with Bringmann and Ono's
function $\Theta\Lpar{\frac{a}{c};z}$ which is given in \cite[Eq.(1.6), p.423]{Br-On10}.
An easy calculation gives
$$
\THA{a}{c}{z}  = -\frac{i}{2c} \, \Theta_2(0,-a,c;z).
$$
From Proposition \ref{propo:thetatrans} we have
\begin{cor}
\mylabel{cor:THAtrans}
\begin{align}
\THA{a}{c}{\tau+1} &= \zeta_{24}\, \THA{a}{c}{\tau}.
\mylabel{eq:THAtrans1}\\
(-i\tau)^{-\frac{3}{2}}\,\THA{a}{c}{-\frac{1}{\tau}}                          
                       &= \frac{\sqrt{3}}{3c}\, \Theta_1(0,-a,c;\tau).
\mylabel{eq:THAtrans2}
\end{align}
\end{cor}
Next we define
$$
\varepsilon_1\Lpar{\frac{a}{c};z} :=
\begin{cases}
\frac{-2}{\sqiz}\,\exp\Lpar{\frac{3\pi i}{z}\Lpar{\frac{a}{c}-\frac{1}{6}}^2} & 
\mbox{if $0<\frac{a}{c}<\frac{1}{6}$},\\
0 & \mbox{if $\frac{1}{6}<\frac{a}{c}<\frac{5}{6}$},\\
\frac{-2}{\sqiz}\,\exp\Lpar{\frac{3\pi i}{z}\Lpar{\frac{a}{c}-\frac{5}{6}}^2} & 
\mbox{if $\frac{5}{6}<\frac{a}{c}<1$},
\end{cases}
$$
$$
\varepsilon_1(a,b,c;z) :=
\begin{cases}
\frac{\zeta_{2c}^{b}}{\sqiz}\,\exp\Lpar{\frac{3\pi i}{z}\Lpar{\frac{a}{c}-\frac{1}{6}}^2} & 
\mbox{if $0\le\frac{a}{c}<\frac{1}{6}$},\\
0 & \mbox{if $\frac{1}{6}<\frac{a}{c}<\frac{5}{6}$},\\
\frac{\zeta_{2c}^{5b}}{\sqiz}\,\exp\Lpar{\frac{3\pi i}{z}\Lpar{\frac{a}{c}-\frac{5}{6}}^2} & 
\mbox{if $\frac{5}{6}<\frac{a}{c}<1$},
\end{cases}
$$
$$
\varepsilon_2\Lpar{\frac{a}{c};z} :=
\begin{cases}
2\,\exp\Lpar{-3\pi iz\Lpar{\frac{a}{c}-\frac{1}{6}}^2} & 
\mbox{if $0<\frac{a}{c}<\frac{1}{6}$},\\
0 & \mbox{if $\frac{1}{6}<\frac{a}{c}<\frac{5}{6}$},\\
2\,\exp\Lpar{-3\pi iz\Lpar{\frac{a}{c}-\frac{5}{6}}^2} & 
\mbox{if $\frac{5}{6}<\frac{a}{c}<1$},
\end{cases}
$$
and
$$
\varepsilon_2(a,b,c;z) :=
\begin{cases}
2\zeta_{c}^{-2b}\,\exp\Lpar{-3\pi iz\Lpar{\frac{a}{c}-\frac{1}{6}}^2} & 
\mbox{if $0\le\frac{a}{c}<\frac{1}{6}$},\\
0 & \mbox{if $\frac{1}{6}<\frac{a}{c}<\frac{5}{6}$},\\
2\,\exp\Lpar{-3\pi iz\Lpar{\frac{a}{c}-\frac{5}{6}}^2} & 
\mbox{if $\frac{5}{6}<\frac{a}{c}<1$}.
\end{cases}
$$

We are now ready to express each of our Mordell integrals as a period integral of a 
theta-function.
\begin{theorem}
\mylabel{thm:Jacids}
Let $a$, $b$, $c$ be as in Theorem \ref{thm:BrOnThm}. Then for $z\in\mathfrak{h}$ we have
\begin{align}
\frac{2\sqrt{3}}{iz} \, \Jac{a}{c}{\frac{2\pi i}{z}} &= \frac{i}{\sqrt{3}}\,
\int_0^{i\infty} \frac{\THA{a}{c}{\tau}}{\sqrt{-i(\tau + z)}}\,d\tau + 
\varepsilon_1\Lpar{\frac{a}{c};z},
\mylabel{eq:Jint1}
\\
2\,\sqrt{3}\,\sqiz\,\Jac{a}{c}{-2\pi i z} &=
-\frac{i}{3c}\, 
\int_0^{i\infty} \frac{\Theta_1(0,-a,c;\tau)}{\sqrt{-i(\tau + z)}}\,d\tau + 
\varepsilon_2\Lpar{\frac{a}{c};z},
\mylabel{eq:Jint2}\\
\frac{\sqrt{3}}{-2i z}\,\Jabc{a}{b}{c}{\frac{2\pi i}{z}}&=
\frac{1}{6c}\,\int_0^{i\infty} \frac{\Theta_1(a,b,c;\tau)}{\sqrt{-i(\tau + z)}}\,d\tau + 
\varepsilon_1(a,b,c;z),
\mylabel{eq:Jint3}
\\
\zeta_{2c}^{-5b}\,\sqrt{3}\,\sqiz \,\Jabc{a}{b}{c}{-2\pi iz}&=
- \zeta_{2c}^{-5b}\,\frac{i\sqrt{3}}{6c}\,
\int_0^{i\infty} \frac{\Theta_2(a,b,c;\tau)}{\sqrt{-i(\tau + z)}}\,d\tau + 
\varepsilon_2(a,b,c;z).             
\mylabel{eq:Jint4}
\end{align}
\end{theorem}
\begin{remark}
We have corrected the results of Bringmann and Ono \cite[p.441]{Br-On10} by including the
necessary correction factors $\varepsilon_1$ and $\varepsilon_2$.
\end{remark}
\begin{proof}
    First we prove \eqn{Jint1}. Assume $0<a<c$ are integers and $a/c\not\in\{\tfrac{1}{6},\, \tfrac{5}{6}\}$. 
We proceed as in the proof of \cite[Lemma 3.2, pp.436--437]{Br-On10}. 
By analytic continuation we may assume that $z=it$ and $t>0$.
We find that
$$
\Jac{a}{c}{\frac{2\pi}{t}} =
t\,\int_0^\infty e^{-3\pi tx^2} \, f(x)\,dx = t\,\int_{-\infty}^\infty e^{-3\pi tx^2}\,g_{11}(x)\,dx,
$$
where
\begin{align*}
f(z) &= 
\frac{\cosh\Lpar{\Lpar{3\frac{a}{c}-2}2\pi z} + \cosh\Lpar{\Lpar{3\frac{a}{c}-1}2\pi z}}
{\cosh(3\pi z)}, \\
g_{11}(z) &= 
\frac{\exp\Lpar{\Lpar{3\frac{a}{c}-2}2\pi z} + \exp\Lpar{\Lpar{3\frac{a}{c}-1}2\pi z}}
{2\cosh(3\pi z)}. 
\end{align*}
We note that the $f(z)$ has poles at $z = z_n = -i(\tfrac{1}{6} + \tfrac{n}{3})$,
where $n\in\mathbb{Z}$. We find
\begin{center}
\begin{tabular}{c c c c}
pole & residue of $f(z)$ \\
$z_{3n-1} = -i(-\tfrac{1}{6}+n)$  &  $\frac{(-1)^{n} \sin\Lpar{\frac{\pi a}{c}(-6n+1)}}{\pi i \sqrt{3}}$ \\
$z_{3n} = -i(\tfrac{1}{6}+n)$  &  $\frac{(-1)^{n+1} \sin\Lpar{\frac{\pi a}{c}(6n+1)}}{\pi i \sqrt{3}}$ \\
$z_{3n+1} = -i(\tfrac{1}{2}+n)$  &  $0$
\end{tabular}
\end{center}
Applying the Mittag-Leffler Theory \cite[pp.134--135]{Wa-Wh-book}, we have            
\begin{align*}
f(z) &= 2 + 
{\sum_{n\in\mathbb{Z}}}^{*} \left(
\frac{i\,(-1)^n \sin\Lpar{\frac{\pi a}{c}(6n+1)}}{\pi \sqrt{3}} 
\Lpar{ \frac{1}{z + i(n + \tfrac{1}{6})} - \frac{1}{i(n + \tfrac{1}{6})}} \right.\\
&{\hskip1in} + 
\left.
\frac{(-i)\,(-1)^n \sin\Lpar{\frac{\pi a}{c}(-6n+1)}}{\pi \sqrt{3}} 
\Lpar{ \frac{1}{z + i(n - \tfrac{1}{6})} - \frac{1}{i(n - \tfrac{1}{6})}} \right)
\end{align*}
for $z\not\in(\pm\tfrac{1}{6} + \mathbb{Z})$, and assuming 
$\frac{1}{6} < \frac{a}{c} < \frac{5}{6}$.
Here we assume that ${\sum_{n\in\mathbb{Z}}}^{*} = \lim_{N\to\infty} \sum_{n=-N}^N$. 
We note that the convergence is uniform on any compact subset of 
$\mathbb{C} \setminus  -i(\pm\tfrac{1}{6}+\mathbb{Z})$.
We must consider three cases.

\subsubsection*{Case 1.1. $\frac{1}{6} < \frac{a}{c} < \frac{5}{6}$}
%% pp.51-58 CJS
We have
\begin{align*}
f(z) &= 2 + 
\frac{(-i)}{\pi \sqrt{3}}
{\sum_{n\in\mathbb{Z}}}^{*} 
(-1)^n \sin\Lpar{\frac{\pi a}{c}(6n+1)} 
\left(
\Lpar{ \frac{1}{z - i(n + \tfrac{1}{6})} + \frac{1}{i(n + \tfrac{1}{6})}} \right.\\
&{\hskip1in} + 
\left.
\Lpar{ \frac{1}{-z - i(n + \tfrac{1}{6})} + \frac{1}{i(n + \tfrac{1}{6})}} \right).
\end{align*}
Thus
\begin{align*}
&\int_0^\infty e^{-3\pi t x^2}\,f(x)\,dx\\
&= \int_0^\infty 2 e^{-3\pi t x^2}\, dx \\
&\qquad - \frac{i}{\pi \sqrt{3}} \int_{-\infty}^\infty e^{-3\pi tx^2}
\Lpar{
{\sum_{n\in\mathbb{Z}}}^{*} (-1)^n 
\sin\Lpar{\frac{\pi a}{c}(6n+1)} 
\Lpar{ \frac{1}{x - i(n + \tfrac{1}{6})} + \frac{1}{i(n + \tfrac{1}{6})}}}\,dx
\end{align*}
By absolute convergence on $\mathbb{R}$ we have
\begin{align*}
&\int_0^\infty e^{-3\pi t x^2}\,f(x)\,dx\\
&= \int_{-\infty}^\infty e^{-3\pi t x^2}\, dx \\
&\qquad - \frac{i}{\pi \sqrt{3}} 
{\sum_{n\in\mathbb{Z}}}^{*} (-1)^n 
\sin\Lpar{\frac{\pi a}{c}(6n+1)} 
\int_{-\infty}^\infty e^{-3\pi tx^2}
\Lpar{
\frac{1}{x - i(n + \tfrac{1}{6})} + \frac{1}{i(n + \tfrac{1}{6})}}\,dx\\
&=- \frac{i}{\pi \sqrt{3}} 
{\sum_{n\in\mathbb{Z}}}^{*} (-1)^n 
\sin\Lpar{\frac{\pi a}{c}(6n+1)} 
\int_{-\infty}^\infty \frac{e^{-3\pi tx^2}}{x - i(n + \tfrac{1}{6})}\,dx,
\end{align*}
since
\beq
{\sum_{n\in\mathbb{Z}}}^{*} (-1)^n 
\frac{\sin\Lpar{\frac{\pi a}{c}(6n+1)}}{i(n+\tfrac{1}{6})} =  \pi \sqrt{3}
\qquad \mbox{for $\frac{1}{6} < \frac{a}{c} < \frac{5}{6}$.}
\mylabel{eq:Fourierid}
\eeq
We leave \eqn{Fourierid} as an exercise for the reader. It can be proved using
\cite[Lemma 1.19, p.19]{Zw-thesis}. 
%%The proof of \eqn{Jint1} is completed
%%by proceeding as in the proof of \cite[Lemma 3.3]{Zw00} by using the identity 
By using the identity  \cite[Eq.(3.8), p.274]{Zw00} 
\beq
\int_{-\infty}^\infty \frac{e^{-\pi t x^2}}{x - is}\,dx
= \pi i s\, 
\int_0^\infty \frac{e^{-\pi u s^2}}{\sqrt{u+ t}}\,du\qquad
\mbox{for $s\in\mathbb{R}\setminus\{0\}$,}
\mylabel{eq:Zwid}
\eeq
we find that
$$
\Jac{a}{c}{\frac{2\pi}{t}} = \frac{t}{6\sqrt{3}} \Zsum
(-1)^n\, (6n+1) \, 
\sin\Lpar{\frac{\pi a}{c}(6n+1)}                                    
\,
\int_0^\infty \frac{e^{-\pi u (n+\tfrac{1}{6})^2 }}{\sqrt{u + 3t}}\,du.
$$
Letting $u=-3i\tau$ in the integral we find
$$
\Jac{a}{c}{\frac{2\pi}{t}} = \frac{-it}{6} \Zsum
(-1)^n\, (6n+1) \, 
\sin\Lpar{\frac{\pi a}{c}(6n+1)}                                    
\,
\int_0^{i\infty} \frac{e^{3\pi i\tau (n+\tfrac{1}{6})^2 }}{\sqrt{-i(\tau+ it)}}\,d\tau.
$$
Arguing as in the proof of  
\cite[Lemma 3.3]{Zw00} we may interchange summation and integration to obtain
$$
\Jac{a}{c}{\frac{2\pi i}{z}} = -\frac{z}{6} 
\,
\int_0^{i\infty} \Zsum (-1)^n\, (6n+1) \, \sin\Lpar{\frac{\pi a}{c}(6n+1)}\,e^{3\pi i\tau (n+\tfrac{1}{6})^2 }
\,\frac{d\tau}{\sqrt{-i(\tau+ z)}},
$$
when $z=it$, for $t>0$. Equation \eqn{Jint1} follows in this case. 

\subsubsection*{Case 1.2. $0 < \frac{a}{c} < \frac{1}{6}$} 
Observe that in this case
%%p.130 CJS
$$
-4\pi x < \Lpar{3 \frac{a}{c} - 2}2\pi x < -3\pi x,\qquad\mbox{(for $x>0$)},
$$
and the Mittag-Leffler Theory does not directly apply. We simply note that
$$
\frac{(1+ e^{2y})}{\cosh y} = 2 e^y,
$$
and find that
$$                     
\int_{-\infty}^\infty e^{-3\pi tx^2} \frac
{\exp\Lpar{\Lpar{3\frac{a}{c}-2}2\pi x}}{\cosh(3\pi x)}
\Lpar{1 + e^{6\pi x}}\,dx 
=
\frac{2}{\sqrt{3t}}\,\exp\Lpar{\frac{\pi}{12 t}\Lpar{6\frac{a}{c}-1}^2}.
$$
Thus we have
\begin{align*}
&\int_0^\infty e^{-3\pi tx^2} \, 
\frac{\cosh\Lpar{\Lpar{3\frac{a}{c}-2}2\pi x} + \cosh\Lpar{\Lpar{3\frac{a}{c}-1}2\pi x}}
{\cosh(3\pi x)} \, dx 
= 
\int_{-\infty}^\infty e^{-3\pi tx^2} \, g_{11}(x) \,dx\\
&=
\frac{1}{\sqrt{3t}}\,\exp\Lpar{\frac{\pi}{12 t}\Lpar{6\frac{a}{c}-1}^2} 
+ 
\int_{-\infty}^\infty e^{-3\pi tx^2} \, g_{12}(x)\,dx,
\end{align*}
where
\begin{align*}
g_{12}(z) &= 
\frac{\exp\Lpar{\Lpar{3\frac{a}{c}-1}2\pi z} - \exp\Lpar{\Lpar{3\frac{a}{c}+1}2\pi z}}
{2\cosh(3\pi z)}.               
\end{align*}
We observe that the function $g_{11}(z)$ from Case 1.1 and the function $g_{12}(z)$ have the same poles and the same residue
at each pole. We note we may apply the Mittag-Leffler Theory to the function $g_{12}(z)$ since
in this case
$$
-2\pi x < \Lpar{3 \frac{a}{c} - 1}2\pi x < -\pi x,\qquad
 2\pi x < \Lpar{3 \frac{a}{c} + 1}2\pi x < 3\pi x,
\qquad\mbox{(for $x>0$)}.
$$
%%
%% z1:=r->(3*r-1)*2;
%% r -> 6 r - 2
%% z2:=r->(3*r+1)*2;
%% r -> 6 r + 2
%% z1(0),z1(1/6);
%%                                    -2, -1
%% z2(0),z2(1/6);
%%                                     2, 3
%% 
The remainder of the proof is analogous to Case 1.1.

\subsubsection*{Case 1.3. $\frac{5}{6} < \frac{a}{c} < 1$} 
The proof is analogous to Case 1.2.
%%p.130 CJS
This time we have 
$$                     
\int_{-\infty}^\infty e^{-3\pi tx^2} \frac
{\exp\Lpar{\Lpar{3\frac{a}{c}-1}2\pi x}}{\cosh(3\pi x)}
\Lpar{1 + e^{-6\pi x}}\,dx 
=
\frac{2}{\sqrt{3t}}\,\exp\Lpar{\frac{\pi}{12 t}\Lpar{6\frac{a}{c}-5}^2}.
$$
Thus we have
\begin{align*}
&\int_0^\infty e^{-3\pi tx^2} \, 
\frac{\cosh\Lpar{\Lpar{3\frac{a}{c}-2}2\pi x} + \cosh\Lpar{\Lpar{3\frac{a}{c}-1}2\pi x}}
{\cosh(3\pi x)} \, dx 
= 
\int_{-\infty}^\infty e^{-3\pi tx^2} \, g_{11}(x) \,dx\\
&=
\frac{1}{\sqrt{3t}}\,\exp\Lpar{\frac{\pi}{12 t}\Lpar{6\frac{a}{c}-5}^2} 
+ 
\int_{-\infty}^\infty e^{-3\pi tx^2} \, g_{13}(x)\,dx,
\end{align*}
where
\begin{align*}
g_{13}(z) &= 
\frac{\exp\Lpar{\Lpar{3\frac{a}{c}-2}2\pi z} - \exp\Lpar{\Lpar{3\frac{a}{c}-4}2\pi z}}
{2\cosh(3\pi z)}.               
\end{align*}
We observe that the function $g_{11}(z)$ from Case 1.1 and the function $g_{13}(z)$ have the same poles and the same residue
at each pole, and we may apply the Mittag-Leffler Theory to the function $g_{13}(z)$ since
in this case
$$
\pi x < \Lpar{3 \frac{a}{c} - 2}2\pi x < 2\pi x,\qquad
 -3\pi x < \Lpar{3 \frac{a}{c} - 4}2\pi x < -2\pi x,
\qquad\mbox{(for $x>0$)}.
$$
%%
%% z1:=r->(3*r-2)*2;
%% r -> 6 r - 4
%% z2:=r->(3*r-4)*2;
%% r -> 6 r - 8
%% z1(5/6),z1(1);
%%                                     1, 2
%% z2(5/6),z2(1);
%%                                     -3, -2
%% 

Equation \eqn{Jint2} follows easily from equations \eqn{THAtrans2} and \eqn{Jint1}.

%%pp.75-? CJS
The proof of \eqn{Jint3} is analogous to that of \eqn{Jint1}. 
This time we assume $0 \le a < c$, $0 < b < c$  are integers  and $a/c\not\in\{1/6,5/6\}$.
Again by analytic continuation we may assume $z=it$ and $t>0$. 
We find that
$$
\Jabc{a}{b}{c}{\frac{2\pi}{t}} =
 t\,\int_{-\infty}^\infty e^{-3\pi tx^2}\,g_{21}(x)\,dx,
$$
where
$$                 
g_{21}(z) = 
\frac{\left( \zeta_c^b e^{-2\pi z} + \zeta_c^{2b} e^{-4\pi z}\right)}
{\cosh\Lpar{3\pi z - 3\pi i \tfrac{b}{c}}}
\exp\Lpar{6\pi z \tfrac{a}{c}}.
$$
We note that the $g_{21}(z)$ has poles at $z = z_n = -i(\tfrac{1}{6} + \tfrac{n}{3} - \tfrac{b}{c})$,
where $n\in\mathbb{Z}$. We find that
$$
\underset{z=z_n}{\mbox{Res}}\,g_{21}(z) = 
\frac{2(-1)^{n+1}}{3\pi}\, \sin\Lpar{\tfrac{\pi}{3}(2n+1)} \,
\exp\Lpar{\frac{\pi ia}{c^2}(6b -c -2nc)}.
$$
Applying the Mittag-Leffler Theory \cite[pp.134--135]{Wa-Wh-book}, we have            
\begin{align*}
g_{21}(z) &= g_{21}(0) + 
{\sum_{n\in\mathbb{Z}}}^{*} 
\frac{2(-1)^{n+1}}{3\pi}\, \sin\Lpar{\tfrac{\pi}{3}(2n+1)} \,
\exp\Lpar{\frac{\pi ia}{c^2}(6b -c -2nc)}\\
&\hskip2in\times
\Lpar{ \frac{1}{z + i(\tfrac{n}{3} + \tfrac{b}{c} - \tfrac{1}{6})} - 
\frac{1}{i(\tfrac{n}{3} + \tfrac{b}{c} - \tfrac{1}{6})}} 
\end{align*}
for $z\ne z_n$, $n\in\mathbb{Z}$, and assuming 
$\frac{1}{6} < \frac{a}{c} < \frac{5}{6}$.
We note that the convergence is uniform on any compact subset of 
$\mathbb{C} \setminus  \{z_n\,:\,n\in\mathbb{Z}\}$.            
Again we must consider 3 cases.
\subsubsection*{Case 2.1. $\frac{1}{6} < \frac{a}{c} < \frac{5}{6}$}
Proceeding as in Case 1.1, using the analog of \eqn{Fourierid}
\beq
{\sum_{n\in\mathbb{Z}}}^{*} 
\frac{2(-1)^{n+1}}{3\pi}\, \sin\Lpar{\tfrac{\pi}{3}(2n+1)} \,
\exp\Lpar{\frac{\pi ia}{c^2}(6b -c -2nc)} \,
\frac{1}{i(\tfrac{n}{3} + \tfrac{b}{c} - \tfrac{1}{6})} 
= g_{21}(0),
\mylabel{eq:Fourierid2}
\eeq
applying \eqn{Zwid} and using \eqn{Theta1abcid} we find that
\begin{align*}
&\Jabc{a}{b}{c}{\frac{2\pi i}{z}} \\
&= -\frac{2iz}{\sqrt{3}} \,\exp\Lpar{\frac{6\pi i ab}{c^2} - \frac{\pi i a}{c}}\\
&\quad \times                     
\int_0^{i\infty} \Zsum (-1)^n\, \Lpar{\tfrac{n}{3} + \tfrac{1}{6} - \tfrac{b}{c}} \, 
\sin\Lpar{\frac{\pi}{3}(2n+1)}\,\exp\Lpar{-2\pi in\frac{a}{c}} \, 
e^{3\pi i\tau (\tfrac{n}{3}+\tfrac{1}{6} -\tfrac{b}{c})^2 }
\,\frac{d\tau}{\sqrt{-i(\tau+ z)}},\\
&=-\frac{2iz}{6\sqrt{3}c}\,\int_0^{i\infty} \frac{\Theta_1(a,b,c;\tau)}{\sqrt{-i(\tau + z)}}\,d\tau ,
\end{align*}
which gives \eqn{Jint3} for this case.

\subsubsection*{Case 2.2. $0 < \frac{a}{c} < \frac{1}{6}$}
We proceed as in Case 1.2. This time we need 
$$
\int_{-\infty}^\infty \zeta_c^{2b} e^{-3\pi tx^2 + 6\pi x \tfrac{a}{c} - 4\pi x}
\frac{(1 + \exp\Lpar{6\pi x - 6\pi i \tfrac{b}{c}})}
{\cosh\Lpar{3\pi x - 3 \pi i \tfrac{b}{c}}}\,dx
=
\frac{2}{\sqrt{3t}} \, \zeta_{2c}^b\,
\exp\Lpar{\frac{\pi}{12t}\Lpar{6\tfrac{a}{c}-1}^2}.
$$
We have
$$
\Jabc{a}{b}{c}{\frac{2\pi}{t}}
= 
\frac{2}{\sqrt{3}} \, \sqrt{t}\, \zeta_{2c}^b\,
\exp\Lpar{\frac{\pi}{12t}\Lpar{6\tfrac{a}{c}-1}^2}
+ t\,\int_{-\infty}^\infty e^{-3\pi t x^2}\, g_{22}(x)\,dx,
$$
where
$$
g_{22}(z) = 
\Lpar{ \frac{ \zeta_c^b e^{-2\pi z} - \zeta_c^{-b} e^{2\pi z}}
{\cosh\Lpar{3\pi z - 3 \pi i \tfrac{b}{c}}}}\,e^{6\pi z \tfrac{a}{c}}.
$$
We observe that the function $g_{21}(z)$ from Case 2.1 and the function $g_{22}(z)$ have the same poles and the 
same residue at each pole. The result follows.

\subsubsection*{Case 2.3. $\frac{5}{6} < \frac{a}{c} < 1$} 
We proceed as in Case 2.2. This time we need 
$$
\int_{-\infty}^\infty \zeta_c^{b} e^{-3\pi tx^2 + 6\pi x \tfrac{a}{c} - 2\pi x}
\frac{(1 + \exp\Lpar{-6\pi x + 6\pi i \tfrac{b}{c}})}
{\cosh\Lpar{3\pi x - 3 \pi i \tfrac{b}{c}}}\,dx
=
\frac{2}{\sqrt{3t}} \, \zeta_{2c}^{5b}\,
\exp\Lpar{\frac{\pi}{12t}\Lpar{6\tfrac{a}{c}-5}^2}.
$$
We have
$$
\Jabc{a}{b}{c}{\frac{2\pi}{t}}
= 
\frac{2}{\sqrt{3}} \, \sqrt{t}\, \zeta_{2c}^{5b}\,
\exp\Lpar{\frac{\pi}{12t}\Lpar{6\tfrac{a}{c}-5}^2}
+ t\,\int_{-\infty}^\infty e^{-3\pi t x^2}\, g_{23}(x)\,dx,
$$
where
$$
g_{23}(z) = 
\Lpar{ \frac{ -\zeta_c^{4b} e^{-8\pi z} + \zeta_c^{2b} e^{-4\pi z}}
{\cosh\Lpar{3\pi z - 3 \pi i \tfrac{b}{c}}}}\,e^{6\pi z \tfrac{a}{c}}.
$$
We observe that the function $g_{21}(z)$ from Case 2.1 and the function $g_{23}(z)$ have the same poles and the 
same residue at each pole. The result follows.

Finally \eqn{Jint4} follows easily from \eqn{tt2} and \eqn{Jint3}.

\end{proof}

%SECTION 3%%%%%%%%%%%%%%%%%%%%%%%%%%%%%%%%%%%%%%%%%%%%%%%%%%%%%%%%%%%%%%%%
\section{Vector valued Maass forms of weight $1/2$}
\mylabel{sec:VMFS}
We describe Bringmann and Ono's vector valued Maass forms of weight $\tfrac{1}{2}$
making all functions and transformations explicit.
Suppose $0\le a < c$ and $0< b < c$ are integers where $(c,6)=1$.
We define
\begin{align}
T_1\Lpar{\frac{a}{c};z} &:= -\frac{i}{\sqrt{3}}\int_{-\zcon}^{i\infty}
\frac{\THA{a}{c}{\tau}}{\sqrt{-i(\tau + z)}}\,d\tau, 
\mylabel{eq:T1def}
\\
T_2\Lpar{\frac{a}{c};z} &:= 
\frac{i}{3c}\, 
\int_{-\zcon}^{i\infty} \frac{\Theta_1(0,-a,c;\tau)}{\sqrt{-i(\tau + z)}}\,d\tau,
\mylabel{eq:T2def}
\\
T_1(a,b,c,;z) 
&:=
\frac{\zeta_{2c}^{-5b}}{3c}\,
\int_{-\zcon}^{i\infty} \frac{\Theta_1(a,b,c;\tau)}{\sqrt{-i(\tau + z)}}\,d\tau,
\mylabel{eq:T1abcdef}
\\
T_2(a,b,c;z) 
&:=
\zeta_{2c}^{-5b}\,\frac{i\sqrt{3}}{6c}\,
\int_{-\zcon}^{i\infty} \frac{\Theta_2(a,b,c;\tau)}{\sqrt{-i(\tau + z)}}\,d\tau 
\mylabel{eq:T2abcdef}
\end{align}
We can now define a family of 
vector valued Maass forms of weight $\tfrac{1}{2}$.
\begin{align}
\mathcal{G}_1\Lpar{\frac{a}{c};z} &:= \Nell{a}{c}{z} - T_1\Lpar{\frac{a}{c};z},
\mylabel{eq:G1acdef}\\
\mathcal{G}_2\Lpar{\frac{a}{c};z} &:= \Mell{a}{c}{z} + \varepsilon_2\Lpar{\frac{a}{c};z}
- T_2\Lpar{\frac{a}{c};z},
\mylabel{eq:G2acdef}\\
\mathcal{G}_1(a,b,c;z) &:= \mathcal{N}(a,b,c,z) - T_1(a,b,c;z),
\mylabel{eq:G1abcdef}\\
\mathcal{G}_2(a,b,c;z) &:= \mathcal{M}(a,b,c,z) + \varepsilon_2(a,b,c;z) 
- T_2(a,b,c;z).
\mylabel{eq:G2abcdef}
\end{align}
We have corrected the definitions of 
$\mathcal{G}_2\Lpar{\frac{a}{c};z}$, and $\mathcal{G}_2(a,b,c;z)$ given on \cite[p.440]{Br-On10}. 
\begin{theorem}
\mylabel{thm:Gtrans} 
Suppose $0 \le a < c$ and $0< b < c$ are integers and assume $(c,6)=1$.
\begin{enumerate}
\item[(1)]
For $z\in\mathfrak{h}$ we have
\begin{align}
\mathcal{G}_1\Lpar{\frac{a}{c};z+1} &= \zeta_{24}^{-1}\,\mathcal{G}_1\Lpar{\frac{a}{c};z}
\mylabel{eq:G1trans1}\\
\mathcal{G}_2\Lpar{\frac{a}{c};z+1} &= \zeta_{2c}^{5a} \, \zeta_{2c^2}^{-3a^2}\,\zeta_{24}^{-1}\,
\mathcal{G}_2(a,a,c;z),
\mylabel{eq:G2trans1}\\
\mathcal{G}_1(a,b,c;z+1)&=
\begin{cases}
\zeta_{2c^2}^{3b^2} \, \zeta_{24}^{-1} \,\mathcal{G}_1(a-b,b,c;z) & \mbox{if $a\ge b$}, \\
-\zeta_{2c^2}^{3b^2} \, \zeta_{c}^{-3b}\,\zeta_{24}^{-1}\, \mathcal{G}_1(a-b+c,b,c;z) & \mbox{otherwise},
\end{cases}
\mylabel{eq:GG1trans1}\\
\mathcal{G}_2(a,b,c;z+1) &=
\zeta_{2c}^{5a} \, \zeta_{2c^2}^{-3a^2}\,\zeta_{24}^{-1}\,
\begin{cases}
\mathcal{G}_2(a,a+b,c;z) & \mbox{if $a+b<c$},\\
\mathcal{G}_2\Lpar{\frac{a}{c};z} & \mbox{if $a+b=c$},\\
\mathcal{G}_2(a,a+b-c,c;z) & \mbox{otherwise},
\end{cases}
\mylabel{eq:GG2trans1}
\end{align}
where $a$ is assumed to be positive in the first and second formula.
\item[(2)]
For $z\in\mathfrak{h}$ we have
\begin{align}
\fsqiz\,\mathcal{G}_1\Lpar{\frac{a}{c};\Sz} &= \mathcal{G}_2\Lpar{\frac{a}{c};z},\mylabel{eq:G1trans2}\\
\fsqiz\,\mathcal{G}_2\Lpar{\frac{a}{c};\Sz} &= \mathcal{G}_1\Lpar{\frac{a}{c};z},\mylabel{eq:G2trans2}\\
\fsqiz\,\mathcal{G}_1\Lpar{a,b,c;\Sz}&=\mathcal{G}_2(a,b,c;z),\mylabel{eq:GG1trans2}\\
\fsqiz\,\mathcal{G}_2\Lpar{a,b,c;\Sz} &= \mathcal{G}_1(a,b,c;z),
\mylabel{eq:GG2trans2}
\end{align}
where again $a$ is assumed to be positive in the first and second formula.
\end{enumerate}
\end{theorem}
\begin{proof}
From \eqn{THAtrans1} we have
\begin{align*}
T_1\Lpar{\frac{a}{c};z+1} &= -\frac{i}{\sqrt{3}}\int_{-\zcon-1}^{i\infty}
\frac{\THA{a}{c}{\tau}}{\sqrt{-i(\tau + z+1)}}\,d\tau 
                          = -\frac{i}{\sqrt{3}}\int_{-\zcon}^{i\infty}
\frac{\THA{a}{c}{\tau-1}}{\sqrt{-i(\tau + z)}}\,d\tau \\
                          &= -\frac{i}{\sqrt{3}}\int_{-\zcon}^{i\infty}
\frac{\zeta_{24}^{-1}\THA{a}{c}{\tau}}{\sqrt{-i(\tau + z)}}\,d\tau 
= \zeta_{24}^{-1} \, T_1\Lpar{\frac{a}{c};z+1}.
\end{align*}
Hence by \eqn{Ntrans1}, \eqn{G1acdef} we have \eqn{G1trans1}.
The proofs of \eqn{G2trans1}--\eqn{GG2trans1} are similar.

We now prove \eqn{G1trans2}.
$$
\fsqiz\,\mathcal{G}_1\Lpar{\frac{a}{c};\Sz} = 
\fsqiz\, \Nell{a}{c}{\Sz}  + \frac{i}{\sqrt{3}\sqrt{-iz}}\,
\int_{-(\overline{-1/z})}^{i\infty}
\frac{\THA{a}{c}{\tau}}{\sqrt{-i(\tau - 1/z)}}\,d\tau.             
$$
$$                     
\int_{-(\overline{-1/z})}^{i\infty}
\frac{\THA{a}{c}{\tau}}{\sqrt{-i(\tau - 1/z)}}\,d\tau 
=
- \int_{0}^{-\zcon} 
\frac{\THA{a}{c}{-1/\tau}}{\sqrt{-i(-1/\tau - 1/z)}}\,\frac{d\tau }{\tau^2}
=
\sqrt{-iz}\,\int_{0}^{-\zcon} 
\frac{(-i\tau)^{-3/2}\,\THA{a}{c}{-1/\tau}}{\sqrt{-i(\tau + z)}}\,d\tau.
$$                
Thus by \eqn{Ntrans3}, \eqn{Jint2}, \eqn{THAtrans2} we have
\begin{align*}
&\fsqiz\,\mathcal{G}_1\Lpar{\frac{a}{c};\Sz} \\
&= 
\Mell{a}{c}{z} + 2\sqrt{3}\sqiz\,\Jac{a}{c}{-2\pi iz} +                            
\frac{i}{\sqrt{3}}\,\int_{0}^{-\zcon} 
\frac{(-i\tau)^{-3/2}\,\THA{a}{c}{-1/\tau}}{\sqrt{-i(\tau + z)}}\,d\tau\\
&= 
\Mell{a}{c}{z} + \varepsilon_2\Lpar{\frac{a}{c};z}
-\frac{i}{3c}\, 
\int_0^{i\infty} \frac{\Theta_1(0,-a,c;\tau)}{\sqrt{-i(\tau + z)}}\,d\tau + 
\frac{i}{3c}\,
\int_0^{-\zcon} \frac{\Theta_1(0,-a,c;\tau)}{\sqrt{-i(\tau + z)}}\,d\tau\\
&= 
\Mell{a}{c}{z} + \varepsilon_2\Lpar{\frac{a}{c};z}
-\frac{i}{3c}\, 
\int_{-\zcon}^{i\infty} \frac{\Theta_1(0,-a,c;\tau)}{\sqrt{-i(\tau + z)}}\,d\tau\\
&= 
\mathcal{G}_2\Lpar{\frac{a}{c};z},
\end{align*}
and we have \eqn{G1trans2}. Equation \eqn{G2trans2} follows immediately from \eqn{G1trans2}.
The proofs of \eqn{GG1trans2}--\eqn{GG2trans2} are analogous.
\end{proof}
\begin{cor}
\mylabel{cor:Vc}
Suppose $c$ is a fixed positive integer relatively prime to $6$. Then
\begin{align*}
\mathfrak{V}_c&:= \left\{\mathcal{G}_1\Lpar{\frac{a}{c};z},\,
             \mathcal{G}_2\Lpar{\frac{a}{c};z}\,:\, 0< a < c\right\}\\
&\qquad \cup \left\{\mathcal{G}_1(a,b,c;z),\, \mathcal{G}_2(a,b,c;z)\,:\,
0\le a < c\quad\mbox{and}\quad 0<b<c\right\}
\end{align*}
is a vector valued Maass form of weight $\tfrac{1}{2}$ for the full modular
group $\SL_2(\mathbb{Z})$.
\end{cor}

%SECTION 4%%%%%%%%%%%%%%%%%%%%%%%%%%%%%%%%%%%%%%%%%%%%%%%%%%%%%%%%%%%%%%%%
\section{A Maass form multiplier}
\mylabel{sec:MFM}

We will find that transformation formulas are more tractable if we modify the
definition of the functions $\mathcal{G}_j$ by multiplying by the Dedekind
eta-function $\eta(z)$. 
For a function $F(z)$, and a weight $k$ we define
the usual stroke operator
\beq
 \stroke{F}{A}{k} := (ad-bc)^{k/2}\,(cz+d)^{-k} \, F\left(A\,z\right),\quad\mbox{for}\quad
A=\AMat\in\GL_2^{+}(\mathbb{Z}),
\mylabel{eq:strokedef}
\eeq
where $k\in\frac{1}{2}\mathbb{Z}$, and when calculating $(cz+d)^{-k}$ we take the principal
value.
Our main result is
\begin{theorem}
\mylabel{thm:mainthm}
Let $p>3$ be prime, suppose $1\le\ell\le(p-1)$, and define
\begin{equation}
\Fell{1}{\ell}{p}{z} := \eta(z) \, \Gell{1}{\ell}{p}{z}.
\mylabel{eq:F1def}
\end{equation}
Then
\beq
\stroke{\Fell{1}{\ell}{p}{z}}{A}{1} = \mu(A,\ell)\,
\Fell{1}{\overline{d\ell}}{p}{z},
\mylabel{eq:Fmult}
\eeq
where
$$
\mu(A,\ell) = \exp\left(\frac{3\pi i c d \ell^2}{p^2}\right) \,
(-1)^{\frac{c\ell}{p}} \, (-1)^{\FL{\frac{d\ell}{p}}},
$$
and
$$
A = \AMat \in \Gamma_0(p).
$$
Here $\overline{m}$ is the least nonnegative residue of $m\pmod{p}$.
\end{theorem}
\begin{remark}
The function $\mu(A,\ell)$ is reminiscent of functions that occur in transformation
formulas of certain theta-functions \cite[Lemma 2.1]{Bi89} on $\Gamma_0(p)$.
\end{remark}
\begin{cor}
\mylabel{cor:Ftrans}
Let $p>3$ be prime and suppose $1\le \ell\le\frac{1}{2}(p-1)$. Then         
\beq
\stroke{\Fell{1}{\ell}{p}{z}}{A}{1} = \Fell{1}{\ell}{p}{z},
\mylabel{eq:Ftrans}
\eeq
and
\beq
\stroke{\Gell{1}{\ell}{p}{z}}{A}{\tfrac{1}{2}} = \frac{1}{\nu_\eta(A)}\,\Gell{1}{\ell}{p}{z},
\mylabel{eq:Gtrans}
\eeq
where $A \in \Gamma_0(p^2) \cap \Gamma_1(p)$ and $\nu_\eta(A)$ is the eta-multiplier
\beq
\stroke{\eta(z)}{A}{\tfrac{1}{2}} = \nu(A) \,\eta(z).
\mylabel{eq:etamult}
\eeq
\end{cor}
\begin{remark}
Equation \ref{eq:Gtrans} strengthens one of the main results 
\cite[Theorem 1.2, p.424]{Br-On10}
of Bringmann and Ono's paper in the case that $c=p>3$ is prime.
\end{remark}
\begin{proof}
Let 
$$
A = \AMat \in \Gamma_0(p^2) \cap \Gamma_1(p).
$$
Then
$c\equiv0\pmod{p^2}$ and $a\equiv d\equiv1\pmod{p}$. So
$$
\mu(A,\ell) =\exp\left(\frac{\pi i}{p^2}( 3cd\ell^2 + c\ell p + p^2\FL{\frac{d\ell}{p}})\right)=1,
$$
since
$$
 3cd\ell^2 + c\ell p + p^2\FL{\frac{d\ell}{p}}       \equiv 0 \pmod{p^2},
$$
and
$$
d\ell=p\FL{\frac{d\ell}{p}} +\ell\equiv\FL{\frac{d\ell}{p}}+\ell\pmod{2},
$$
so that 
%% \begin{align*}
%%  &3cd\ell^2 + c\ell p + p^2\FL{\frac{d\ell}{p}} \\
%%  &\qquad\equiv cd\ell + c\ell + d\ell + \ell \pmod{2}\\
%%  &\qquad\equiv \ell(c+1)(d+1) \equiv 0\pmod{2}
%% \end{align*}
\beqs
3cd\ell^2 + c\ell p + p^2\FL{\frac{d\ell}{p}} 
 \equiv cd\ell + c\ell + d\ell + \ell 
 \equiv \ell(c+1)(d+1) \equiv 0\pmod{2}
\eeqs
since either $c$ or $d$ is odd. Thus (\ref{eq:Ftrans}) follows from (\ref{eq:Fmult})
and (\ref{eq:Gtrans}) is immediate.
\end{proof}

\begin{cor}
\mylabel{cor:F2trans}
Let $p>3$ be prime, suppose $1\le\ell\le(p-1)$, and define
\begin{equation}
\Fell{2}{\ell}{p}{z} := \eta(z) \, \Gell{2}{\ell}{p}{z}.
\mylabel{eq:F2def}
\end{equation}
Then         
\beq
\stroke{\Fell{2}{\ell}{p}{p^2\,z}}{A}{1} = \Fell{2}{\ell}{p}{p^2\,z},
\mylabel{eq:F2trans}
\eeq
for $A \in \Gamma_0(p^2) \cap \Gamma_1(p)$.
\end{cor}
\begin{remark}
We prove Corollary \corol{F2trans} in section \subsect{pfcorF2}
\end{remark}
%%\begin{proof}
%%Let 
%%$$
%%A = \AMat \in \Gamma_0(p^2) \cap \Gamma_1(p),\quad
%%S = \SMat.
%%$$
%%From Theorem \thm{Ftrans} we have
%%$$
%%\Fell{2}{\ell}{p}{z} = -i \,\stroke{\Fell{1}{\ell}{p}{z}}{S}{1}.
%%$$
%%We let
%%$$
%%P = \begin{pmatrix} p^2 & 0 \\ 0 & 1\end{pmatrix},
%%$$
%%and find that
%%$$
%%S\, P\, A = B \, S\, P,
%%$$
%%where
%%$$
%%B = \begin{pmatrix} d & - c/p^2 \\ -p^2b & a\end{pmatrix}
%%\in \Gamma_0(p^2) \cap \Gamma_1(p).
%%$$
%%Therefore
%%\begin{align*}
%%\stroke{\Lpar{\stroke{\Fell{2}{\ell}{p}{z}}{P}{1}}}{A}{1} 
%%%%\Fell{2}{\ell}{p}{p^2\,A\,z} 
%%& = -i \, \stroke{\Fell{1}{\ell}{p}{z}}{SPA}{1}\\
%%& = -i \, \stroke{\Fell{1}{\ell}{p}{z}}{BSP}{1}\\
%%& = -i \, \stroke{\Fell{1}{\ell}{p}{z}}{SP}{1}
%%\qquad \mbox{(by \eqn{Ftrans})}\\
%%& = -\stroke{\Fell{2}{\ell}{p}{z}}{P}{1},
%%\end{align*}
%%and we have \eqn{F2trans}.
%%\end{proof}

\subsection{Proof of Theorem \ref{thm:mainthm}}
\mylabel{subsec:pfmainthm}

It is well-known that the matrices
$$
S=\SMat,\qquad T=\TMat.
$$
generate $\SL_2(\mathbb{Z})$, and
\beq
\stroke{\eta(z)}{T}{\tfrac{1}{2}} = \zeta_{24} \, \eta(z),\qquad
\stroke{\eta(z)}{S}{\tfrac{1}{2}} = \exp\Lpar{-\tfrac{\pi i}{4}}\, \eta(z).
\mylabel{eq:etatrans}
\eeq
We need
\begin{theorem}[Rademacher\cite{Ra29}]
Let $p$ be prime. Then a set of generators for $\Gamma_0(p)$ is given
by
$$
T,\qquad V_k \quad(1\le k \le p-1),
$$
where
$$
V_k = S T^k S T ^{-\ks} S^{-1} = 
\begin{pmatrix}
-\ks & - 1 \\
k\ks + 1 & k
\end{pmatrix}
$$
and $\ks$ is given by $1\le \ks \le p-1$ and $k\ks \equiv -1\pmod{p}$.
Furthermore for $p>3$ the number of generators can be reduced to $2\FL{\frac{p}{12}}+3$.
\end{theorem}

As in \eqn{F1def}, \eqn{F2def}
we define functions $\Fell{j}{}{}{}$ by multiplying $\Gell{j}{}{}{}$ by
$\eta(z)$:
\begin{align}
%%\mathcal{F}_1\Lpar{\frac{a}{p};z} &:=  \eta(z) \, \mathcal{G}_1\Lpar{\frac{a}{c};z} 
%%\mylabel{eq:F1acdef}\\
%%\mathcal{F}_2\Lpar{\frac{a}{p};z} &:=  \eta(z) \, \mathcal{G}_2\Lpar{\frac{a}{c};z} 
%%\mylabel{eq:F2acdef}\\
\mathcal{F}_1(a,b,c;z) &:=  \eta(z) \, \mathcal{G}_1(a,b,c;z) 
\mylabel{eq:F1abcdef}\\
\mathcal{F}_2(a,b,c;z) &:=  \eta(z) \, \mathcal{G}_2(a,b,c;z).
\mylabel{eq:F2abcdef}
\end{align}
The following follows from Theorem \ref{thm:Gtrans} and (\ref{eq:etatrans}).
\begin{theorem}
\mylabel{thm:Ftrans} 
Suppose $0 \le a < c$ and $0< b < c$ are integers and assume $(c,6)=1$.
\begin{enumerate}
\item[(1)]
For $z\in\mathfrak{h}$ we have
\begin{align*}
\stroke{\mathcal{F}_1\Lpar{\frac{a}{c};z}}{T}{1} &= \mathcal{F}_1\Lpar{\frac{a}{c};z},\\
\stroke{\mathcal{F}_2\Lpar{\frac{a}{c};z}}{T}{1} &= \zeta_{2c}^{5a} \, \zeta_{2c^2}^{-3a^2}\,
\mathcal{F}_2(a,a,c;z),\\
\stroke{\mathcal{F}_1(a,b,c;z)}{T}{1}&=
\begin{cases}
\zeta_{2c^2}^{3b^2} \, \mathcal{F}_1(a-b,b,c;z) & \mbox{if $a\ge b$},\\
-\zeta_{2c^2}^{3b^2} \, \zeta_{c}^{-3b}\, \mathcal{F}_1(a-b+c,b,c;z) & \mbox{otherwise},
\end{cases}\\
\stroke{\mathcal{F}_2(a,b,c;z)}{T}{1} &=
\zeta_{2c}^{5a} \, \zeta_{2c^2}^{-3a^2}\,
\begin{cases}
\mathcal{F}_2(a,a+b,c;z) & \mbox{if $a+b<c$},\\
\mathcal{F}_2\Lpar{\frac{a}{c};z} & \mbox{if $a+b=c$},\\
\mathcal{F}_2(a,a+b-c,c;z) & \mbox{otherwise},
\end{cases}
\end{align*}
where $a$ is assumed to be positive in the first and second formula.
\item[(2)]
For $z\in\mathfrak{h}$ we have
\begin{align*}
\stroke{\mathcal{F}_1\Lpar{\frac{a}{c};z}}{S}{1} &= -i\,\mathcal{F}_2\Lpar{\frac{a}{c};z},\\
\stroke{\mathcal{F}_2\Lpar{\frac{a}{c};z}}{S}{1} &= -i\,\mathcal{F}_1\Lpar{\frac{a}{c};z},\\
\stroke{\mathcal{F}_1\Lpar{a,b,c;z}}{S}{1}&=-i\,\mathcal{F}_2(a,b,c;z),\\
\stroke{\mathcal{F}_2\Lpar{a,b,c;z}}{S}{1} &= -i\,\mathcal{F}_1(a,b,c;z),
\end{align*}
where again $a$ is assumed to be positive in the first and second formula.
\end{enumerate}
\end{theorem}
Throughout this section we assume $p>3$ is prime and $1\le \ell \le p-1$.
Since $(ST)^3 = -I$ we have
\beq
\mathcal{F}_2(0,\ell,p;z) = i\,\zeta_{2p}^{-5\ell}\,\Fell{1}{\ell}{p}{z}
\mylabel{eq:F20id}
\eeq
using Theorem \ref{thm:Ftrans}. 
%%Recall that
%%$$
%%(AB : z) = (A : Bz)\,(B:z),
%%$$
%%for $A$, $B\in \SL_2(\mathbb{Z})$. 
We require the transformation
$$
\stroke{\mathcal{F}_1(a,b,c;z)}{T^{-1}}{1}=
\begin{cases}
\zeta_{2c^2}^{-3b^2} \, \mathcal{F}_1(a+b,b,c;z) & \mbox{if $a+b< c$},\\
-\zeta_{2c^2}^{-3b^2} \, \zeta_{p}^{3b}\, \mathcal{F}_1(a+b-c,b,c;z) & \mbox{otherwise}.
\end{cases}
$$
which follows from Theorem \ref{thm:Ftrans} assuming $0\le a,b<c$.
        
Our first goal is to show that Theorem \ref{thm:mainthm} holds when $A$ is a
generator of $\Gamma_0(p)$. The result is clearly true when $A=T$. We assume
$1\le k \le p-1$. Applying Theorem \thm{Ftrans} we have
$$
\stroke{\Fell{1}{\ell}{p}{z}}{S}{1} = (-i)\, \Fell{2}{\ell}{p}{z},
$$
$$
\stroke{\Fell{2}{\ell}{p}{z}}{T^k}{1}
= \zeta_{2p}^{5k\ell} \, \zeta_{2p^2}^{-3\ell^2k} \, 
\mathcal{F}_{2}(\ell,\overline{k\ell},p;z),
$$
$$
\stroke{\mathcal{F}_{2}(\ell,\overline{k\ell},p;z)}{S}{1} = -i\, 
\mathcal{F}_{1}(\ell,\overline{k\ell},p;z), 
$$
$$
\stroke{\mathcal{F}_{1}(\ell,\overline{k\ell},p;z)}{T^{-\ks}}{1}
= \left(\zeta_{2p^2}^{-3(\overline{k\ell})^2}\right)^{\ks}\,
\left(-\zeta_p^{3\overline{k\ell}}\right)^j\,\mathcal{F}_1(0,\overline{k\ell},p;z),
$$
where 
$$
j = \frac{\ell + \overline{k\ell}\ks}{p},
$$
and
$$
\stroke{\mathcal{F}_1(0,\overline{k\ell},p;z)}{S^{-1}}{1} = i \,
\mathcal{F}_2(0,\overline{k\ell},p;z).                                      
$$
Putting all this together we find that
$$
\stroke{\Fell{1}{\ell}{p}{z}}{V_k}{1}
= -i\,\zeta_{2p^2}^{
-3(\ell^2k+\ks(\overline{k\ell})^2)}\,
\zeta_{2p}^{5k\ell + 6j\overline{k\ell}} \,(-1)^j \,
\mathcal{F}_2(0,\overline{k\ell},p;z).
$$
Using \eqn{F20id} we have
$$
\stroke{\Fell{1}{\ell}{p}{z}}{V_k}{1} 
= \mutwid(V_k,\ell) 
\, \Fell{1}{\overline{k\ell}}{p}{z},                           
$$
where
$$
\mutwid(V_k,\ell) = 
\zeta_{2p^2}^{
-3(\ell^2k+\ks(\overline{k\ell})^2)}\,
\zeta_{2p}^{5(k\ell-\overline{k\ell}) + 6j\overline{k\ell}} \,(-1)^j,
$$
and
$$
j = \frac{\ell + \overline{k\ell}\ks}{p}.
$$

Next we must show that
$$
\mutwid(V_k,\ell) = \mu(V_k,\ell).
$$
This is equivalent to showing that
\begin{align}
&5p(k\ell-\overline{k\ell}) - 3(\ell^2k + \ks (\overline{k\ell})^2) + p(\ell + \ks\overline{k\ell})
+6\overline{k\ell}(\ell+\overline{k\ell}\ks)
\mylabel{eq:congy}\\
&\qquad \equiv
3(1+k\ks)k\ell^2 + p(1+k\ks)\ell + p^2\FL{\frac{k\ell}{p}} \pmod{2p^2}.
\nonumber             
\end{align}
This congruence holds mod $p^2$ since
$$
\ks(k\ell - \overline{k\ell})^2 \equiv 0 \pmod{p^2}.
$$
It remains to show that the congruence \eqn{congy} holds mod $2$.
Mod $2$ the congruence reduces to
$$
\overline{k\ell} + \ell \equiv k\ell + \ell + \FL{\frac{k\ell}{p}} \pmod{2},
$$
which is true since when $x$ is an integer and $p$ is a positive odd integer 
\beq
x = p \FL{\frac{x}{p}} + \overline{x},\quad\mbox{and}\quad
\FL{\frac{x}{p}} \equiv x + \overline{x} \pmod{2}.
\mylabel{eq:xpcong}
\eeq
Thus we have shown that Theorem \thm{mainthm}
holds when $A$ is one of Rademacher's generators for $\Gamma_0(p)$.

In the next part of the proof we show that Theorem \thm{mainthm} holds when $A$ is the inverse
of any of Rademacher's generators.  The result clearly holds for $A=T^{-1}$. Let
$1\le k \le p-1$. We find that
$$
V_k^{-1} 
=
\begin{pmatrix}
k &  1 \\
-k\ks - 1 & -\ks
\end{pmatrix}
=
-\begin{pmatrix}
 -k &  -1 \\
 k\ks + 1 & \ks
 \end{pmatrix}
= - V_{\ks}.
$$
We know that
$$
\stroke{\Fell{1}{\ell}{p}{z}}{V_{\ks}}{1} = \mu(V_{\ks},\ell)\,
\Fell{1}{\overline{\ks\ell}}{p}{z}.
$$
Hence
\begin{align*}
&\stroke{\Fell{1}{\ell}{p}{z}}{V_k^{-1}}{1} = 
\stroke{\Fell{1}{\ell}{p}{z}}{-V_{\ks}}{1} \\
&\quad =-\stroke{\Fell{1}{\ell}{p}{z} }{V_{\ks}}{1}
=-\mu(V_{\ks},\ell)\,
\Fell{1}{\overline{\ks\ell}}{p}{z}\\
&\quad=-\mu(V_{\ks},\ell)\,
\Fell{1}{\overline{-\ks\ell}}{p}{z}\\
&\quad= 
\exp\left(\frac{3\pi i (1+k\ks) \ks \ell^2}{p^2}\right) \,
(-1)^{\frac{(1+k\ks)\ell}{p}} \, (-1)^{\FL{\frac{\ks\ell}{p}}+1}\,
\Fell{1}{\overline{-\ks\ell}}{p}{z}\\
&\quad= 
\exp\left(\frac{3\pi i (1+k\ks) \ks \ell^2}{p^2}\right) \,
(-1)^{\frac{(1+k\ks)\ell}{p}} \, (-1)^{\FL{\frac{-\ks\ell}{p}}}\,
\Fell{1}{\overline{-\ks\ell}}{p}{z}\\
&\hskip3in
\mbox{(since $\FL{x}+1=-\FL{-x}$ when $x\not\in\mathbb{Z}$)}
\\
&\quad = \mu(V_{k}^{-1},z) \,
\Fell{1}{\overline{-\ks\ell}}{p}{z},
\end{align*}
and Theorem \thm{mainthm} holds for $A=V_{k}^{-1}$.

Finally we need to show that if Theorem \thm{mainthm} holds for
$$
A=\AMat,\qquad B=\BMat,
$$
then it also holds for 
$$
AB = \begin{pmatrix}
aa' + bc' & ab' + bd' \\
ca' + dc' & cb' + dd'
\end{pmatrix}.
$$
We must show that
$$  
\stroke{\Fell{1}{\ell}{p}{z}}{AB}{1} = \mu(AB,\ell)\,
\Fell{1}{\overline{(cb'+dd')\ell}}{p}{z}.
$$
Now
\begin{align*}
& \stroke{\Fell{1}{\ell}{p}{z}}{AB}{1} 
= \stroke{\left(\stroke{\Fell{1}{\ell}{p}{z}}{A}{1}\right)}{B}{1}\\
&\quad = \mu(A,\ell)\,\stroke{\Fell{1}{\overline{d\ell}}{p}{z}}{B}{1}
\qquad\mbox{(since Theorem \thm{mainthm} holds for the matrix $A$)}\\
&\quad = \mu(B,\overline{d\ell})\,\mu(A,\ell)\,\Fell{1}{\overline{dd'\ell}}{p}{z}
= \mu(A,\ell)\,\mu(B,\overline{d\ell})\,\Fell{1}{\overline{(cb'+dd')\ell}}{p}{z},
\end{align*}
since $c\equiv0\pmod{p}$ and Theorem \thm{mainthm} holds for the matrix $B$.         
This mean we need only verify that
$$
\mu(A,\ell)\,\mu(B,\overline{d\ell}) = \mu(AB, \ell).
$$
This is equivalent to showing that
\begin{align}
&3(ca' + dc')(cb' + dd')\ell^2 + p\ell(ca' + dc') + p^2\FL{\frac{(cb'+dd')\ell}{p}} 
\mylabel{eq:multcong}
\\
&\equiv
3cd\ell^2 + pc\ell + p^2\FL{\frac{d\ell}{p}} +
3c'd'(\overline{d\ell})^2 + pc'\overline{d\ell} + p^2\FL{\frac{d'\overline{d\ell}}{p}} 
\pmod{2p^2}.
\nonumber
\end{align}
This can be easily verified mod $p^2$ using the congruences $c\equiv c'\equiv 0\pmod{p}$
and $a'd'\equiv 1\pmod{p}$. It remains to verify that the congruence holds mod $2$; i.e.
\begin{align*}
&(ca' + dc')(cb' + dd')\ell + \ell(ca' + dc') +  cb'\ell +\FL{\frac{dd'\ell}{p}} 
\\
&\equiv
cd\ell + c\ell + \FL{\frac{d\ell}{p}} +
c'd'(\overline{d\ell}) + c'\overline{d\ell} + \FL{\frac{d'\overline{d\ell}}{p}} 
\pmod{2}.
\end{align*}
Using \eqn{xpcong} we see that this is equivalent to showing
\begin{align*}
&(ca' + dc')(cb' + dd'+1)\ell +  cb'\ell +dd'\ell +\overline{dd'\ell} 
\\
&\equiv
cd\ell + c\ell + d\ell + \overline{d\ell}
+ c'd'(\overline{d\ell}) + c'\overline{d\ell} + d'\overline{d\ell} + \overline{d'd\ell} 
\pmod{2},
\end{align*}
or
\begin{align*}
&((ca' + dc')(cb' + dd'+1) +  cb' +dd')\ell
\\
&\equiv
((c+1)(d+1)-1)\ell + (c'+1)(d'+1)\overline{d\ell}\pmod{2},
\end{align*}
or
$$
(ca' + dc')(cb' + dd'+1) +  cb' +dd' \equiv 1 \pmod{2},
$$
since at least one of $c$, $d$ is odd, and at least one of $c'$, $d'$ is odd.
But
\begin{align*}
&(ca' + dc')(cb' + dd'+1) +  cb' +dd' \\
&\equiv c(a'b'+a'+b') + d(c'd'+c'+d') + dc (a'd'+c'b') \pmod{2}\\
&\equiv c + d + dc \equiv 1 \pmod{2},
\end{align*}
since at least one of $a'$, $b'$ is odd and $a'b'+a'+b\equiv 1 \pmod{2}$,
similarly $c'd'+c'+d'\equiv1\pmod{2}$, $a'd'+c'b'\equiv1\pmod{2}$, and
at least one $c$, $d$ is odd.  Thus \eqn{multcong} holds mod $2$, mod $p^2$ and
hence mod $2p^2$.

We have shown that Theorem \thm{mainthm} holds for generators of $\Gamma_0(p)$, inverses
of generators, products of generators and hence for all matrices in $\Gamma_0(p)$,
which completes the proof.

\subsection{Proof of Corollary \corol{F2trans}}   
\mylabel{subsec:pfcorF2}
Let 
$$
A = \AMat \in \Gamma_0(p^2) \cap \Gamma_1(p),\quad
\mbox{and recall}\quad
S = \SMat.
$$
From Theorem \thm{Ftrans} we have
$$
\Fell{2}{\ell}{p}{z} = -i \,\stroke{\Fell{1}{\ell}{p}{z}}{S}{1}.
$$
We let
$$
P = \begin{pmatrix} p^2 & 0 \\ 0 & 1\end{pmatrix},
$$
and find that
$$
S\, P\, A = B \, S\, P,
$$
where
$$
B = \begin{pmatrix} d & - c/p^2 \\ -p^2b & a\end{pmatrix}
\in \Gamma_0(p^2) \cap \Gamma_1(p).
$$
Therefore
\begin{align*}
\stroke{\Lpar{\stroke{\Fell{2}{\ell}{p}{z}}{P}{1}}}{A}{1} 
%%\Fell{2}{\ell}{p}{p^2\,A\,z} 
& = -i \, \stroke{\Fell{1}{\ell}{p}{z}}{SPA}{1}\\
& = -i \, \stroke{\Fell{1}{\ell}{p}{z}}{BSP}{1}\\
& = -i \, \stroke{\Fell{1}{\ell}{p}{z}}{SP}{1}
\qquad \mbox{(by \eqn{Ftrans})}\\
& = -\stroke{\Fell{2}{\ell}{p}{z}}{P}{1},
\end{align*}
and we have \eqn{F2trans}.

%SECTION 5%%%%%%%%%%%%%%%%%%%%%%%%%%%%%%%%%%%%%%%%%%%%%%%%%%%%%%%%%%%%%%%%
\section{Extending Ramanujan's Dyson rank function identity}
\mylabel{sec:EXTRAM}

Equation \eqn{Ramid5} is Ramanujan's identity for the $5$-dissection of
$R(\zeta_5,q)$. In equation \eqn{Ramid5V2} we showed how this identity
could be written in terms of generalized eta-functions. 
In this section we show that there is an analogous 
result for the $p$-dissection of $R(\zeta_p,q)$ when $p$ is any prime 
greater than $3$. 

We assume $p>3$ is prime, and define
\begin{align}      
&\Jpz{z}
\mylabel{eq:Jdef}\\
&= \eta(p^2 z) \,
\left(\Nell{1}{p}{z} - 
2\,\chi_{12}(p) \,
\sum_{\ell=1}^{\frac{1}{2}(p-1)} (-1)^{\ell}
\sin\Lpar{\frac{6\ell\pi}{p}} \,
\left(
\Mell{\ell}{p}{p^2z} + \varepsilon_2\Lpar{\frac{\ell}{p};p^2z}
\right)\right),
\nonumber          
\end{align}
where $\chi_{12}(n)$ is defined in \eqn{chi12}.
%%\beq
%%\chi(n) = \leg{12}{n} = 
%%\begin{cases}
%%1 & \mbox{if $n\equiv\pm1\pmod{12}$,}\\
%%-1 & \mbox{if $n\equiv\pm5\pmod{12}$,}\\
%%0 & \mbox{otherwise.}
%%\end{cases}
%%\mylabel{eq:chi12}
%%\eeq
%%Our main result is
Using \eqn{Macid}, \eqn{RNacid} we find that
$$
\eta(p^2 z) \, \Rpz{z} 
=
\sin\Lpar{\frac{\pi}{p}} \, \Jpz{z}.
$$
%%HERE
Thus we rewrite one of our main results, Theorem \thm{mainJp0}, in 
the  equivalent form:
\begin{theorem}
\mylabel{thm:mainJp}
Let $p > 3$ be prime. Then the function $\Jpz{z}$, defined in \eqn{Jdef},
is a weakly holomorphic modular form of weight $1$ on the group
$\Gamma_0(p^2) \cap \Gamma_1(p)$.
\end{theorem}

This theorem leads to our analogue of Ramanujan's identity
\eqn{Ramid5} or \eqn{Ramid5V2}.
%%For $0 < a \le \frac{1}{2}(p-1)$ we define
%%\beq
%%\Phi_{p,a}(q) :=
%%\begin{cases}
%%\displaystyle
%%\frac{q^a}{(q^p;q^p)_\infty} \sum_{n=-\infty}^\infty 
%%                  \frac{(-1)^n q^{\frac{3}{2}pn(n+1)}}
%%                       {1-q^{pn+a}} + 1 &
%%\mbox{if $0 < 6a < p$},\\ 
%%\displaystyle
%%\frac{q^a}{(q^p;q^p)_\infty} \sum_{n=-\infty}^\infty 
%%                  \frac{(-1)^n q^{\frac{3}{2}pn(n+1)}}
%%                       {1-q^{pn+a}}  &
%%\mbox{if $p < 6a < 3p$}.
%%\end{cases}
%%\eeq
\begin{cor}
\mylabel{cor:Rpz}
Let $p > 3$ be prime. Then the function
\begin{align}      
\Rpz{z} 
&= q^{-\frac{1}{24}} R(\zeta_p,q) 
\mylabel{eq:Rpdefv2}\\
&\quad 
- 
4\,q^{-\frac{1}{24}}\,\chi_{12}(p) \,
\sum_{a=1}^{\frac{1}{2}(p-1)} (-1)^a 
\sin\Lpar{\frac{\pi}{p}} \,
\sin\Lpar{\frac{6a\pi}{p}} \,
q^{\frac{a}{2}(p-3a)-\frac{1}{24}(p^2-1)}\,\Phi_{p,a}(q^p)
\nonumber               
\end{align}      
is a modular form of weight $\frac{1}{2}$ on $\Gamma_0(p^2) \cap \Gamma_1(p)$
with multiplier. In particular
%%$$
%%{(A:z)}^{-\tfrac{1}{2}}\,\Rpz{A\,z} 
%%= \frac{1}{\nu_\eta(A)}\,\Rpz{z},   
%%$$
%%HERE
\beq
\stroke{\mathcal{R}_p(z)}{A}{\tfrac{1}{2}} = 
\frac{1}{\nu_\eta({}^{p^2}A)} \, \Rpz{z},
\mylabel{eq:Rptrans}
\eeq
for $A=\AMat\in\Gamma_0(p^2) \cap \Gamma_1(p)$, and where
$$
{}^{p^2}A = \AppMat.
$$
%%$\nu_\eta(A)$ is the eta-multiplier.
\end{cor}
\begin{remark}
Equation \eqn{Rpdefv2} follows easily
from the definition of $\Rpz{z}$, which is given in \eqn{Rpdef}. 
Equation \eqn{Rptrans} follows from Theorem \thm{mainJp} 
and \eqn{etamult}.
\end{remark}

This result will follow from Corollary \corol{Ftrans} and
\begin{prop}
\mylabel{propo:theta1id}
Let $p > 3$ be prime. Then 
\beq
\THA{1}{p}{z}
=
-\frac{2}{\sqrt{3}} \,\chi_{12}(p) \,
\sum_{a=1}^{\frac{1}{2}(p-1)} (-1)^a 
\sin\Lpar{\frac{6a\pi}{p}} \,\Theta_1(0,-a,p;p^2z).
\mylabel{eq:theta1id}
\eeq
\end{prop}
\begin{proof}
From \eqn{Theta1acdef} we recall
$$                      
\THA{1}{p}{z} 
= \sum_{n=-\infty}^\infty (-1)^n (6n+1) \sin\Lpar{\frac{\pi (6n+1)}{p}}
\exp\Lpar{3\pi iz\Lpar{n + \frac{1}{6}}^2}.
$$                     
We assume $p >3$ is prime and consider two cases.
\subsubsection*{Case 1} $p\equiv 1 \pmod{6}$. We let $p_1 = \frac{1}{6}(p-1)$
so that $6p_1+1=p$. We note that each integer $n$ satisfying $6n+1\not\equiv0
\pmod{p}$ can be written uniquely as
\begin{align*}
&(i) \qquad n=p(2pm+\ell_1) + a  + p_1, \quad
     \mbox{where $1\le a \le \frac{1}{2}(p-1)$, $0 \le \ell_1 < 2p$, and 
           $m \in \mathbb{Z}$,}\\
&\quad\mbox{or}\\
&(ii) \qquad n=p(-2pm-\ell_1) - a  + p_1, \quad
     \mbox{where $1\le a \le \frac{1}{2}(p-1)$, $1 \le \ell_1 \le 2p$, and 
           $m \in \mathbb{Z}$.}
\end{align*}
If $n=p(2pm+\ell_1) + a  + p_1$, then
\begin{align*}
6n+1 &= 12p^2m + 2p\ell + 6a+p,\qquad\mbox{where $\ell=3\ell_1$ and 
$0\le \ell < 6p$},\\
\sin\Lpar{\frac{\pi (6n+1)}{p}} &= - \sin\Lpar{\frac{6a\pi}{p}},\qquad
\sin\Lpar{\frac{\pi}{3}(2\ell+1)} = \frac{\sqrt{3}}{2}, \quad
(-1)^n = (-1)^{\ell + a + p_1}.
\end{align*}
If $n=p(-2pm-\ell_1) - a  + p_1$, then
\begin{align*}
6n+1 &= -(12p^2m + 2p\ell + 6a+p),\qquad\mbox{where $\ell=3\ell_1-1$ and 
$0<\ell \le 6p-1$},\\
\sin\Lpar{\frac{\pi (6n+1)}{p}} &= \sin\Lpar{\frac{6a\pi}{p}},\qquad
\sin\Lpar{\frac{\pi}{3}(2\ell+1)} = -\frac{\sqrt{3}}{2},\qquad
(-1)^n = -(-1)^{\ell + a + p_1}.
\end{align*}
Hence we have
\begin{align*}
&\THA{1}{p}{z} 
= \sum_{\substack{n=-\infty \\ 6n+1\not\equiv 0 \pmod{p}}}^\infty
                   (-1)^n (6n+1) \sin\Lpar{\frac{\pi (6n+1)}{p}}
\exp\Lpar{3\pi iz\Lpar{n + \frac{1}{6}}^2}\\
&=
-\sum_{a=1}^{\frac{1}{2}(p-1)}
\sum_{\ell=0}^{6p-1}
(-1)^{\ell + a + p_1} \, \sin\Lpar{\frac{\pi}{3}(2\ell+1)} \, \frac{2}{\sqrt{3}}
\, \sin\Lpar{\frac{6a\pi}{p}} \\
&\qquad\qquad \times
\sum_{m=-\infty}^\infty (12p^2m + 2\ell p + 6a + p) \,
\exp\Lpar{\frac{\pi iz}{12}(12p^2m + 2\ell p + 6a+p)^2}\\
&=
-\frac{2}{\sqrt{3}} \,\chi_{12}(p) \,
\sum_{a=1}^{\frac{1}{2}(p-1)} (-1)^a 
\sin\Lpar{\frac{6a\pi}{p}} \,\Theta_1(0,-a,p;p^2z),
\end{align*}
since $(-1)^{p_1} = \chi_{12}(p)$.

\subsubsection*{Case 2} $p\equiv -1 \pmod{6}$. We proceed as in Case 1
except this time we let  $p_1 = \frac{1}{6}(p+1)$
so that $6p_1-1=p$, and we find  that
each integer $n$ satisfying $6n+1\not\equiv0 \pmod{p}$ can be written uniquely as
\begin{align*}
&(i) \qquad n=p(2pm+\ell_1) + a  - p_1, \quad
     \mbox{where $1\le a \le \frac{1}{2}(p-1)$, $1 \le \ell_1 \le 2p$, and 
           $m \in \mathbb{Z}$,}\\
&\quad\mbox{or}\\
&(ii) \qquad n=p(-2pm-\ell_1) - a  - p_1, \quad
     \mbox{where $1\le a \le \frac{1}{2}(p-1)$, $0 \le \ell_1 < 2p$, and 
           $m \in \mathbb{Z}$.}
\end{align*}
The result \eqn{theta1id} follows as in Case 1.
\end{proof}

\subsection{Proof of Theorem \thm{mainJp} -- Part 1 -- Transformations}
\mylabel{subsec:pfmainJp1}

First we show that
\beq
\Jpz{z} = \JSpz{z},
\mylabel{eq:JJSid}
\eeq
where $\Jpz{z}$ is defined in \eqn{Jdef} and
\begin{equation}   
\JSpz{z} 
= \frac{\eta(p^2 z)}{\eta(z)} \, \Lpar{\Fell{1}{1}{p}{z}
 - 
2\,\chi_{12}(p) \,
\sum_{\ell=1}^{\frac{1}{2}(p-1)} (-1)^{\ell} 
\sin\Lpar{\frac{6\ell\pi}{p}} \,
\Fell{2}{\ell}{p}{p^2 z}}.
\mylabel{eq:JSdef}
\end{equation}   
From \eqn{T2def} we have
$$
T_2\Lpar{\frac{\ell}{p};p^2 z} = 
\frac{i}{3}\, 
\int_{-\zcon}^{i\infty} 
\frac{\Theta_1(0,-\ell,c;p^2\tau)}{\sqrt{-i(\tau + z)}}\,d\tau.
$$
Therefore using \eqn{T1def} and Proposition \propo{theta1id} we have
\begin{align*}
T_1\Lpar{\frac{1}{p};z} & = \frac{-i}{\sqrt{3}} 
\int_{-\zcon}^{i\infty}
\frac{\THA{1}{p}{\tau}}{\sqrt{-i(\tau + z)}}\,d\tau\\
&= \frac{2i}{3}\, \,\chi_{12}(p) \,
\sum_{\ell=1}^{\frac{1}{2}(p-1)} (-1)^{\ell} 
\sin\Lpar{\frac{6\ell\pi}{p}} \,
\int_{-\zcon}^{i\infty} 
\frac{\Theta_1(0,-\ell,c;p^2\tau)}{\sqrt{-i(\tau + z)}}\,d\tau,
\end{align*}
and
\beq
T_1\Lpar{\frac{1}{p};z}  =
2 \,\chi_{12}(p) \,
\sum_{\ell=1}^{\frac{1}{2}(p-1)} (-1)^{\ell}
\sin\Lpar{\frac{6\ell\pi}{p}} \,
T_2\Lpar{\frac{\ell}{p};p^2 z}.
\mylabel{eq:T1T2id}
\eeq
From \eqn{JSdef}, \eqn{G1acdef}, \eqn{G2acdef}
we have
\begin{align*}
&\JSpz{z}
= 
\eta(p^2 z) \, \left(\Gell{1}{1}{p}{z}
 - 
2\,\chi_{12}(p) \,
\sum_{\ell=1}^{\frac{1}{2}(p-1)} (-1)^{\ell}
\sin\Lpar{\frac{6\ell\pi}{p}} \,
\Gell{2}{\ell}{p}{p^2 z} \right)\\
&= 
\eta(p^2 z) \, \left(
\Nell{1}{p}{z} - T_1\Lpar{\frac{1}{p};z} \right.\\
&\qquad \left.- 2\, \chi_{12}(p) \,
\sum_{\ell=1}^{\frac{1}{2}(p-1)} (-1)^{\ell} 
\sin\Lpar{\frac{6\ell\pi}{p}} \,
\left(
\Mell{\ell}{p}{p^2z} + \varepsilon_2\Lpar{\frac{\ell}{p};p^2z}
               - T_2\Lpar{\frac{\ell}{p};p^2z} 
\right)\right)\\
&= \Jpz{z},
\end{align*}
by \eqn{T1T2id}.  By Corollaries \corol{Ftrans}, \corol{F2trans}, and equation
\eqn{JSdef} we have
$$
%%\frac{1}{(A:z)}\,\JSpz{A\,z} = \JSpz{z},
\strokevb{\JSpz{z}}{A}{1} = \JSpz{z},
$$
for $A\in\Gamma_0(p^2) \cap \Gamma_1(p)$, using 
the well-known result
that $\frac{\eta(p^2z)}{\eta(z)}$ is a modular function on $\Gamma_0(p^2)$ when
$p>3$ is prime. Hence
$$
%%\frac{1}{(A:z)}\,\Jpz{A\,z} = \Jpz{z},
\strokevb{\Jpz{z}}{A}{1} = \Jpz{z},
$$
for $A\in\Gamma_0(p^2) \cap \Gamma_1(p)$.               

\subsection{Proof of Theorem \thm{mainJp} -- Part 2 -- Checking Cusp Conditions}
\mylabel{subsec:pfmainJp2}

%%
%%For a function $F(z)$, and a weight $k$ we define
%%the usual stroke operator
%%$$
%% F\,\left|\,\left[A\right]_k\right. := (cz+d)^{-k} \, F\left(A\,z\right),\quad\mbox{for}\quad
%%A=\AMat\in\SL_2(\mathbb{Z}).
%%$$
We note that the function $\Jpz{z}$ is holomorphic on $\mathfrak{h}$.
%%by \eqn{JJSid}, \eqn{JSdef}.
In section \subsect{pfmainJp1} we showed that
$$
\stroke{\Jpz{z}}{A}{1} = \Jpz{z},\qquad\mbox{for all $A\in\Gamma_0(p^2) \cap \Gamma_1(p)$}.
$$
By \cite[p.125]{Kob} this implies that the function 
$\stroke{\Jpz{z}}{\gA}{1}(z)$ has period $p^2$ for each $\gA\in\SL_2(\mathbb{Z})$. 
Hence the function $\Jpz{z}$ has a Fourier expansion
in powers of $q_{p^2}=\exp(2\pi iz/p^2)$. To complete the proof of 
Theorem \thm{mainJp}
%%Theorem \thm{mainJp}(i)
we need to show that this expansion has only finitely many negative powers of 
$q_{p^2}$. %%The proof of Theorem \thm{Kpthm}(ii) is analogous.

We need
\begin{lemma}
\mylabel{lem:holo}
Suppose the functions
$$
f_1(z),\,f_2(z), \dots, f_n(z),
$$
are holomorphic on an open connected set $\mathfrak{D}$, and linearly independent over $\mathbb{C}$.
Suppose the functions
$$
G_1(z),\,G_2(z), \dots, G_n(z),
$$
are holomorphic on $-\overline{\mathfrak{D}}=\{-\overline{d}\,:\,d\in \mathfrak{D}\}$, and
$$
\sum_{j=1}^n f_j(z)\,G_j(-\zcon) = 0
$$
on $\mathfrak{D}$. Then 
$$
G_1(z)=G_2(z)= \cdots = G_n(z) = 0
$$
on $-\overline{\mathfrak{D}}$.
\end{lemma}
\begin{proof}
We proceed by induction on $n$. The result is clearly true for $n=1$.
Now suppose the result is true for $n=m$ where $m\ge1$ is fixed.
Suppose the functions
$$
f_1(z),\,f_2(z), \dots, f_{m+1}(z),
$$
are holomorphic on an open connected set $\mathfrak{D}$, and linearly independent over $\mathbb{C}$.
Suppose the functions
$$
G_1(z),\,G_2(z), \dots, G_{m+1}(z),
$$
are holomorphic on $-\overline{\mathfrak{D}}=\{-\overline{d}\,:\,d\in \mathfrak{D}\}$, and
$$
\sum_{j=1}^{m+1} f_j(z)\,G_j(-\zcon) = 0
$$
on $\mathfrak{D}$. Let
$$
\mathfrak{D}_1 = \mathfrak{D} \setminus \{z_0\,:\, f_{m+1}(z_0) = 0\}.
$$
Then $\mathfrak{D}_1$ is an open connected set and on $\mathfrak{D}_1$ the functions $F_j(z) = \frac{f_j(z)}{f_{m+1}(z)}$
are holomorphic, linearly independent over $\mathbb{C}$ and on $\mathfrak{D}_1$
\beq
\sum_{j=1}^{m} F_j(z) \, G_j(-\overline{z}) = - G_{m+1}(-\overline{z}).
\mylabel{eq:Gcomb}
\eeq
Since $G_{m+1}(z)$ is holomorphic,
$$
\frac{\partial}{\partial z}\, \sum_{j=1}^{m} F_j(z) \, G_j(-\overline{z}) = 
-\frac{\partial}{\partial z}\, G_{m+1}(-\overline{z}) = 0,
$$
and
\beq
\sum_{j=1}^{m} F_j'(z) \, G_j(-\overline{z}) = 0
\mylabel{eq:Fdcomb}
\eeq
on $\mathfrak{D}_1$. 
We next show that the $F_j'(z)$ are linearly independent over $\mathbb{C}$.
Suppose there are complex numbers $a_1$, $a_2$, \dots, $a_m$ such that
$$
a_1 F_1'(z) + a_2F_2'(z) + \cdots + a_m F_m'(z) = 0 
$$
on $\mathfrak{D}_1$. Then
$$
a_1 F_1(z) + a_2F_2(z) + \cdots + a_m F_m(z) = a_{m+1}
$$
for some constant $a_{m+1}$. But then
$$
a_1 f_1(z) + a_2f_2(z) + \cdots + a_m f_m(z) - a_{m+1}f_{m+1}(z) = 0,
$$
on $\mathfrak{D}_1$ and hence $\mathfrak{D}$. This implies
$$
a_1 = a_2 = \cdots = a_{m} = a_{m+1} = 0
$$
by the linear independence of the $f_j$.
Thus that the $F_j'(z)$ are linearly independent over $\mathbb{C}$.

This together with \eqn{Fdcomb} and the fact that the  $F_j'(z)$ are holomorphic on $\mathfrak{D}_1$, implies 
that
$$
G_1(z)=G_2(z)= \cdots = G_m(z) = 0
$$
on $\mathfrak{D}_1$, 
by the induction hypothesis. By \eqn{Gcomb} we have
$$
G_1(z)=G_2(z)= \cdots = G_m(z) = G_{m+1}(z) = 0
$$
on $\mathfrak{D}_1$ and hence on $\mathfrak{D}$, and the result is true for $n=m+1$ thus completing the induction
proof.
\end{proof}

We define                                    
\begin{align}
\mathfrak{W}_p&:= \left\{\mathcal{F}_1\Lpar{\frac{a}{p};z},\,
             \mathcal{F}_2\Lpar{\frac{a}{p};z}\,:\, 0< a < p\right\}
\mylabel{eq:Wpdef}\\
&\qquad \cup \left\{\mathcal{F}_1(a,b,p;z),\, \mathcal{F}_2(a,b,p;z)\,:\,
0\le a < p\quad\mbox{and}\quad 0<b<p\right\}.
\nonumber
\end{align}

Now let
$$
\gA = \AMat \in \SL_2(\mathbb{Z}),\quad\mbox{so that $\gA(\infty) = a/c$ is a cusp.}
$$
As mentioned above we must show that $\stroke{\Jpz{z}}{\gA}{1}$ expanded
as a series in $q_{p^2}$ has only finitely many terms with negative exponents.
We examine each of the functions
$$
\frac{\eta(p^2 z)}{\eta(z)} \, \Fell{1}{1}{p}{z}, \qquad
\Fell{2}{\ell}{p}{p^2 z} \qquad (1 \le \ell \le \tfrac{1}{2}(p-1)),
$$
which occur on the right side of \eqn{JSdef}.
By Theorem \thm{Ftrans} we have
$$
\stroke{\Fell{1}{1}{p}{z}}{\gA}{1} = \varepsilon_A \, \mathcal{F}_A(z),
$$
for some $\mathcal{F}_A \in \mathfrak{W}_p$ and some root of unity $\varepsilon_A$.
Thus
$$
\stroke{\frac{\eta(p^2 z)}{\eta(z)} \, \Fell{1}{1}{p}{z}}{\gA}{1} 
= S_{\gA}(z) + W_{\gA}(z) \,
\int_{-\overline{z}}^{i\infty}
\frac{g_{\gA}(\tau)}{\sqrt{-i\Lpar{\tau + z}}}\,d\tau,
$$
for some functions 
$S_{\gA}$, $W_{\gA}$, and $g_{\gA}$ holomorphic on $\mathfrak{h}$.
The function $S_{\gA}(z)$ is the product of a constant, 
the function
$$
\stroke{\frac{\eta(p^2 z)}{\eta(z)}}{\gA}{1},
$$
the function $\eta(z)$ and
one of the following
$$
\Nell{a}{p}{z},\,
\Mell{a}{p}{z} + \varepsilon_2\Lpar{\frac{a}{p};z},\,
\mathcal{N}(a,b,p,z)\,
\mathcal{M}(a,b,p,z) + \varepsilon_2(a,b,p;z).
$$
Using the fact that $\frac{\eta(p^2 z)}{\eta(z)}$ is a modular function
on $\Gamma_0(p^2)$ and by 
examining \eqn{Ndef}--\eqn{Nabcdef} we
find that
$S_{\gA}(z)$ has only finitely many terms with negative exponents
when expanded as a series in $q_{p^2}$.

We let $1\le \ell \le \tfrac{1}{2}(p-1)$ and we show an analogous result
holds for
$\stroke{\Fell{2}{\ell}{p}{z}}{\gA}{1}$.

\subsubsection*{Case 1} 
$c\not\equiv 0 \pmod{p}$.
From Theorem \thm{Ftrans} we have
$$
\Fell{2}{\ell}{p}{z}  = i\,
\stroke{\Fell{1}{\ell}{p}{z}}{S}{1}.
$$
We choose $b'$ so that $b' c \equiv d \pmod{p^2}$ and
$$
\PPMat \, \AMat = Y \, \begin{pmatrix}1 & b' \\ 0 & p^2\end{pmatrix},
$$
where
$$
S\, Y = \begin{pmatrix} -c & \frac{1}{p^2}(b'c - d) \\ p^2 a & b - b'a\end{pmatrix}
\in \Gamma_0(p^2).
$$
Hence by Theorem \thm{mainthm} we have
\begin{align*}
&\stroke{\Fell{2}{\ell}{p}{p^2z}}{\gA}{1}
= i\,p^{-1}\,
\stroke{\Fell{1}{\ell}{p}{z}}{S}{1}
\strokeb{\PPMat}{1} \strokeb{A}{1}\\
&= i\,p^{-1}\,
\stroke{\Fell{1}{\ell}{p}{z}}{SY}{1}
\strokeb{\BpMat}{1}\\
&= i\,p^{-1}\,\mu(SY,\ell)\,
\Fell{1}{\ell'}{p}{z}
\strokeb{\BpMat}{1}\\
&= i\,\mu(SY,\ell)\,
\Fell{1}{\ell'}{p}{\frac{z+b'}{p^2}},
\end{align*}
for some integer $\ell'$. 
Hence
$$
\stroke{\Fell{2}{\ell}{p}{p^2z}}{\gA}{1} 
= S_{\ell,\gA}(z) + W_{\ell,\gA}(z) \,
\int_{-\overline{\zbpp}}^{i\infty}
\frac{g_{\ell,\gA}(\tau)}{\sqrt{-i\Lpar{\tau + \zbpp}}}\,d\tau,
$$
for some functions 
$S_{\ell,\gA}$, $W_{\ell,\gA}$, and $g_{\ell,\gA}$ holomorphic on $\mathfrak{h}$.
By considering the transformation $\tau \mapsto \frac{\tau-b'}{p^2}$ we find that
$$
\stroke{\Fell{2}{\ell}{p}{p^2z}}{\gA}{1} 
= S_{\ell,\gA}(z) + W_{\ell,\gA}(z) \,
\int_{-\overline{z}}^{i\infty}
\frac{\gtwid_{\ell,\gA}(\tau)}{\sqrt{-i\Lpar{\tau + z}}}\,d\tau,
$$
for some holomorphic function $\gtwid_{\ell,\gA}$.
This time the function $S_{\ell,\gA}(z)$ is product of a constant and 
$$
\eta(z)\,\Nell{a}{p}{z},
$$
with $z$ replace by $\zbpp$ and thus has only finitely many 
many terms with negative exponents
when expanded as a series in $q_{p^2}$. In fact in this case the exponents
are nonnegative.

\subsubsection*{Case 2} 
$c\equiv 0 \pmod{p^2}$. Then
$$
\PPMat \, A = Y' \, \PPMat,
$$
where
$$
Y' = 
\begin{pmatrix} a & p^2b\\ cp^{-2} & d\end{pmatrix}\in\SLZ.
$$
By Theorem \thm{Ftrans} we have
\begin{align*}
&\stroke{\Fell{2}{\ell}{p}{p^2z}}{\gA}{1}
= p^{-1}\,
\stroke{\Fell{2}{\ell}{p}{z}}{\PPMat}{1}
\strokeb{A}{1}\\
&= p^{-1}\,
\stroke{\Fell{2}{\ell}{p}{z}}{Y}{1}
\strokeb{\PPMat}{1}\\
&= \varepsilon_{\ell,A} \, \mathcal{F}_{\ell,A}(p^2z),
\end{align*}
for some $\mathcal{F}_{\ell,A} \in \mathfrak{W}_p$ and some root of unity 
$\varepsilon_{\ell,A}$.
As in Case 1 we find that
$$
\stroke{\Fell{2}{\ell}{p}{p^2z}}{\gA}{1} 
= S_{\ell,\gA}(z) + W_{\ell,\gA}(z) \,
\int_{-\overline{z}}^{i\infty}
\frac{\gtwid_{\ell,\gA}(\tau)}{\sqrt{-i\Lpar{\tau + z}}}\,d\tau,
$$
for some functions 
$S_{\ell,\gA}$, $W_{\ell,\gA}$, and $\gtwid_{\ell,\gA}$ holomorphic on $\mathfrak{h}$,
and for which
$S_{\gA}(z)$ has only finitely many terms with negative exponents
when expanded as a series in $q_{p^2}$.

\subsubsection*{Case 3} 
$c\equiv 0 \pmod{p}$ and $c\not\equiv 0 \pmod{p^2}$. 
We choose $b''$ so that $b''c \equiv dp\pmod{p^2}$, and
$$
\PPMat \, \AMat = Y'' \, \BbpMat,
$$
where
$$
Y'' = \begin{pmatrix} pa & bp - a b''\\ c/p & \frac{1}{p^2}(dp-cb'')\end{pmatrix}
\in \SLZ.
$$
  
Again by Theorem \thm{Ftrans} it follows that
$$                   
\stroke{\Fell{2}{\ell}{p}{p^2z}}{\gA}{1}
= \varepsilon_{\ell,A} \, \mathcal{F}_{\ell,A}\left(z + \frac{b''}{p}\right).
$$
As in Case 2 we find that
$$
\stroke{\Fell{2}{\ell}{p}{p^2z}}{\gA}{1} 
= S_{\ell,\gA}(z) + W_{\ell,\gA}(z) \,
\int_{-\overline{z}}^{i\infty}
\frac{\gtwid_{\ell,\gA}(\tau)}{\sqrt{-i\Lpar{\tau + z}}}\,d\tau,
$$
for some functions 
$S_{\ell,\gA}$, $W_{\ell,\gA}$, and $\gtwid_{\ell,\gA}$ holomorphic on $\mathfrak{h}$,
and for which
$S_{\gA}(z)$ has only finitely many terms with negative exponents
when expanded as a series in $q_{p^2}$.

From \eqn{JJSid}, \eqn{JSdef} we have
$$
\stroke{\Jpz{z}}{A}{1} = 
\stroke{\JSpz{z}}{A}{1} = 
\sum_{\ell=0}^{\frac{1}{2}(p-1)}
\Stwid_{\ell,\gA}(z) + \sum_{\ell=0}^{\frac{1}{2}(p-1)} \Wtwid_{\ell,\gA}(z) \,
\int_{-\overline{z}}^{i\infty}
\frac{\gtwid_{\ell,\gA}(\tau)}{\sqrt{-i(\tau+z)}}\,d\tau
$$
where
$$
\Stwid_{0,\gA}(z) = S_A(z),\qquad \Wtwid_{0,A}(z) = W_A(z),
$$
$$
\Stwid_{\ell,\gA}(z) = 
2\,\chi_{12}(p) \,
(-1)^{\ell+1} 
\sin\Lpar{\frac{6\ell\pi}{p}} \,S_{\ell,A}(z),\qquad
\Wtwid_{\ell,\gA}(z) = 
2\,\chi_{12}(p) \,
(-1)^{\ell+1} 
\sin\Lpar{\frac{6\ell\pi}{p}} \,W_{\ell,A}(z),
$$
for $1 \le \ell \le \tfrac{1}{2}(p-1)$.
We claim that the sum
$$
\sum_{\ell=0}^{\frac{1}{2}(p-1)} \Wtwid_{\ell,\gA}(z) \,
\int_{-\overline{z}}^{i\infty}
\frac{\gtwid_{\ell,\gA}(\tau)}{\sqrt{-i(\tau+z)}}\,d\tau
$$
is identically zero. Hence we may suppose that not all the functions
$$
\Wtwid_{0,\gA}(z) \, \Wtwid_{1,\gA}(z) \, \dots, \Wtwid_{\frac{1}{2}(p-1),\gA}(z)
$$
are identically zero.  
We take a maximal linearly independent subset of them, say,
$$
\Wdtwid_{1,\gA}(z) \, \Wdtwid_{1,\gA}(z) \, \dots, \Wdtwid_{m,\gA}(z).
$$
Then for each $\ell$, $0\le \ell\le \tfrac{1}{2}(p-1)$, 
there exists constants $\beta_{j,\ell}$ such
that
$$
\Wtwid_{\ell,\gA}(z) = \sum_{j=1}^m \beta_{j,\ell} \Wdtwid_{j,\gA}(z),
$$
and we have
\begin{align*}
\sum_{\ell=0}^{\frac{1}{2}(p-1)} \Wtwid_{\ell,\gA}(z) \,
\int_{-\overline{z}}^{i\infty}
\frac{\gtwid_{\ell,\gA}(\tau)}{\sqrt{-i(\tau+z)}}\,d\tau
&=
\sum_{\ell=0}^{\frac{1}{2}(p-1)} 
\sum_{j=1}^m \beta_{j,k} \Wdtwid_{j,\gA}(z)\,
\int_{-\overline{z}}^{i\infty}
\frac{\gtwid_{\ell,\gA}(\tau)}{\sqrt{-i(\tau+z)}}\,d\tau
\\
&=
\sum_{j=1}^m 
\Wdtwid_{j,\gA}(z)\,
\int_{-\overline{z}}^{i\infty}
\sum_{\ell=0}^{\frac{1}{2}(p-1)} 
\beta_{j,\ell} \gtwid_{\ell,\gA}(\tau)
\frac{d\tau}{\sqrt{-i(\tau+z)}}
\\
&=
\sum_{j=1}^m 
\Wtwid_{j,\gA}(z)\,
\int_{-\overline{z}}^{i\infty}
\frac{g_{j}^{*}(\tau)\,d\tau}{\sqrt{-i(\tau+z)}},
\end{align*}
where
$$
g_j^{*}(\tau) = 
\sum_{\ell=0}^{\frac{1}{2}(p-1)} 
\beta_{j,\ell} \gtwid_{\ell,\gA}(\tau)
$$
is holomorphic on $\mathfrak{h}$. Applying $\frac{\partial}{\partial \zcon}$ to
$\stroke{\JSpz{z}}{\gA}{1}$ gives
\begin{align*}
0 &= \frac{\partial}{\partial \zcon} 
\sum_{j=1}^m 
\Wtwid_{j,\gA}(z)\,
\int_{-\overline{z}}^{i\infty}
\frac{g_{j}^{*}(\tau)\,d\tau}{\sqrt{-i(\tau+z)}},
\\
&=
\sum_{j=1}^m 
\frac{\Wdtwid_{j,\gA}(z)\, g_{j}^{*}(-\zcon)}{\sqrt{-i(z+\zcon)}},
\end{align*}
since $\JSpz{z}=\Jpz{z}$ and the 
$S_{\ell,\gA}(z)$ are holomorphic.
Thus
$$
\sum_{j=1}^m 
\Wtwid_{j,\gA}(z)\, g_{j}^{*}(-\zcon) = 0,
$$
and so all the $g_{j}^{*}$ are identically zero by Lemma \lem{holo}.
Hence
$$
\stroke{\Jpz{z}}{\gA}{1}(z)
= 
\sum_{\ell=0}^{\frac{1}{2}(p-1)} S_{\ell,\gA}(z).
\mylabel{eq:JpAz}
$$
Each of the functions $S_{\ell,\gA}(z)$ 
has only finitely many terms with negative exponents
when expanded as a series in $q_{p^2}$.
Thus $\Jpz{z}$ is a weakly holomorphic modular form of weight $1$ on
$\Gamma_0(p^2) \cap \Gamma_1(p)$,
which completes the proof of Theorem \thm{mainJp}.

%SECTION 6%%%%%%%%%%%%%%%%%%%%%%%%%%%%%%%%%%%%%%%%%%%%%%%%%%%%%%%%%%%%%%%%
\section{Dyson's rank conjecture and beyond}
\mylabel{sec:DRC}

In this section we give a new proofs of Dyson's rank conjecture,
and related results for the rank mod 11 due to Atkin and Hussain and
for the rank mod 13 due to O'Brien. 

We define the (weight $k$) Atkin $U_p$ operator by
\beq
\stroke{F}{U_p}{k} := \frac{1}{{p}} \sum_{r=0}^{p-1} F\Lpar{\frac{z+r}{p}} 
= p^{\frac{k}{2}-1}\sum_{n=0}^{p-1} \stroke{F}{T_r}{k},
\mylabel{Updef}
\eeq
where
$$
T_r = \begin{pmatrix} 1 & r \\ 0 & p \end{pmatrix},
$$
and the more general $U_{p,m}$ defined by
\beq
\stroke{F}{U_{p,m}}{k} := \frac{1}{p} \sum_{r=0}^{p-1} 
\exp\Lpar{-\frac{2\pi irm}{p}}\,F\Lpar{\frac{z+r}{p}} 
= p^{\frac{k}{2}-1} 
\sum_{r=0}^{p-1} \exp\Lpar{-\frac{2\pi irm}{p}}\, \stroke{F}{T_r}{k}. 
\mylabel{Upkdef}
\eeq
We note that $U_p = U_{p,0}$. In addition, if
$$
F(z) = \sum_n a(n) q^n = \sum_n a(n)\,\exp(2\pi izn),
$$
then
$$
\stroke{F}{U_{p,m}}{k} = q^{m/p} \sum_n a(pn+m)\,q^n = \exp(2\pi imz/p)\, 
\sum_n a(pn+m)\,\exp(2\pi inz).
$$

\begin{definition}
\mylabel{def:Kpm}
For $p>3$ prime and $0 \le m \le p-1$ define
\beq
\mathcal{K}_{p,m}(z) :=  \sin\Lpar{\frac{\pi}{p}} \, \stroke{\Jpz{z}}{U_{p,m}}{1},
\mylabel{eq:Kpmdef}
\eeq
where $\Jpz{z}$ is defined in \eqn{Jdef}.
\end{definition} 

Then a straightforward calculation gives
\begin{prop}
\mylabel{propo:Kpmprop}        
For $p>3$ prime and $0 \le m \le p-1$. 
\begin{enumerate}
\item[(i)]
For $m=0$ or $\leg{-24m}{p}=-1$ we 
have  
\beq
\mathcal{K}_{p,m}(z) = q^{m/p}\,\prod_{n=1}^\infty (1-q^{pn})\, 
\sum_{n=\CL{\frac{1}{p}(s_p-m)}}^\infty \left(\sum_{k=0}^{p-1} N(k,p,pn + m -s_p)\,\zeta_p^k\right)q^n,
\mylabel{eq:Kpm1prop}
\eeq
where $s_p=\frac{1}{24}(p^2-1)$, and $q=\exp(2\pi iz)$. 
\item[(ii)]
If $\leg{-24m}{p}=1$ we choose $1\le a \le \frac{1}{2}(p-1)$ so that
$$
-24m \equiv \Lpar{6a}^2 \pmod{p},
$$
and we have 
\begin{align}
&\mathcal{K}_{p,m}(z) 
\mylabel{eq:Kpm2prop}
\\
&= q^{m/p}\,\prod_{n=1}^\infty (1-q^{pn})\, 
\Bigg(
\sum_{n=\CL{\frac{1}{p}(s_p-m)}}^\infty \left(\sum_{k=0}^{p-1} N(k,p,pn + m -s_p)\,\zeta_p^k\right)q^n
\nonumber               
\\
& \quad            
- \chi_{12}(p) \, (-1)^a \,
  \left( \zeta_p^{3a + \frac{1}{2}(p+1)} + \zeta_p^{-3a - \frac{1}{2}(p+1)}  
  - 
    \zeta_p^{3a + \frac{1}{2}(p-1)} - \zeta_p^{-3a + \frac{1}{2}(p-1)}\right)
  \,
  q^{\frac{1}{p}( \frac{a}{2}(p - 3a) - m)}
  \,
  \Phi_{p,a}(q) 
  \Bigg).
\nonumber
\end{align}
\end{enumerate}
%% RSIFTID:=proc(m,p)
%%   local legm,a,a1,F1,F2;
%%   legm:=numtheory[legendre](-24*m,p);
%% if legm=1 then 
%%    for a1 from 1 to (p-1)/2 do if modp(-24*m-36*a1^2,p)=0 then a:=a1: 
%%      fi:od: print("a=",a);fi:
%% F1:=RzetaV2(m,p): 
%% F2:=chi12(p)*(-1)^a*xfac(a,p)*PHIpa(p,a,q,10,100)
%%       *q^( (a*(p-3*a)/2 - m)/p ): 
%% RETURN(F1-F2):
%% end:
%%                            
%%                             /      /      1     1   \
%%           (a, p) -> simplify|ccroot|3 a + - p + -, p|
%%                             \      \      2     2   /
%% 
%%                      /      1     1   \\
%%              - ccroot|3 a + - p - -, p||
%%                      \      2     2   //
%% 
\end{prop}              

Dyson's rank conjecture is equivalent to showing
$$
\mathcal{K}_{5,0}(z) = \mathcal{K}_{7,0}(z) = 0.
$$
In this section we will give a new proof of Dyson's rank conjecture and much more.
We will prove 
\begin{theorem}
\mylabel{thm:Kpthm}
Let $p>3$ be prime and suppose $0 \le m \le p-1$. Then 
\begin{enumerate}
\item[(i)]
$\mathcal{K}_{p,0}(z)$ is a weakly holomorphic modular form of weight $1$ on
$\Gamma_1(p)$.
\item[(ii)]
If $1 \le m \le (p-1)$ then $\mathcal{K}_{p,m}(z)$ is a 
weakly holomorphic modular form of weight $1$ on
$\Gamma(p)$.                           
In particular,
$$
\stroke{\mathcal{K}_{p,m}(z)}{A}{1}
= \exp\Lpar{\frac{2\pi ibm}{p}}\,\mathcal{K}_{p,m}(z),
$$
for $A=\AMat\in\Gamma_1(p)$.                       
\end{enumerate}
\end{theorem}

\begin{definition}
\mylabel{def:Rpm}
For $p>3$ prime and $0 \le m \le p-1$ define
$$
\mathcal{R}_{p,m}(z) := \frac{1}{\eta(pz)}\, \mathcal{K}_{p,m}(z).
$$
\end{definition}

From Theorem \thm{Kpthm} and \eqn{etamult} we have
\begin{cor}
\mylabel{cor:Rpm}
Let $p>3$ be prime and suppose %%$m=0$ or $\leg{-24m}{p}=-1$,
$0 \le m \le p-1$. Then $\mathcal{R}_{p,m}(z)$ 
is a weakly holomorphic modular form of weight $1/2$ on
$\Gamma_1(p)$ with multiplier. In particular,
$$
\stroke{\mathcal{R}_{p,m}(z)}{A}{\tfrac{1}{2}}
= \frac{\exp\Lpar{\frac{2\pi ibm}{p}}}{\nu_\eta({}^pA)}\,
\mathcal{R}_{p,m}(z),
$$
for $A=\AMat\in\Gamma_1(p)$ and where
$$
{}^pA = \ApMat.
$$
\end{cor}

\begin{remark}
When $m=0$ or $\leg{-24m}{p}=-1$ 
$$  
\mathcal{R}_{p,m}(z) = \sum_{n=\CL{\frac{1}{p}(s_p-m)}}^\infty 
\left(\sum_{k=0}^{p-1} N(k,p,pn+m-s_p)\,\zeta_p^k\right)
q^{n+\tfrac{m}{p}-\tfrac{p}{24}},
$$
and we note that in these cases Corollary \corol{Rpm}
greatly strengthens a theorem of Ahlgren and Treneer 
\cite[Theorem 1.6, p.271]{Ah-Tr08},
and a theorem of Bringmann, Ono and Rhoades 
\cite[Theorem 1.1(i)]{Br-On-Rh}.
\end{remark}

\begin{remark}
When $\leg{-24m}{p}=1$ we have
\begin{align*}
&\mathcal{R}_{p,m}(z) 
\\
&= 
\sum_{n=\CL{\frac{1}{p}(s_p-m)}}^\infty \left(\sum_{k=0}^{p-1} N(k,p,pn + m -s_p)\,\zeta_p^k\right)q^{n+\tfrac{m}{p}-\tfrac{p}{24}}
\\
& \quad            
- \chi_{12}(p) \, (-1)^a \,
  \left( \zeta_p^{3a + \frac{1}{2}(p+1)} + \zeta_p^{-3a - \frac{1}{2}(p+1)}  
  - 
    \zeta_p^{3a + \frac{1}{2}(p-1)} - \zeta_p^{-3a + \frac{1}{2}(p-1)}\right)
  \,
  q^{\frac{1}{p}( \frac{a}{2}(p - 3a) -\tfrac{p^2}{24})}
  \,
  \Phi_{p,a}(q).
\end{align*}
The result that this is a weakly holomorphic modular form is new.
Here $a$ is defined as in Proposition \propo{Kpmprop}(ii).
\end{remark}

%%\begin{cor}
%%\mylabel{cor:Kp2}
%%When $\leg{-24m}{p}=1$,
%%where $1 \le m \le p-1$. Let $a$ be as in Definition \refdef{Kpm}. Then 
%%\begin{align}
%%&\mathcal{R}_{p,m}(z) 
%%\mylabel{eq:Rpmdef}
%%\\
%%&:= 
%%\sum_{n=\CL{\frac{1}{p}(s_p-m)}}^\infty \left(\sum_{k=0}^{p-1} N(k,p,pn + m -s_p)\,\zeta_p^k\right)q^{n-?}
%%\nonumber               
%%\\
%%& \quad            
%%- \chi_{12}(p) \, (-1)^a \,
%%  \left( \zeta_p^{3a + \frac{1}{2}(p+1)} + \zeta_p^{-3a - \frac{1}{2}(p+1)}  
%%  - 
%%    \zeta_p^{3a + \frac{1}{2}(p-1)} - \zeta_p^{-3a + \frac{1}{2}(p-1)}\right)
%%  \,
%%  q^{\frac{1}{p}( \frac{a}{2}(p - 3a) - m +?)}
%%  \,
%%  \Phi_{p,a}(q) 
%%\nonumber
%%\end{align}
%%is a weakly holomorphic modular form of weight $1/2$ on
%%$\Gamma_1(p)$ with multiplier. In particular,
%%$$
%%{(A:z)}^{-\tfrac{1}{2}}\,R_{p,m}(A\,z) 
%%= \frac{\exp\Lpar{\frac{2\pi ibm}{p}}}{\nu_\eta({}^pA)}\,R_{p,m}(z),
%%$$
%%for $A=\AMat\in\Gamma_1(p)$ and where
%%$$
%%{}^pA = \ApMat.
%%$$
%%\end{cor}

\subsection{Proof of Theorem \thm{Kpthm}}
\mylabel{subsec:Kpthmproof}

We let
$$
A = \AMat \in \Gamma_1(p)
$$
so that $a\equiv d\equiv 1\pmod{p}$ and $c\equiv0\pmod{p}$. Let $0 \le k \le p-1$.
We take $k'\equiv b+k\pmod{p}$ so that      
$$
T_k \, A = B_k \, T_{k'},
$$
and
$$
B_k = T_k \, A \, T_{k'}^{-1} = 
\begin{pmatrix} a+ck & \tfrac{1}{p}(-k'(a+kc) + b+kd) \\
pc & d -k'c
\end{pmatrix} \in  \Gamma_0(p^2) \cap \Gamma_1(p). 
$$
Firstly
$$
\mathcal{K}_{p,m}(z) 
= \frac{1}{\sqrt{p}}\, 
\sin\Lpar{\frac{\pi}{p}} \, 
\sum_{k=0}^{p-1} 
\exp\Lpar{-\frac{2\pi ikm}{p}}\,
\stroke{\Jpz{z}}{T_r}{1}. 
$$
Thus
\begin{align*}
&\stroke{\mathcal{K}_{p,m}(z)}{A}{1}\\
&= \frac{1}{\sqrt{p}}\, 
\sin\Lpar{\frac{\pi}{p}} \, 
\sum_{k=0}^{p-1} 
\exp\Lpar{-\frac{2\pi ikm}{p}}\,
\stroke{\Jpz{z}}{\left(T_k \,A\right)}{1} \\
&= \frac{1}{\sqrt{p}}\, 
\sin\Lpar{\frac{\pi}{p}} \, 
\sum_{k=0}^{p-1} 
\exp\Lpar{-\frac{2\pi i(k'-b)m}{p}}\,
\stroke{\Jpz{z}}{\left(B_k \, T_{k'}\right)}{1} \\
&= \frac{1}{\sqrt{p}}\, 
\sin\Lpar{\frac{\pi}{p}} \, 
\exp\Lpar{\frac{2\pi ibm}{p}}\,
\sum_{k'=0}^{p-1} 
\exp\Lpar{-\frac{2\pi ik'm}{p}}\,
\stroke{\Jpz{z}}{T_{k'}}{1} \\
&{\hskip 2in}
\mbox{(by Theorem \thm{mainJp} since $B_k\in \Gamma_0(p^2) \cap \Gamma_1(p)$)}\\
&=
\exp\Lpar{\frac{2\pi ibm}{p}}\,
\mathcal{K}_{p,m}(z),
\end{align*}
as required. Thus each function $\mathcal{K}_{p,m}(z)$ has the desired
transformation property. It is clear that each $\mathcal{K}_{p,m}(z)$
is holomorphic on $\mathfrak{h}$. The cusp conditions follow
by a standard argument. For $m=0$ we will examine 
orders at each cusp in more detail in the next section.
%%Section XXX.
%%\end{proof}

\subsection{Orders at cusps}
\mylabel{subsec:cuspords}
Recall from Corollary \corol{Vc} that
\begin{align}
\mathfrak{V}_p&:= \left\{\mathcal{G}_1\Lpar{\frac{a}{p};z},\,
             \mathcal{G}_2\Lpar{\frac{a}{p};z}\,:\, 0< a < p\right\}
\mylabel{eq:Vp}\\
&\qquad \cup \left\{\mathcal{G}_1(a,b,p;z),\, \mathcal{G}_2(a,b,p;z)\,:\,
0\le a < p\quad\mbox{and}\quad 0<b<p\right\}
\nonumber
\end{align}
is a vector valued Maass form of weight $\tfrac{1}{2}$ for the full modular
group $\SL_2(\mathbb{Z})$, and that the action of $\SL_2(\mathbb{Z})$ on
each element is given explicitly by Theorem \thm{Gtrans}. Also for each
$\mathcal{G}\in\mathfrak{V}_p$ there are unique holomorphic functions
$\mathcal{G}_{\mbox{\scriptsize holo}}(z)$ and $\mathcal{G}_{\mbox{\scriptsize shadow}}(z)$ such that
$$
\mathcal{G}(z) = 
\mathcal{G}_{\mbox{\scriptsize holo}}(z) + \int_{-\zcon}^\infty 
\frac{\mathcal{G}_{\mbox{\scriptsize shadow}}(\tau)\,d\tau}{\sqrt{-i(\tau+z)}}.
$$
Also each $\mathcal{G}_{\mbox{\scriptsize holo}}(z)$ has a $q$-expansion
$$
\mathcal{G}_{\mbox{\scriptsize holo}}(z) = \sum_{m\ge m_0} a(m) 
\exp\Lpar{2\pi i z \frac{m}{24p^2}},
$$ 
where $a(m_0)\ne 0$. We define
$$
\hord(\mathcal{G};\infty) := \frac{m_0}{24p^2}.
$$
For any cusp $\frac{a}{c}$ with $(a,c)=1$ we define 
$$
\hord\Lpar{\mathcal{G};\frac{a}{c}} := \hord(\stroke{\mathcal{G}}{A}{1};\infty),
$$
where $A\in\SL_2(\mathbb{Z})$ and $A\infty = \frac{a}{c}$. 
We note that $\stroke{\mathcal{G}}{A}{1}\in
\mathfrak{V}_p$. One can easily check that this definition does not depend on the choice
of $A$ so that $\hord$ is well-defined.
We also note that when $\mathcal{G}(z)$ is a weakly holomorphic modular form 
this definition
coincides with the definition of invariant order at a cusp \cite[p.2319]{Br-Sw}, \cite[p.275]{Bi89}
The order of each function 
$\mathcal{F}(z)=\eta(z)\,\mathcal{G}(z)\in \eta(z) \cdot \mathfrak{V}_p$ is
defined in the natural way;i.e.\ 
$$
\hord\Lpar{\mathcal{F};\frac{a}{c}} 
:= \ord\Lpar{\eta(z);\frac{a}{c}}\,+\, \hord\Lpar{\mathcal{G};\frac{a}{c}},
$$
where $\ord\Lpar{\eta(z);\frac{a}{c}}$ is the usual invariant order of $\eta(z)$ at the cusp
$\frac{a}{c}$ \cite[p.34]{Ko-book}.

We determine $\hord\Lpar{\mathcal{F};\infty}$     for each 
$\mathcal{F}(z)=\eta(z)\,\mathcal{G}(z)\in \mathfrak{W}_p = \eta(z) \cdot \mathfrak{V}_p$.  
%%W_p was defined in Wpdef
After some calculation we find
\begin{prop}
\mylabel{propo:Fords}
Let $p>3$ be prime. Then
$$
\hord\Lpar{\Fell{1}{a}{p}{z};\infty} = 0,
$$
$$
\hord\Lpar{\mathcal{F}_1(a,b,p;z);\infty} 
= 
\begin{cases}
\frac{b}{2p} - \frac{3b^2}{2p^2} & \mbox{if $0\le \frac{b}{p} < \frac{1}{6}$},\\
\frac{3b}{2p} - \frac{3b^2}{2p^2} & \mbox{if $\frac{1}{6} <   \frac{b}{p} < \frac{1}{2}$},\\
\frac{5b}{2p} - \frac{3b^2}{2p^2} & \mbox{if $\frac{1}{2} <   \frac{b}{p} < \frac{5}{6}$},\\
\frac{7b}{2p} - \frac{3b^2}{2p^2} & \mbox{if $\frac{5}{6} <  \frac{b}{p} < 1$}, 
\end{cases}
$$
$$
\hord\Lpar{\Fell{2}{a}{p}{z};\infty}
=
\hord\Lpar{\mathcal{F}_2(a,b,p;z);\infty} 
= 
\begin{cases}
\frac{a}{2p} - \frac{3a^2}{2p^2} & \mbox{if $0\le \frac{a}{p} < \frac{1}{6}$},\\
\frac{3a}{2p} - \frac{3a^2}{2p^2} & \mbox{if $\frac{1}{6} <   \frac{a}{p} < \frac{5}{6}$},\\
\frac{5a}{2p} - \frac{3a^2}{2p^2} - 1& \mbox{if $\frac{5}{6} <   \frac{a}{p} < 1$}.
\end{cases}
$$
\end{prop}

We also need \cite[Corollary 2.2]{Ko-book}
\begin{prop}
\mylabel{propo:etaord}
Let $N\ge 1$ and let
$$
F(z) = \prod_{m\mid N} \eta(m z)^{r_m},
$$
where each $r_m\in\mathbb{Z}$. Then for $(a,c)=1$,
$$
\ord\Lpar{F(z);\frac{a}{c}} = \sum_{m\mid N} \frac{(m,c)^2 r_m}{24m}.
$$
\end{prop}

From \cite[Corollary 4, p.930]{Ch-Ko-Pa} we have
\begin{prop}
\mylabel{propo:cusps1}
Let $p>3$ be prime. Then a set of inequivalent cusps for $\Gamma_1(p)$ is given by
$$
i\infty,\, 0,\, 
\frac{1}{2},\, 
\frac{1}{3},\, 
\dots,\,
\frac{1}{\tfrac{1}{2}(p-1)},\,
\frac{2}{p},\, 
\frac{3}{p},\, 
\dots,\,
\frac{\tfrac{1}{2}(p-1)}{p}.
$$
\end{prop}

We next calculate lower bounds of the invariant order of 
$\Kp(z)$ at each cusp of $\Gamma_1(p)$.

\begin{theorem}
\mylabel{thm:Kpords}
Let $p>3 $ be prime,  
and suppose $2 \le m \le \tfrac{1}{2}(p-1)$. Then
\begin{enumerate}
\item[(i)]
$$
\ord\Lpar{\Kp(z);0}
\quad
\begin{cases}
\ge 0 & \mbox{$p=5$ or $7$},\\
=-\frac{1}{24p}(p-5)(p-7) & \mbox{if $p>7$};
\end{cases}
$$
\item[(ii)]
$$
\ord\Lpar{\Kp(z);\frac{1}{m}}
\quad
\begin{cases}
= - \frac{3}{2p}\Parans{\frac{1}{6}(p-1) - m}\Parans{\frac{1}{6}(p+1)-m} &
%%%=\frac{m}{2} - \frac{3m^2}{p} + \frac{1}{24p} - \frac{p}{24} & 
\mbox{if $2\le m <\tfrac{1}{6}(p-1)$},\\
\ge 0 & \mbox{otherwise};
\end{cases}
$$
and
\item[(iii)]
$$
\ord\Lpar{\Kp(z);\frac{m}{p}}
\ge \Lpar{\frac{p^2-1}{24p}}.
$$
\end{enumerate}
\end{theorem}
\begin{proof}
We derive lower bounds for $\ord\Lpar{\Kp(z);\zeta}$  for each cusp $\zeta$
of $\Gamma_1(p)$ not equivalent to $i\infty$.
%%%In each case we utilise equations \eqn{H1Tk1}, \eqn{H1Tk2}.
%%%Added 11.14.15:

First we show that
$$
\stroke{\Fell{2}{\ell}{p}{p^2 z}}{U_{p}}{1}=0,
$$
for $1 \le \ell \le \tfrac{1}{2}(p-1)$. From Theorem \thm{Ftrans}
$$
\Fell{2}{\ell}{p}{z+p} = \zeta_p^{\ell'}\, \Fell{2}{\ell}{p}{z} ,
$$
where $\ell' \equiv  \frac{3}{2}(p-1)\ell^2\pmod{p}$. We have
\begin{align*}
\stroke{\Fell{2}{\ell}{p}{p^2 z}}{U_{p}}{1}
&= \frac{1}{p} \sum_{r=0}^{p-1} \Fell{2}{\ell}{p}{p^2z + rp} \\
&= \frac{1}{p} \Fell{2}{\ell}{p}{p^2z} 
    \sum_{r=0}^{p-1} \Lpar{\zeta_p^{\ell'}}^r
\\
&=0.
\end{align*}
Thus from \eqn{JJSid}, \eqn{JSdef}, \eqn{Kpmdef} we have
$$
\mathcal{K}_{p,0}(z) =  \sin\Lpar{\frac{\pi}{p}} \, \stroke{\Jpz{z}}{U_{p,0}}{1}
=  \sin\Lpar{\frac{\pi}{p}} \, 
\stroke{\SFell{1}{1}{p}{z}}{U_p}{1}.
$$
where
\beq
\SFell{1}{1}{p}{z} = \frac{\eta(p^2 z)}{\eta(z)}\, \Fell{1}{1}{p}{z}.
\mylabel{eq:SFelldef}
\eeq

Next we calculate 
$$
\stroke{\SFell{1}{1}{p}{z}}{T_k\,A}{1}.
$$
for each $0 \le k \le p-1$ and each $A=\AMat\in\SLZ$.

\subsubsection*{Case 1} 
\mylabel{subsubsec:CuspConcase1}
$a+kc\not\equiv0\pmod{p}$. 
Choose $0\le k'\le p-1$ such that
$$
(a+kc)\, k' \equiv (b+kd) \pmod{p}.
$$
Then
$$
T_k \, \gA = C_k \, T_{k'},
$$
where
$$
C_k = T_k \, \gA \, T_{k'}^{-1} = 
\begin{pmatrix} a+ck & \tfrac{1}{p}(-k'(a+kc) + b+kd) \\
pc & d -k'c
\end{pmatrix} \in \Gamma_0(p).
$$
From Theorem \thm{mainthm} we have
\begin{align}
\stroke{\SFell{1}{1}{p}{z}}{T_k\,A}{1}
=
\stroke{\SFell{1}{1}{p}{z}}{C_k\,T_{k'}}{1}
\mylabel{eq:Tk1}\\
=
\Lpar{\stroke{F_p(z)}{C_k\,T_{k'}}{0}}\,
\Lpar{\mu(C_k,1)\,\stroke{\Fell{1}{\ell}{p}{z}}{T_{k'}}{1}},
\nonumber
\end{align}
where
\beq
F_p(z) = \frac{\eta(p^2 z)}{\eta(z)}.
\eeq
\subsubsection*{Case 2} 
\mylabel{subsubsec:CuspConcase2}
$a+kc\equiv0\pmod{p}$. In this case we find that
$$
T_k \, \gA = D_k \, P,
$$
where
$$
P = \begin{pmatrix} p & 0 \\ 0 & 1\end{pmatrix},
$$
$$
D_k = \begin{pmatrix} \frac{1}{p}(a+kc) & b + kd \\ c & pd \end{pmatrix}\in
\SL_2(\mathbb{Z}),
$$
and  
$$
E_k = D_k\,S = 
\begin{pmatrix}
b+kd  &  \frac{-1}{p}(a+kc)\\
pd & -c
\end{pmatrix}\in\Gamma_0(p). 
$$
From Theorem \thm{mainthm} we have
$$
\stroke{\Fell{1}{1}{p}{z}}{E_k}{1}
= \mu(E_k,1) \, \Fell{1}{\overline{-c}}{p}{z}.
$$
By Theorem \thm{Ftrans}(2) we have
$$
\stroke{\Fell{1}{1}{p}{z}}{D_k}{1}
= \mu(E_k,1) \, \stroke{\Fell{1}{\overline{-c}}{p}{z}}{S^{-1}}{1}
= i \mu(E_k,1) \, \Fell{2}{\overline{-c}}{p}{z},
$$
and
$$
\stroke{\Fell{1}{1}{p}{z}}{T_k\,A}{1}
= i \mu(E_k,1) \, \stroke{\Fell{2}{\overline{-c}}{p}{z}}{P}{1},
$$
so that
\begin{equation}
\stroke{\SFell{1}{1}{p}{z}}{T_k\,A}{1}
=
\Lpar{\stroke{F_p(z)}{D_k\,P}{0}}\,
\Lpar{i \mu(E_k,1) \, \stroke{\Fell{2}{\overline{-c}}{p}{z}}{P}{1}}.
\mylabel{eq:Tk2}
\end{equation}

Now we are ready to exam each cusp $\zeta$ of $\Gamma_1(p)$.
We choose 
$$
\gA = \AMat \in \SL_2(\mathbb{Z}),\quad\mbox{so that $\gA(\infty) = \frac{a}{c}=\zeta$.}
$$

\subsubsection*{(i) $\zeta=0$}
Here $a=0$, $c=1$ and we assume $0 \le k \le p-1$.
If $k\ne0$ then applying \eqn{Tk1} we have
\begin{align*}
\hord\Lpar{\SFTk;0} = 
&= \frac{1}{p}\, \ord\Lpar{F_p(z);\frac{k}{p}} + 
  \frac{1}{p} \, \hord \Lpar{\Fell{1}{\ell}{p}{z}; i\infty}\\
&= 0 + 0 = 0,
\end{align*}
by Propositions \propo{etaord} and \propo{Fords}. Now applying \eqn{Tk2}
with $k=0$ we have
\begin{align*}
\hord\Lpar{\SFTZ;0}
&= p\, \ord\Lpar{F_p(z);0} + 
   p\, \hord \Lpar{\Fell{2}{1}{p}{z}; i\infty}\\
&= -\frac{1}{24p}(p^2-1)  + 
\begin{cases} \frac{6}{5} & \mbox{if $p=5$},\\
              \frac{1}{2p}(p-3) & \mbox{if $p>5$},
\end{cases}\\
&=                              
\begin{cases} 1 & \mbox{if $p=5$},\\
              -\frac{1}{24p}(p-5)(p-7) & \mbox{if $p>5$},
\end{cases}
\end{align*}
again by Propositions \propo{etaord} and \propo{Fords}. The result (i) follows since
$$
\ord\Lpar{\Kp(z);0} \ge \, \min_{0 \le k \le p-1} 
\hord\Lpar{\SFTk;0}.
$$
\subsubsection*{(ii) $\zeta=\frac{1}{m}$, where $2 \le m \le \tfrac{1}{2}(p-1)$}
Let $A=\begin{pmatrix}1 & 0 \\ m & 1 \end{pmatrix}$ so that $A(\infty)=1/m$.
If $km\not\equiv-1\pmod{p}$ we apply \eqn{Tk1} with 
$C_k = \begin{pmatrix}1 + km & * \\ pm & 1-k'm\end{pmatrix}$, and find that
\begin{align*}
\hord\Lpar{\SFTk;\frac{1}{m}}
&= \frac{1}{p}\, \ord\Lpar{F_p(z);\frac{1+km}{pm}} + 
  \frac{1}{p} \, \hord \Lpar{\Fell{1}{\ell}{p}{z}; i\infty}\\
&= 0 + 0 = 0.
\end{align*}
Now we assume $km\equiv-1\pmod{p}$ and we will apply \eqn{Tk2}. We have
\begin{align*}
&\hord\Lpar{\SFTk;\frac{1}{m}}
= p\, \ord\Lpar{F_p(z);\frac{(1+km)/p}{m}} + 
   p\, \hord \Lpar{\Fell{2}{m}{p}{z}; i\infty}\\
&= -\frac{1}{24p}(p^2-1)  + 
\begin{cases} \frac{m}{2} - \frac{3m^2}{2p}  & \mbox{if $2 \le m < \frac{p}{6}$},\\
              \frac{3m}{2} - \frac{3m^2}{2p}  & \mbox{if $\frac{p}{6} < m \le \frac{p-1}{2}$},
\end{cases}\\
&                              
\begin{cases}
= - \frac{3}{2p}\Parans{\frac{1}{6}(p-1) - m}\Parans{\frac{1}{6}(p+1)-m} &
%%%=\frac{m}{2} - \frac{3m^2}{p} + \frac{1}{24p} - \frac{p}{24} & 
\mbox{if $2\le m \le \tfrac{1}{6}(p-1)$},\\
> 0 & \mbox{otherwise},
\end{cases}
\end{align*}
and the result (ii) follows.

\subsubsection*{(iii) $\zeta=\frac{m}{p}$, where $2 \le m \le \tfrac{1}{2}(p-1)$}
Choose $b$,$d$ so that $A=\begin{pmatrix}m & b \\ p & d \end{pmatrix}
\in \SL_2(\mathbb{Z})$ and 
$A(\infty)=m/p$. Since $m\not\equiv0\pmod{p}$ we may apply \eqn{Tk1} for each $k$. We find that
$C_k = \begin{pmatrix}m + kp & * \\ p^2 & d-k'p\end{pmatrix}$, and 
\begin{align*}
\hord\Lpar{\SFTk;\frac{m}{p}}
&= \frac{1}{p}\, \ord\Lpar{F_p(z);\frac{m+kp}{p^2}} + 
  \frac{1}{p} \, \hord \Lpar{\Fell{1}{\ell}{p}{z}; i\infty}\\
&= \frac{p^2-1}{24p} + 0 = \frac{p^2-1}{24p}.
\end{align*}
The result (iii) follows.
\end{proof}
Since 
$$
\ord\Lpar{\Kp(z);0}
=-\frac{1}{24p}(p-5)(p-7) < 0,
$$
for $p >7$ we have
\begin{cor}
\mylabel{cor:Dysonanalog}
The analog of the Dyson Rank Conjecture does not hold for any prime $p >7$.
In other words $\Kp(z) \not \equiv 0$ for any prime $p > 7$.
\end{cor}

\subsection{Proof of Dyson's rank conjecture}
\mylabel{subsec:proofDRC}

As noted before
Dyson's rank conjecture is equivalent to showing
\beq
\mathcal{K}_{5,0}(z) = \mathcal{K}_{7,0}(z) = 0.
\mylabel{eq:K57}
\eeq
The first proof was given by Atkin and Swinnerton-Dyer \cite{At-Sw}.
Their proof involved finding and  proving identities for basic hypergeometric functions,
theta-functions and Lerch-type series using the theory of elliptic functions.
It also involved identifying the generating functions for rank differences
$N(0,p,pn+r)-N(k,p,pn+r)$  for $p=5$, $7$ for each $ 1 \le k \le\tfrac{1}{2}(p-1)$ and
each $r=0$, $1$,\dots $p-1$. Atkin and Swinnerton-Dyer note \cite[p.84]{At-Sw} that they
are unable to simplify their proof so as only to obtain Dyson's results. In particular
to prove the result for $(p,r)=(5,4)$ or $(p,r)=(7,5)$ they must simultaneously  
prove identities for \text{all} $r$ with $0\le r \le p-1$. Here we show how to avoid this
difficulty.

To prove \eqn{K57} we use
\begin{theorem}[The Valence Formula \cite{Ra}(p.98)]
\mylabel{thm:val}
Let $f\ne0$ be a modular form of weight $k$ with respect to a subgroup $\Gamma$ of finite index
in $\Gamma(1)=\SL_2(\mathbb{Z})$. Then
\beq
\ORD(f,\Gamma) = \frac{1}{12} \mu \, k,
\mylabel{eq:valform}
\eeq
where $\mu$ is index $\widehat{\Gamma}$ in $\widehat{\Gamma(1)}$,
$$
\ORD(f,\Gamma) := \sum_{\zeta\in R^{*}} \ORD(f,\zeta,\Gamma),
$$
$R^{*}$ is a fundamental region for $\Gamma$,
and
$$
\ORD(f;\zeta;\Gamma) = n(\Gamma;\zeta)\, \ord(f;\zeta),
$$
for a cusp $\zeta$ and
$n(\Gamma;\zeta)$ denotes the fan width of the cusp $\zeta \pmod{\Gamma}$.
\end{theorem}
\begin{remark}
For $\zeta\in\mathfrak{h}$,
$\ORD(f;\zeta;\Gamma)$ is defined in terms of 
 the invariant order $\ord(f;\zeta)$, which  is interpreted
in the usual sense. See \cite[p.91]{Ra} for details of this and the 
notation used.
\end{remark}

\subsubsection*{$p=5$}

\begin{center}
\begin{tabular}{c c c c}
cusp $\zeta$ & $n(\Gamma_1(5);\zeta)$ & $\ord(\DK{5}(z);\zeta)$ 
& $\ORD(\DK{5}(z),\Gamma_1(5),\zeta)$ \\
$i\infty$    & $1$ & $\ge 1$ & $\ge 1$ \\
$0$          & $5$ & $\ge 0$ & $\ge 0$ \\
$\frac{1}{2}$ & $5$ & $\ge 0$ & $\ge 0$ \\
$\frac{2}{5}$ & $1$ & $\ge 1$ & $\ge 1$ 
\end{tabular}
\end{center}

Hence $\ORD(\DK{5}(z);\Gamma_1(5))\ge 2$.
But $\mu k = \frac{12}{12} = 1$. The Valence Formula implies that
$\DK{5}(z)$ is identically zero which proves Dyson's conjecture for $p=5$.

\subsubsection*{$p=7$}

\begin{center}
\begin{tabular}{c c c c}
cusp $\zeta$ & $n(\Gamma_1(7);\zeta)$ & $\ord(\DK{7}(z);\zeta)$ 
& $\ORD(\DK{7}(z),\Gamma_1(7),\zeta)$ \\
$i\infty$    & $1$ & $\ge 1$ & $\ge 1$ \\
$0$          & $7$ & $\ge 0$ & $\ge 0$ \\
$\frac{1}{2}$ & $7$ & $\ge 0$ & $\ge 0$ \\
$\frac{1}{3}$ & $7$ & $\ge 0$ & $\ge 0$ \\
$\frac{2}{7}$ & $1$ & $\ge 1$ & $\ge 1$ \\
$\frac{3}{7}$ & $1$ & $\ge 1$ & $\ge 1$    
\end{tabular}
\end{center}

Hence $\ORD(\DK{7}(z);\Gamma_1(7))\ge 3$.
But $\mu k = \frac{24}{12} = 2$. The Valence Formula implies that
$\DK{7}(z)$ is identically zero which proves Dyson's conjecture for $p=7$.

\subsection{The rank mod $11$}                     
\mylabel{subsec:rank11}

Atkin and Hussain \cite{At-Hu} studied the rank mod $11$.
In this section we find an identity for $\DK{11}(z)$.
For $1\le N\nmid k$  and following Robins \cite{Ro94} we define the generalized eta-function
\begin{equation}
\mylabel{eq:Geta}
\eta_{N,k}(\tau) =q^{\frac{N}{2} P_2(k/N) }
\prod_{
       \substack{m>0 \\ m\equiv \pm k\pmod{N}}} (1 - q^m),
\end{equation}
where $\tau \in \mathfrak{h}$, $P_2(t) = \{t\}^2 - \{t\} + \tfrac16$ 
is the second periodic Bernoulli polynomial, 
and $\{t\}=t-[t]$ is the fractional part of $t$.

%%\prod_{\substack{j=1 \\ j\not\equiv0\pmod{p}}}^{n}               

We will prove that
\beq
(q^{11};q^{11})_\infty \,
\sum_{n=1}^\infty \Lpar{\sum_{k=0}^{10} N(k,11,11n-5)\,\zeta_{11}^k}q^n
=
\sum_{k=1}^5 c_{11,k}\, j_{11,k}(\tau),
\mylabel{eq:rank11id}
\eeq
where
$$
j_{11,k}(\tau) = \frac{\eta(11\tau)^4}{\eta(\tau)^2} \, 
\frac{1}{\eta_{11,4k}(\tau)\,\eta_{11,5k}(\tau)^2},
$$
and
\begin{align*}
%%> c1:=sort(coeff(REL11[1],X[1],1));
%%                9         8       7       4         3         2    
%%          2 zeta  + 2 zeta  + zeta  + zeta  + 2 zeta  + 2 zeta  + 1
%%> latex(%);
c_{11,1} &= 2\,{\zeta_{11}}^{9}+2\,{\zeta_{11}}^{8}+{\zeta_{11}}^{7}+{\zeta_{11}}^{4}+2\,{\zeta_{11}}^{3}+2
\,{\zeta_{11}}^{2}+1,\\
%%> c2:=sort(coeff(REL11[1],X[2],1));
%%        9       8         7       6       5         4       3       2    
%%   -zeta  - zeta  - 2 zeta  - zeta  - zeta  - 2 zeta  - zeta  - zeta  - 1
%%> latex(%);
c_{11,2} & =-({\zeta_{11}}^{9}+{\zeta_{11}}^{8}+2\,{\zeta_{11}}^{7}+{\zeta_{11}}^{6}+{\zeta_{11}}^{5}+2\,{
\zeta_{11}}^{4}+{\zeta_{11}}^{3}+{\zeta_{11}}^{2}+1),\\
%%> c3:=sort(coeff(REL11[1],X[3],1));
%%                        8         7         4         3    
%%                  2 zeta  + 2 zeta  + 2 zeta  + 2 zeta  + 3
%%> latex(%);
c_{11,3} &= 2\,{\zeta_{11}}^{8}+2\,{\zeta_{11}}^{7}+2\,{\zeta_{11}}^{4}+2\,{\zeta_{11}}^{3}+3,\\
%%> c4:=sort(coeff(REL11[1],X[4],1));
%%      9       8         7         6         5         4       3         2    
%%4 zeta  + zeta  + 2 zeta  + 2 zeta  + 2 zeta  + 2 zeta  + zeta  + 4 zeta  + 4
%%> latex(%);
c_{11,4} &= 4\,{\zeta_{11}}^{9}+{\zeta_{11}}^{8}+2\,{\zeta_{11}}^{7}+2\,{\zeta_{11}}^{6}
+2\,{\zeta_{11}}^{5}+2\,{\zeta_{11}}^{4}+{\zeta_{11}}^{3}+4\,{\zeta_{11}}^{2}+4,\\
%%> c5:=sort(coeff(REL11[1],X[5],1));
%%      9         8       7         6         5       4         3       2    
%% -zeta  - 2 zeta  + zeta  - 2 zeta  - 2 zeta  + zeta  - 2 zeta  - zeta  - 3
%%> latex(%);
c_{11,5} &=-({\zeta_{11}}^{9}+2\,{\zeta_{11}}^{8}-{\zeta_{11}}^{7}+2\,{\zeta_{11}}^{6}+2\,{\zeta_{11}}^{5}-
{\zeta_{11}}^{4}+2\,{\zeta_{11}}^{3}+{\zeta_{11}}^{2}+3).
\end{align*}

By Theorem \thm{Kpthm} we know that the left side of \eqn{rank11id} is  a weakly
holomorphic modular form of weight $1$ on $\Gamma_1(11)$. We prove \eqn{rank11id}
by showing that the right side is also a 
weakly
holomorphic modular form of weight $1$ on $\Gamma_1(11)$ and using the Valence Formula
\eqn{valform}.  Following Biagioli \cite{Bi89} we define
\beq
f_{N,\rho}(\tau) := q^{(N-2\rho)^2/(8N)}\,(q^\rho,q^{N-\rho},q^N;q^N)_\infty.
\mylabel{eq:fdef}
\eeq
Then
$$
f_{N,\rho}(\tau) = f_{N,N+\rho}(\tau) = f_{N,-\rho}(\tau),
$$
and
$$
f_{N,\rho}(\tau) = \eta(N\tau)\,\eta_{N,\rho}(\tau).
$$
We observe that
$$
j_{11,k}(\tau) = \frac{\eta(11\tau)^7}{\eta(\tau)^2} \, 
\frac{1}{f_{11,4k}(\tau)\,f_{11,5k}(\tau)^2},
$$
for $1 \le k \le 5$.
\begin{theorem}[Biagioli \cite{Bi89}]
Let $A=\AMat\in\Gamma_0(N)$. Then
$$
\stroke{f_{N,\rho}(\tau)}{A}{\tfrac{1}{2}} = (-1)^{\rho b + \FL{\rho a/N} + \FL{\rho/N}}\, 
\exp\Lpar{\frac{\pi iab}{N}\rho^2}
\,\nu_{\theta_1}\Lpar{{}^NA}\,
f_{N,\rho a}(\tau),
$$ 
where
$$
{}^NA = \NAMat \in \SL_2(\mathbb{Z}),
$$
and
$\nu_{\theta_1} = \nu_\eta^3$ is the theta-multiplier.
\end{theorem}
Biagioli has also calculated the orders of these functions at the cusps.
\begin{prop}
\mylabel{propo:ford}
Let $1\le N\nmid \rho$, and $(a,c)=1$. Then
$$
\ord\Lpar{f_{N,\rho}(\tau),\frac{a}{c}} = 
\frac{g}{2N}\Lpar{\frac{a\rho}{g} - \FL{\frac{a\rho}{g}} - \frac{1}{2}}^2,
$$
where $g=(N,c)$.
\end{prop}
We need Knopp's \cite[Theorem 2, p.51]{Kn-book} formula for the eta-multiplier
given in
\begin{theorem}
\mylabel{thm:Knetamult}
For $A=\AMat\in\SL_2(\mathbb{Z})$ we have
$$
\stroke{\eta(\tau)}{A}{\tfrac{1}{2}} = \nu_\eta(A)\,\eta(\tau),
$$
where
$$
\nu_\eta(A) = 
\begin{cases}
\jacsa{d}{c}\,\exp\Lpar{\frac{\pi i}{12}((a+d)c - bd(c^2-1)-3c} & \mbox{if $c$ is odd},\\
\jacsb{c}{d}\,\exp\Lpar{\frac{\pi i}{12}((a+d)c - bd(c^2-1)+3d-3-3cd} & \mbox{if $d$ is odd},
\end{cases}
$$
$$
\jacsa{d}{c} = \leg{c}{\abs{d}},\qquad
\jacsb{c}{d} = \leg{c}{\abs{d}}\,(-1)^{(\sgn c - 1)(\sgn d -1)/4},
$$
and $\leg{\cdot}{\cdot}$ is the Jacobi symbol.
\end{theorem}

We now show that each $j_{11,k}(\tau)$ is a weakly holomorphic function of weight $1$
on $\Gamma_1(11)$. Suppose $A=\AMat\in\Gamma_1(11)$. Then
$$
\eta(11 A\tau) = \eta({}^{11}A(11\tau)) = \nu_\eta({}^{11}A) \,\sqrt{c\tau+d}\,\eta(11\tau),
$$
and
\begin{align*}
&\stroke{j_{11,k}}{A}{1} \\
&= \nu_\eta({}^{11}A)^7\, \nu_\eta(A)^{-2}\, \nu_{\theta_1}^{-3}({}^{11}A) \,
  (-1)^{\FL{4ka/11} + \FL{4k/11}} \,
  \exp\Lpar{\frac{-\pi iab}{11}(66k^2)}\\
&\qquad\qquad \times
\frac{\eta(11\tau)^7}{\eta(\tau)^2} \, 
\frac{1}{f_{11,4ka}(\tau)\,f_{11,5ka}(\tau)^2},\\
&= 
\nu_\eta({}^{11}A)^{-2}\, \nu_\eta(A)^{-2}\, j_{11,k},
\end{align*}
since $a\equiv1\pmod{11}$ and 
$$
\FL{4ka/11} \equiv \FL{4k/11}  \pmod{2}.
$$
So we must show that
$$
\nu_\eta({}^{11}A)^{2}\, \nu_\eta(A)^{2} = 1.
$$
\subsubsection*{Case 1} $c$ is odd. For $p>3$ prime we find 
$$
\nu_\eta({}^{p}A)^{2}\, \nu_\eta(A)^{2} 
= \exp\Lpar{\frac{\pi i}{6p}(p+1)(-3c-bdc^2 + bdp+ ca + cd)} = 1
$$
in this case since $p\mid c$ and $p\equiv 11\pmod{12}$.
\subsubsection*{Case 2} $d$ is odd. For $p>3$ prime we find 
$$
\nu_\eta({}^{p}A)^{2}\, \nu_\eta(A)^{2} 
= \exp\Lpar{\frac{\pi i}{6p}(p+1)(-acd^2 - ac + dbp -dc)} = 1
$$
in this case since $p\mid c$ and $p\equiv 11\pmod{12}$.
It follows that  each $j_{11,k}(\tau)$ is a weakly holomorphic function of weight $1$
on $\Gamma_1(11)$. 

Next we calculate orders at cusps. From Propositions \propo{etaord} and \propo{ford} we have
\begin{center}
\begin{tabular}{c c c c c c}
cusp  &            &            & $\ord(j_{11,k},\zeta)$ &            &            \\
$\zeta$ & $j_{11,1}$ & $j_{11,2}$ & $j_{11,3}$ & $j_{11,4}$ & $j_{11,5}$ \\
$i\infty$ & $3$          &  $1$         & $2$          & $2$          & $2$ \\
$1/m$ & $-1/11$          &  $-1/11$         & $-1/11$          & $-1/11$          & $-1/11$ \\
$2/11$ & $1$          &  $2$         & $2$          & $2$          & $3$ \\
$3/11$ & $2$          &  $2$         & $1$          & $3$          & $2$ \\
$4/11$ & $2$          &  $2$         & $3$          & $2$          & $1$ \\
$5/11$ & $2$          &  $3$         & $2$          & $1$          & $2$ 
\end{tabular}
\end{center}
where $2\le m\le 5$.

Now we calculate lower bounds of orders at cusps of both sides of equation \eqn{rank11id}.
\begin{center}
\begin{tabular}{c c c c c}
cusp $\zeta$ & $n(\Gamma_1(11);\zeta)$ & $\ord(LHS;\zeta)$ & $\ord(RHS;\zeta)$ 
& $\ORD(LHS-RHS;\zeta)$ \\
$i\infty$    & $1$ &         &        &                 \\
$0$          & $11$ &  $\ge-1/11$       &  $\ge-1/11$      &   $\ge -1$      \\
$1/m$          & $11$ &  $\ge0$       &  $\ge-1/11$      &   $\ge -1$      \\
$m/11$          & $1$ &  $\ge1$       &  $1$      &   $\ge 1$     
\end{tabular} 
\end{center} 
where $2\le m\le 5$. But $\mu k = 5$. The result follows from the Valence Formula \eqn{valform} 
provided we can show that $\ORD(LHS-RHS,i\infty)\ge 7$. This is easily 
carried out using MAPLE.  

\subsection{The rank mod $13$}                     
\mylabel{subsec:rank13} 
J.~N.~O'Brien \cite{OB} has studied the rank mod $13$.  Using the methods 
of the previous section we may obtain an identity for $\DK{13}(z)$.  We 
state the identity and omit the details.  
\beq
(q^{13};q^{13})_\infty \,
\sum_{n=1}^\infty \Lpar{\sum_{k=0}^{10} N(k,13,13n-7)\,\zeta_{13}^k}q^n
=
\sum_{k=1}^6 \Lpar{c_{13,k} + d_{13,k}\,13\,\frac{\eta(13\tau)^2}{\eta(\tau)^2}}\, j_{13,k}(\tau),
\mylabel{eq:rank13id}
\eeq
where
$$
j_{13,k}(\tau) = \frac{\eta(13\tau)^3}{\eta(\tau)} \, 
\frac{1}{\eta_{13,2k}(\tau)^2\,\eta_{13,3k}(\tau)\,\eta_{13,4k}(\tau)\,\eta_{13,5k}(\tau)\,
\eta_{13,6k}(\tau)^2},
$$
and
\begin{align*}
c_{13,1} &= 5 \zeta_{13}^{11}+\zeta_{13}^{10}+5 \zeta_{13}^9+2 \zeta_{13}^8+3 \zeta_{13}^7+3 \zeta_{13}^6+2 \zeta_{13}^5+5 \zeta_{13}^4+\zeta_{13}^3+5 \zeta_{13}^2+6,\\
c_{13,2} &= -\zeta_{13}^{11}+2 \zeta_{13}^9+2 \zeta_{13}^8-\zeta_{13}^7-\zeta_{13}^6+2 \zeta_{13}^5+2 \zeta_{13}^4-\zeta_{13}^2+3,\\
c_{13,3} &=  \zeta_{13}^{11}+2 \zeta_{13}^{10}+2 \zeta_{13}^9+2 \zeta_{13}^4+2 \zeta_{13}^3+\zeta_{13}^2-1,\\
c_{13,4} &= 3 \zeta_{13}^{11}+3 \zeta_{13}^{10}+5 \zeta_{13}^8+\zeta_{13}^7+\zeta_{13}^6+5 \zeta_{13}^5+3 \zeta_{13}^3+3 \zeta_{13}^2+5,\\
c_{13,5} &= \zeta_{13}^{11}-3 \zeta_{13}^{10}-\zeta_{13}^9+2 \zeta_{13}^8-2 \zeta_{13}^7-2 \zeta_{13}^6+2 \zeta_{13}^5-\zeta_{13}^4-3 \zeta_{13}^3+\zeta_{13}^2-2,\\
c_{13,6} &= -\zeta_{13}^{11}-\zeta_{13}^{10}-2 \zeta_{13}^9-\zeta_{13}^8-2 \zeta_{13}^7-2 \zeta_{13}^6-\zeta_{13}^5-2 \zeta_{13}^4-\zeta_{13}^3-\zeta_{13}^2-1,\\
d_{13,1} &=  2 \zeta_{13}^{11}+\zeta_{13}^9+\zeta_{13}^8+\zeta_{13}^7+\zeta_{13}^6+\zeta_{13}^5+\zeta_{13}^4+2 \zeta_{13}^2+2,\\
d_{13,2} &=  -\zeta_{13}^{11}-\zeta_{13}^{10}-\zeta_{13}^7-\zeta_{13}^6-\zeta_{13}^3-\zeta_{13}^2,\\
d_{13,3} &=  \zeta_{13}^{11}+\zeta_{13}^{10}+\zeta_{13}^9+\zeta_{13}^8+\zeta_{13}^5+\zeta_{13}^4+\zeta_{13}^3+\zeta_{13}^2+1,\\
d_{13,4} &= \zeta_{13}^{10}-\zeta_{13}^9+\zeta_{13}^8+\zeta_{13}^5-\zeta_{13}^4+\zeta_{13}^3+1,\\
d_{13,5} &=  -\zeta_{13}^{10}-\zeta_{13}^9-\zeta_{13}^7-\zeta_{13}^6-\zeta_{13}^4-\zeta_{13}^3-2,\\
d_{13,6} &=  -\zeta_{13}^8-\zeta_{13}^5.
\end{align*}

% if trans=_T and type1=2 then
%    if a-b>=0 then
%       newconst1:=subs(z=z+1,const1)*exp(2*Pi*I*3*b^2/2/c^2):
%       newtype1:=type1: newa:=a-b: newb:=b: newc:=c:
%    else
%      newconst1:=(-1)*subs(z=z+1,const1)*exp(2*Pi*I*3*b^2/2/c^2-6*Pi*I*b/c):
%      newtype1:=type1: newa:=a-b+c: newb:=b: newc:=c:
%    fi:
% fi:
% if trans=_T and type1=3 then
%   newconst1:=subs(z=z+1,const1)*exp(2*Pi*I*(5*a/2/c - 3*a^2/2/c^2)):
%   newtype1:=4: newa:=a: newb:=a: newc:=c:
% fi:
% if trans=_T and type1=4 then
%   if modp(a+b,c)<>0 then
%     newconst1:=subs(z=z+1,const1)*exp(2*Pi*I*(5*a/2/c - 3*a^2/2/c^2)): newtype1:=type1:
%     newa:=a: newb:=modp(a+b,c): newc:=c:
%   else
%     newconst1:=subs(z=z+1,const1)*exp(2*Pi*I*(5*a/2/c - 3*a^2/2/c^2)): newtype1:=3:
%     newa:=a: newb:=0: newc:=c:
%   fi:
% fi:

%SECTION 7%%%%%%%%%%%%%%%%%%%%%%%%%%%%%%%%%%%%%%%%%%%%%%%%%%%%%%%%%%%%%%%%
\section{Concluding Remarks}                            
\mylabel{sec:conclusion}

Hickerson and Mortenson \cite{Hi-Mo16} have an 
alternative approach to the work of Bringmann, Ono and Rhoades
\cite{Br-On10}, \cite{Br-On-Rh}
on Dyson's rank. Using \cite[Theorem 3.5 and 3.9]{Hi-Mo14} they show how 
to make the results of 
Bringmann, Ono and Rhoades more explicit. For integers $0 \le a < M$ they define
\beq
D(a,M) = D(a,M,q) := \sum_{n=0}^\infty
\Lpar{N(a,M,n) - \frac{p(n)}{M}}q^n.
\mylabel{eq:HMDdef}.
\eeq
Their main theorem decomposes \eqn{HMDdef} into modular and mock modular
components
\beq
D(a,M,q) = d(a,M,q) + T_{a,M}(q),
\mylabel{eq:Damq}
\eeq
where $d(a,M,q)$ the mock modular part is given explicitly in terms of
Appell-Lerch series, and $T_{a,M}(q)$ are certain linear combinations
of theta-quotients. Hickerson and Mortenson do not consider
modular transformation properties of the Dyson rank function.
Our approach is extend Bringmann and Ono's results. Hickerson and Mortenson
results depend on properties for Appell-Lerch series. It is not clear
that Hickerson and Mortenson results imply our Theorem \thm{mainJp0}.
It would also be interesting to extend our theorem to $p=2$, $3$ and prime
powers. When $M=p>3$ is prime it should be possible to use Theorem \thm{mainJp0}
to obtain a result like \eqn{Damq} except $T_{a,M}(q)$ would not be 
given explicitly but rather as an element of an explicit space of
modular forms.

Finally we note that Andersen \cite{An16} has applied the results of this paper
as well as results of Zwegers \cite{Zw-thesis} to give a new proof
of Ramanujan's mock-theta conjectures \cite{An-Ga89}.
As mentioned earlier, 
the first proof of the mock-theta conjectures was given by Hickerson \cite{Hi88a}.
Folsom \cite{Fo} showed how the mock-theta conjectures could be proved using
the theory of Maass forms. However this involved verifying an identity to over
$10^{13}$ coefficients which is clearly beyond the limits of computation.
Andersen's proof involves nonholomorphic vector-valued modular forms and does
not rely on any computational verification.

%%HERE

\subsection*{Acknowledgments}
I would like to thank Chris Jennings-Shaffer for his careful reading of this
paper and for his suggestions and corrections.
I would also like to thank Eric Mortenson reminding me about
\cite{Hi-Mo14} and \cite{Hi-Mo16}, and thank Nick Andersen for \cite{An16}.

%%%%%%%%%%%%%%%%%%%%%%%%%%%%%%%%%%%%%%%%%%%%%%%%%%%%%%%%%%%%%%%%%%%%%%%%%?
\bibliographystyle{amsplain}

%%%--------------------------------------
%%%%%%% %%%%%%%%%%%%%%%%%%%%%%%%%%%%%%%%%%%%%%%%%%%%%%%%%%%%%%%%%%%%%%%%%%%%%
%%
\end{document}